\documentclass[reqno,10pt]{amsart}

\usepackage[a4paper,left=25mm,right=25mm,top=30mm,bottom=30mm,marginpar=25mm]{geometry} 
\usepackage{amsmath}
\usepackage{amssymb}
\usepackage{amsthm}
\usepackage{amscd}
\usepackage[ansinew]{inputenc}
\usepackage{color}
\usepackage[final]{graphicx}
\usepackage{tikz}
\usepackage{bbm}
\usepackage{comment}
\usepackage{algorithm}
\usepackage{algorithmic}
\allowdisplaybreaks

\usetikzlibrary{patterns} 
\usetikzlibrary{positioning}
\usetikzlibrary{calc}

\usetikzlibrary{intersections}

\usepackage{enumitem}

\usepackage{float}
 \usepackage[colorlinks=true, pdfstartview=FitV, linkcolor=olive,
             citecolor=olive, urlcolor=olive]{hyperref}

\usepackage{hyperref}

\renewcommand{\epsilon}{\varepsilon}

\numberwithin{equation}{section}

\newtheoremstyle{thmlemcorr}{10pt}{10pt}{\itshape}{}{\bfseries}{.}{10pt}{{\thmname{#1}\thmnumber{ #2}\thmnote{ (#3)}}}
\newtheoremstyle{thmlemcorr*}{10pt}{10pt}{\itshape}{}{\bfseries}{.}\newline{{\thmname{#1}\thmnumber{ #2}\thmnote{ (#3)}}}
\newtheoremstyle{defi}{10pt}{10pt}{\itshape}{}{\bfseries}{.}{10pt}{{\thmname{#1}\thmnumber{ #2}\thmnote{ (#3)}}}
\newtheoremstyle{remexample}{10pt}{10pt}{}{}{\bfseries}{.}{10pt}{{\thmname{#1}\thmnumber{ #2}\thmnote{ (#3)}}}
\newtheoremstyle{ass}{10pt}{10pt}{}{}{\bfseries}{.}{10pt}{{\thmname{#1}\thmnumber{ A#2}\thmnote{ (#3)}}}

\theoremstyle{thmlemcorr}
\newtheorem{theorem}{Theorem}
\numberwithin{theorem}{section}
\newtheorem{lemma}[theorem]{Lemma}
\newtheorem{corollary}[theorem]{Corollary}
\newtheorem{proposition}[theorem]{Proposition}

\theoremstyle{thmlemcorr*}
\newtheorem{theorem*}{Theorem}
\newtheorem{lemma*}[theorem]{Lemma}
\newtheorem{corollary*}[theorem]{Corollary}
\newtheorem{proposition*}[theorem]{Proposition}
\newtheorem{problem*}[theorem]{Problem}
\newtheorem{conjecture*}[theorem]{Conjecture}
\newtheorem{questionn*}[theorem]{Open Problem}

\theoremstyle{defi}
\newtheorem{definition}[theorem]{Definition}

\theoremstyle{remexample}
\newtheorem{remark}[theorem]{Remark}
\newtheorem{example}[theorem]{Example}

\theoremstyle{ass}

\newcommand{\ZZ}{\mathbb{Z}}

\newcommand{\Ecal}{\mathcal{E}}
\newcommand{\Fcal}{\mathcal{F}}
\newcommand{\Gcal}{\mathcal{G}}

\newcommand{\Ical}{\mathcal{I}}

\newcommand{\Lcal}{\mathcal{L}}
\newcommand{\Mcal}{\mathcal{M}}

\newcommand{\Xcal}{\mathcal{X}}
\newcommand{\Ycal}{\mathcal{Y}}

\newcommand{\NN}{\mathbb{N}}

\DeclareMathOperator{\inte}{int}
\DeclareMathOperator{\essinf}{ess\,inf}
\DeclareMathOperator{\esssup}{ess\,sup}

\DeclareMathOperator{\diverg}{div}
\DeclareMathOperator{\curl}{curl}
\DeclareMathOperator{\dist}{dist}

\DeclareMathOperator{\supp}{supp}

\DeclareMathOperator*{\prox}{prox}
\DeclareMathOperator{\Lip}{Lip}

\newcommand{\dd}{\;\mathrm{d}}

\newcommand{\N}{\mathbb{N}}
\newcommand{\R}{\mathbb{R}}
\newcommand{\RR}{\mathbb{R}}
\newcommand{\XX}{\mathbb{X}}

\newcommand{\loc}{\mathrm{loc}}

\newcommand{\weakly}{\rightharpoonup}
\newcommand{\weaklystar}{\overset{*}\rightharpoonup}

\newcommand{\epsi}{\epsilon}
\newcommand{\ffi}{\varphi}

\newcommand{\argmin}{\operatorname{argmin}}

\newcommand\restr[2]{{\left.\kern-\nulldelimiterspace #1\vphantom{\big|}\right|_{#2}}}

\DeclareMathOperator{\arginf}{arginf}
\DeclareMathOperator{\Div}{div}

\newcommand*{\medcup}{\mathbin{\scalebox{1.25}{\ensuremath{\cup}}}}%

\allowdisplaybreaks

\newcommand{\scrL}{\mathscr{L}} 
\newcommand{\scrP}{\mathscr{P}}

\newcommand{\scrS}{\mathscr{S}} 
 
\newcommand{\scrV}{\mathscr{V}}   
 
 \def\leq{\leqslant}
\def\geq{\geqslant}

\usepackage[scr=boondoxo]{mathalfa}
\usepackage{multicol}

\usepackage{aurical}
\usepackage[T1]{fontenc}
\usepackage{lmodern}
\usepackage[bbgreekl]{mathbbol}
\usepackage{amsfonts}

\DeclareSymbolFontAlphabet{\mathbb}{AMSb}
\DeclareSymbolFontAlphabet{\mathbbl}{bbold}

\title[Dyadic partition-based training schemes for TV/TGV denoising]{Dyadic partition-based training schemes for TV/TGV denoising}

\author[Elisa Davoli]{Elisa Davoli}
\address{TU Wien, Institute of Analysis and Scientific Computing, Wiedner Hauptstrasse 8-10, 1040 Vienna, Austria}
\email{elisa.davoli@tuwien.ac.at}

\author[Rita Ferreira]{Rita Ferreira}
\address{King Abdullah University of Science and Technology (KAUST), CEMSE Division, Thuwal 23955-6900, Saudi Arabia}
\email{rita.ferreira@kaust.edu.sa}

\author[Irene Fonseca]{Irene Fonseca}
\address{Carnegie Mellon University, 5000 Forbes Avenue, Pittsburgh, PA 15213, USA}
\email{fonseca@andrew.cmu.edu}

\author[Jos\'e A. Iglesias]{Jos\'e A. Iglesias}
\address{Department of Applied Mathematics, University of Twente, P.O. Box 217, 7500 AE Enschede, The Netherlands}
\email{jose.iglesias@utwente.nl}


\usepackage[normalem]{ulem}

\begin{document}

\begin{abstract}  
Due to their ability to handle discontinuous images while having a well-understood behavior, regularizations with total variation (TV) and total generalized variation (TGV) are some of the best-known methods in image denoising. However, like other variational models including a fidelity term, they crucially depend on the choice of their tuning parameters. A remedy is to choose these automatically through multilevel approaches, for example by optimizing performance on noisy/clean image pairs. In this work, we consider such methods with space-dependent parameters which are piecewise constant on dyadic grids, with the grid itself being part of the minimization. We prove existence of minimizers for fixed discontinuous parameters under mild assumptions on the data, which lead to existence of finite optimal partitions. We further establish that these assumptions are equivalent to the commonly used box constraints on the parameters. On the numerical side, we consider a simple subdivision scheme for optimal partitions built on top of any other bilevel optimization method for scalar parameters, and demonstrate its improved performance on some representative test images when compared with constant optimized parameters.
\end{abstract}
  
\maketitle

\tableofcontents 

\section{Introduction}\label{sect:intro} 

A fundamental problem in image processing is the restoration of a given ``noisy'' image. 
Images are  often deteriorated due to several factors occurring, for instance, in the
process of transmission or acquisition, such as blur caused by motion or a deficient lens adjustment. 

A well-established and successful approach for image restoration is
hinged on variational PDE methods, where minimizers of certain energy functionals provide the sought ``clean'' and ``sharp'' images. In the particular case where the degradation consists
of additive noise, 
these energy functionals  usually take the form
\begin{equation}
\label{eq:efir}
\begin{aligned}
E(u):= \Vert u - u_\eta\Vert^p_X + R_\alpha(u)\quad \text{for \(u\in\widetilde X\)},
\end{aligned}
\end{equation}
where \(u_\eta\) represents the given noisy image and \(\widetilde X\) is the class of possible reconstructions of \(u_\eta\). The first  term in \eqref{eq:efir},  \(\Vert u - u_\eta\Vert^p_X\), is the fidelity or data fitting term that, in a minimization process, controls the distance between \(u\) and \(u_\eta\) in some space \(X\). The second term, \(R_\alpha(u)\), is the so-called filter term, and is responsible for the regularization of the images. The parameter \(\alpha\) is often called a tuning or regularization parameter, and accounts for a balance between the fidelity and filter terms.

A milestone approach in imaging denoising is due to Rudin, Osher, and Fatemi \cite{RuOsFa92}, who proposed (in a discrete setting, later extended to a function space framework in \cite{AcVo94, ChLi97}) an energy functional of the type \eqref{eq:efir} with \(X:=L^2(Q)\), \(p:=2\), \(\widetilde X := BV(Q)\), and \(R_\alpha(u) := \alpha TV(u,Q)\) with \(\alpha>0\), where \(Q\subset\RR^2\) is the image domain  and \( TV(u,Q)\) is the total variation in \(Q\) of a function of bounded variation \(u\in BV(Q)\). Precisely, given an observed noisy version \(u_\eta \in L^2(Q)\) of a true image, the ROF or TV model consists in finding a reconstruction of the original clean image as the solution of the minimization problem
\begin{equation}
\label{eq:ROF-intro}
\begin{aligned}
\min_{u\in BV(Q)} \Big\{\Vert u - u_\eta\Vert^2_{L^2(Q)} +\alpha TV(u,Q)\Big\}.
\end{aligned}
\end{equation}
A striking feature of this model is that it removes noise while preserving images' edges. This model has been extended in several ways, including higher-order and vectorial settings to address color images, and gave rise to numerous related filter terms seeking to overcome some of its drawbacks, such as blurring and the staircasing effect (see, for instance, \cite{AuKo06,ChKaSh01, BrHo20} for an overview).     

In a nutshell, the TV model yields functions $u$ that best  fit the data, measured in terms of the \(L^2\) norm, and whose gradient (total variation) is low so that noise is removed. The choice of the parameter \(\alpha\) plays a decisive role in the success of this and similar variational approaches, as it balances the fitting and regularization features of such models. In fact, higher values of \(\alpha \) in \eqref{eq:ROF-intro} lead to an oversmoothed reconstruction of \(u_\eta\) because the total variation has to be ``small'' to compensate for high values of \(\alpha\);  conversely, lower values of \(\alpha \) in \eqref{eq:ROF-intro} inhibit
noise removal and, in particular, the reconstructed image  provided by \eqref{eq:ROF-intro} converges to \( u_\eta\) as \(\alpha\to0\) (see \cite{DeScVa16}). 

In principle, the ``optimal'' parameter \(\alpha\) needs to be chosen individually for each noisy image, which makes such models require additional information to be complete. To address this issue, a partial automatic selection of an  ``optimal'' parameter \(\alpha\) was proposed in \cite{DeScVa16,DeScVa17} (see also \cite{ChPoRaBi13,ChRaPo14,Do12,TaLiAdFr07}) in the flavour of Machine Learning optimization schemes. This automatic selection is based on a bilevel optimization scheme searching for the optimal \(\alpha\) that minimizes the distance, in some space, between the reconstruction of a  noisy image and the original clean image. In this setting, both the noisy image, \(u_\eta\), and the original clean image, \(u_c\),  are known a priori and called the training data. The rationale is to use the same parameter \(\alpha\) to reconstruct noisy images that are \textit{qualitatively similar} to that of the training scheme and corrupted by a similar type and amount of noise, and are thus expected to require a similar balance between fitting and regularization effects. 

In the context of the TV model in \eqref{eq:ROF-intro}, one such bilevel optimization scheme reads as follows.
Here, and in the sequel, \(\RR^+\) stands for the set of positive real numbers, \((0,\infty)\). Moreover, for minimization problems over $\R^+$ or $\R^+ \times \R^+$, we write $\arginf$ instead of $\argmin$ to include the case where the infimum would be attained at the boundary of these open sets.
\begin{flalign}
& (\scrL\!\scrS)_{TV}  \textit{ \bf  TV learning scheme}&
\label{lsTV} \end{flalign}

\textbf{Level 1.}
Find
\begin{equation*}
\begin{aligned}
\bar \alpha={\rm arg inf}\left\{\int_Q |u_c-u_\alpha|^2\dd x:\,\alpha\in \RR^+\right\};
\end{aligned}
\end{equation*}

\textbf{Level 2.} Given \(\alpha\in \RR^+\), find
\begin{equation*}
\begin{aligned}
u_\alpha= \argmin\left\{\int_{Q}|u_\eta-u|^2\dd x+\alpha TV(u,Q):\,u\in BV(Q)\right\}.
\end{aligned}
\end{equation*}

This approach yields a unified way of identifying the best fitting parameters for every class of training data lying in the same $L^2$-neighborhood. However,  the learning scheme \eqref{lsTV} does not address  a major drawback of the TV and similar models using scalar regularization parameters.  In fact, it does not take into account  possible inhomogeneous noise  (occurring, e.g., in parallel acquisition in magnetic resonance imaging \cite{DiEtAl07}) and other local features in a  given deteriorated image that would benefit from an adapted treatment.

 A solution to this issue consists in resorting to adaptive methods and varying fitting parameters instead. The mathematical literature in this direction is vast, from which we single out the following contributions: \cite{KuPo12, Po12} for results in the finite-dimensional case and for optimal image filters, \cite{DeSc13} for bilevel learning in function spaces and development of numerical optimization, \cite{DaLi18, DaFoLi23, LiLu18, Li19, LiLu19} for a study of optimal regularizers, \cite{LiSc19} for a bilevel analysis of novel classes of semi-norms,  \cite{PaPaRaVi22} for an approach via Young measures, and \cite{CaEtAl17, HiPa19, DeVi23, CrFe22} and the references therein for an overview. 

A relevant question in image reconstruction (as pointed out in \cite{Li17}, among others) is the possibility of adapting the fitting parameters to the specific features of a given class of noisy images by performing, e.g., a stronger regularization in areas which have been highly deteriorated and by tuning down the filtering actions in portions that, instead, have been left unaffected.

Here, starting from the ideas in \cite{Li17}, we propose space-dependent learning schemes that locally search for the optimal level of refinement and the optimal regularization parameters. The optimal level of refinement translates into finding an optimal partition of the noisy image's domain that takes into account its local features. Precisely, as before, \(Q=(0,1)^2\) represents the  images' domain. We say that \(\mathscr{L}\) is an admissible partition of \(Q\) if it consists of dyadic squares, each of which we often denote by \(L\) (see Section~\ref{sect:glo} for a more detailed description of these partitions). Note that an admissible partition might be more or less refined in different
parts of the domain. We denote by \(\mathscr{P}\)  the class of all such admissible partitions \(\scrL\) of \(Q\). Finally, let   $(u_\eta,u_c)\in BV(Q)\times BV(Q)$ be a training pair of noisy and clean images. The first space-dependent learning scheme that we propose to restore \(u_\eta\), based on the a priori knowledge of \(u_c\), is as follows.

\begin{flalign}
& (\scrL\!\scrS)_{{TV\!}_\omega}  \textit{ \bf Weighted-TV learning scheme}& \label{lsTVomega} \end{flalign}

\begin{itemize}[leftmargin=20mm]

\item[\textbf{Level 3.}] \text{(optimal local training
parameter)} Fix  \(\mathscr{L}\in \mathscr{P}\); for each \(L\in \mathscr{L}\),
find
\begin{equation}
\label{eq:alpha-L}
\alpha_L:=\inf\left\{\arginf\left\{\int_L |u_c-u_{\alpha,L}|^2\,\dd x\!:\,\alpha\in
\RR^+\right\}\right\},
\end{equation}
where, for \(\alpha\in
\RR^+\),
\begin{equation}
\label{eq:ROF}
\begin{aligned}
u_{\alpha,L}:= \argmin\left\{\int_{L}|u_\eta-u|^2\dd x+\alpha TV(u,L)\!:\,u\in
BV(L)\right\}.
\end{aligned}
\end{equation}

\item[\textbf{Level 2.}] \text{(space-dependent image denoising)} For each  \(\mathscr{L}\in \mathscr{P}\),
find%
\begin{equation}\label{eq:minprS2a}
\begin{aligned}
u_{\mathscr{L}}:=\argmin\bigg\{\int_Q|u_\eta-u|^2\dd x +TV_{\omega_\mathscr{L}}(u,Q)\!:\,u\in
BV_{\omega_{\mathscr{L}}}(Q)\bigg\},
\end{aligned}
\end{equation}
where we consider the  piecewise constant weight
 \(\omega_{\mathscr{L}}\) defined by %
\begin{equation}
\label{eq:weight}
\begin{aligned}
\omega_{\mathscr{L}}(x):=\sum_{L\in \mathscr{L} }\alpha_L
\chi_L(x)  \quad \text{with \(\alpha_L\) given by Level~3},
\end{aligned}
\end{equation}
and    \(BV_{\omega_{\mathscr{L}}}\) is the space of  \(\omega_{\mathscr{L}}\)-weighted
$BV$-functions (see Section~\ref{sect:L2}).
\textit{}\smallskip

\item[\textbf{Level 1.}] \text{(optimal partition and 
image restoration)}
Find
\begin{equation*}
\begin{aligned}
u^*\in \argmin\left\{\int_Q|u_c-u_{\mathscr{L}}|^2\dd x:\,\mathscr{L}\in
\mathscr{P}\right\} \quad \text{with \(u_{\mathscr{L}}\) given by Level~2}.
\end{aligned}
\end{equation*}
\end{itemize}

\begin{remark}\label{rmk:onTVw} \textbf{(i)}~We observe that by taking the infimum in \eqref{eq:alpha-L}, the corresponding parameter \(\alpha_L\)   is always well defined. On the other hand, if \(TV(u_\eta,L) >\ TV(u_c,L)\) and $\displaystyle \Vert u_\eta - u_c\Vert^2_{L^2(L)}
<\Vert[  u_\eta]_L - u_c\Vert^2_{L^2(L)}$ as in \cite{DeScVa16}, with \([  u_\eta]_L:=\frac{1}{|L|}\int_L u_\eta\dd x \),  we prove in Theorem~\ref{thm:onalpha} that there
exists   \(\tilde \alpha_L\in(0,\infty)\) satisfying
\begin{equation*}
\tilde\alpha_L\in\argmin\left\{\int_L |u_c-u_{\alpha,L}|^2\,\dd x\!:\,\alpha\in
\RR^+\right\}
\end{equation*}
(see \cite{DeScVa16} for similar statements), in which case 
 the infimum on such \(\tilde\alpha_L\) as in  \eqref{eq:alpha-L} may be regarded as a choice criterium on the optimal parameter.    

\textbf{(ii)}~We  refer to Section~\ref{sect:L2} for the definition and discussion
of the space  \(BV_{\omega_{\mathscr{L}}}\) of  \(\omega_{\mathscr{L}}\)-weighted
$BV$-functions, as introduced     in \cite{Ba01}. In particular, using the
results   in \cite{Ba01}
 (and also  \cite{Ca08,Ca10}),   we prove under appropriate conditions that
\({u}_{\mathscr{L}} \in BV(Q)\) and
\begin{equation}\label{eq:VarBVw}
\begin{aligned}
TV_{\omega_\mathscr{L}}({u}_{\mathscr{L}},Q) = \int_Q 
\omega_{\mathscr{L}}^{sc^-}\!(x) \dd |D{u}_{\mathscr{L}}|(x),
\end{aligned}
\end{equation}
where \(\omega_\mathscr{L}^{sc^-}\) denotes the lower-semicontinuous envelope
of $\omega_\mathscr{L}$. We further mention the works in \cite{AtJeNoOr17,HiHoPa18} addressing the study of inverse  problems that include a weighted-\(TV\) model of the form of the one in \eqref{eq:minprS2a}.

\end{remark}

The existence of solutions to the learning scheme \((\scrL\!\scrS)_{{TV\!}_\omega}\)
in \eqref{lsTVomega} is intimately related to the existence of a stopping criterion for
the refinement of the admissible partitions or, in other words, a lower bound
on the size of the dyadic squares \(L\in \scrL\), with \(\scrL\in\scrP\).
This notion is made precise in the following definition.

\begin{definition}[stopping criterion for
the refinement of the admissible partitions]\label{def:stop}
We say that a condition \((\scrS)\) on \(\scrP\)is a
stopping criterion for
the refinement of the admissible partitions if there exist \(\kappa\in \NN\) and \(\mathscr{L}_1, ...
, \mathscr{L}_\kappa \in \mathscr{P}\) such that
\begin{equation*}
\begin{aligned}
\argmin\left\{\int_Q|u_c-u_{\mathscr{L}}|^2\dd x:\,\mathscr{L}\in
\mathscr{P}\right\} = \argmin\left\{\int_Q|u_c-u_{\mathscr{L}_i}|^2\dd x:\,i\in
\{1,...,\kappa\}\right\}
\end{aligned}
\end{equation*}
provided that  \((\scrS)\) holds, where \(u_{\mathscr{L}}\)  and \(u_{\mathscr{L}_i}\) are given by \eqref{eq:minprS2a}. In
this case, we write \( \bar{ \mathscr{P}}:=\medcup_{i=1}^\kappa
\{\scrL_i\}
\).
\end{definition}

We refer to Section~\ref{sect:box} for examples of stopping criteria as in Definition~\ref{def:stop}, from which we highlight the box-constraint that we discuss next.

\begin{remark}[box constraint as a stopping criterion]\label{rmk:box}
To prove the existence of a solution to the learning scheme \((\scrL\!\scrS)_{{TV\!}_\omega}\) in \eqref{lsTVomega}, we  adopt the usual box-constraint approach
in which we replace \(\alpha\in\RR^+\) by 
\begin{equation}\label{eq:boxc}
\begin{aligned}
\alpha\in \bigg[c_0,\frac{1}{c_0}\bigg]\quad \text{ for some } c_0\in (0,1).
\end{aligned}
\end{equation}
In this case, the analog of \eqref{eq:alpha-L} becomes
\begin{equation}
\label{eq:box-min}
\bar \alpha_L=\inf\left\{\arginf\left\{\int_L |u_c-u_{\alpha,L}|^2\dd x:\,\alpha\in
\big[c_0,\tfrac{1}{c_0}\big]\right\}\right\}.
\end{equation}
Under some assumptions on the training
data, we prove in Subsection~\ref{sect:box} (see Theorem~\ref{thm:equiv} below) that this box constraint is equivalent to  the existence of a stopping criterion for the refinement of the admissible partitions as in Definition~\ref{def:stop}.
\end{remark}

\begin{theorem}[Equivalence between box constraint and
stopping criterion]\label{thm:equiv}
Consider the learning scheme \((\scrL\!\scrS)_{{TV\!}_\omega}\)
in \eqref{lsTVomega}. The two following statements hold:\begin{itemize}
\item[(a)] If we replace    \eqref{eq:alpha-L} 
 by \eqref{eq:box-min}, then there exists a
stopping criterion \((\scrS)\)  for
the refinement of the admissible partitions as in Definition~\ref{def:stop}. 
\item[(b)] Assume that
there exists a
stopping criterion \((\scrS)\)  for
the refinement of the admissible partitions as in Definition~\ref{def:stop}
such that the training data
  satisfies for all \(L\in\medcup_{\scrL\in\bar\scrP} \scrL\), with
  \(\bar\scrP\) as in Definition~\ref{def:stop}, the   conditions
\begin{itemize}[leftmargin=12mm]
\item[(i)] $TV(u_c,L) < TV(u_\eta,L)$;

\item[(ii)] $\displaystyle \Vert u_\eta - u_c\Vert^2_{L^2(L)}
<\Vert[u_\eta]_L - u_c\Vert^2_{L^2(L)} $, where  \([  u_\eta]_L=\frac{1}{|L|}\int_L
u_\eta\dd x. \)
\end{itemize}
Then, there exists \(c_0\in\RR^+\) such that the optimal
solution \(u^*\) provided by  \((\scrL\!\scrS)_{{TV\!}_\omega}\)
with  \(\scrP\) replaced by  \(\bar \scrP\)
coincides with the optimal solution
\(u^*\) provided by  \((\scrL\!\scrS)_{{TV\!}_\omega}\)
with  \eqref{eq:alpha-L} 
replaced by \eqref{eq:box-min}.
   \end{itemize}

\end{theorem}

Next, we state our main theorem regarding existence of solutions for the learning scheme \((\scrL\!\scrS)_{{TV\!}_\omega}\)
in \eqref{lsTVomega}. We state this result under the box-constraint condition. However, in view of Theorem~\ref{thm:equiv}, this result holds true under any stopping criterion for
the refinement of the admissible partitions, and in particular if the the training data satisfies the conditions (i) and (ii) above.

\begin{theorem}[Existence of solutions to \((\scrL\!\scrS)_{{TV\!}_\omega}\)]\label{thm:TVomega}
There exists an  optimal solution
\(u^*\) to the learning scheme \((\scrL\!\scrS)_{{TV\!}_\omega}\)
in \eqref{lsTVomega}, whenever \eqref{eq:alpha-L} is replaced by \eqref{eq:box-min} for some fixed $c_0 \in (0,1).$ 

\end{theorem}

The proofs of Theorems~\ref{thm:equiv} and \ref{thm:TVomega} are presented in Section~\ref{sect:wTV}, where we also explore alternative
stopping criteria.

As shown in \cite[Theorem~2.4.17]{Ja12}, given a positive, bounded, and Lipschitz
continuous function \(\omega:Q\to(0,\infty)\) with \(\nabla \omega \in BV(Q;\RR^2)\), the solution of \eqref{eq:minprS2a} with \(\omega_{\mathscr{L}}\) replaced by \(\omega\) may exhibit jumps inherited from the weight
\(\omega\)  
that are not present in the data \(u_\eta\), see Figure \ref{fig:parameterdiscont} for a numerical example. Because \(\omega_{\mathscr{L}}\) in Level~2 is constructed using the local optimal parameters given by Level~3, we heuristically expect that, in most applications, these extra jumps do not induce clearly visible artifacts. However, this possible issue has led us to consider two alternative adaptive space-dependent learning schemes. 

First, we consider a learning scheme based on    \((\scrL\!\scrS)_{{TV\!}_\omega}\)
in \eqref{lsTVomega} with  \(\omega_{\mathscr{L}}\) replaced by a smooth regularization  \((\omega_\epsi)_{\mathscr{L}}\) (see the regularized weighted TV learning scheme 
\((\scrL\!\scrS)_{{TV\!}_{\omega_\epsi}}\) in \eqref{lsTVomegae} below). 
Second, using the fact that the minimizer in \eqref{eq:ROF} coincides with  
\begin{equation*}
\begin{aligned}
 \argmin\left\{\frac1\alpha\int_{L}|u_\eta-u|^2\dd x+ TV(u,L)\!:\,u\in
BV(L)\right\},
\end{aligned}
\end{equation*}
we consider the weighted-fidelity learning scheme 
\((\scrL\!\scrS)_{{TV-Fid}_\omega}\)  in \eqref{lsFido} below, where the weight appears in the fidelity term. Let us point out that a detailed analysis of the differences arising between weighted-fidelity and weighted-regularization parameter for TV has been carried out in the one-dimensional case in \cite{HiPaRa17}.

We begin by describing the regularized scenario.

\begin{flalign}
& (\scrL\!\scrS)_{{TV\!}_{\omega_\epsi}}  \textit{ \bf Regularized
weighted-TV learning scheme}&
\label{lsTVomegae} \end{flalign}

\begin{itemize}[leftmargin=20mm]

\item[\textbf{Level 3.}] \text{(optimal local training
parameter)} Fix  \(\mathscr{L}\in \mathscr{P}\); for each \(L\in
\mathscr{L}\),
find
\begin{equation}
\label{eq:alpha-LL}
\alpha_L=\inf\left\{\arginf\left\{\int_L |u_c-u_{\alpha,L}|^2\,\dd
x\!:\,\alpha\in
\RR^+\right\}\right\},
\end{equation}
where, for \(\alpha\in
\RR^+\),
\begin{equation*}
\begin{aligned}
u_{\alpha,L}:= \argmin\left\{\int_{L}|u_\eta-u|^2\dd x+\alpha TV(u,L)\!:\,u\in
BV(L)\right\}.
\end{aligned}
\end{equation*}

\item[\textbf{Level 2.}] \text{(space-dependent image denoising)}
For each
 \(\mathscr{L}\in \mathscr{P}\) and for \(\epsi>0\)
 fixed,
find%
\begin{equation*}
\begin{aligned}
u^\epsi_{\mathscr{L}}:=\argmin\bigg\{\int_Q|u_\eta-u|^2\dd x
+TV_{\omega^\epsi_{\mathscr{L}}}(u,Q)\!:\,u\in
BV_{\omega^\epsi_{\mathscr{L}}}(Q)\bigg\},
\end{aligned}
\end{equation*}
where we consider a  regularized weight
 \(\omega^\epsi_{\mathscr{L}}:Q\to[0,\infty)\) of \(\omega_{\mathscr{L}}\)
in \eqref{eq:weight} such that
\begin{equation}
\label{eq:regweight}
\begin{aligned}
\omega^\epsi_{\mathscr{L}} \in C^1(Q) \enspace \text{ and } \enspace
\omega^\epsi_{\mathscr{L}} \nearrow \omega_{\mathscr{L}} \text{
as \(\epsi\to0^+\) and a.e.~in \(Q\)}.
\end{aligned}
\end{equation}

\item[\textbf{Level 1.}] \text{(optimal partition and 
image restoration)}
Find
\begin{equation*}
\begin{aligned}
u^*_\epsi\in \argmin\left\{\int_Q|u_c-u^\epsi_{\mathscr{L}}|^2\dd
x:\,\mathscr{L}\in
\mathscr{P}\right\} \quad \text{with \(u^\epsi_{\mathscr{L}}\)
given by Level~2}.
\end{aligned}
\end{equation*}
 \end{itemize}

For each \(\epsi>0\) fixed, similar results to those regarding
the learning scheme \((\scrL\!\scrS)_{{TV\!}_\omega}\)
in \eqref{lsTVomega}   hold for the learning scheme 
\((\scrL\!\scrS)_{{TV\!}_{\omega_\epsi}}\) in \eqref{lsTVomegae}.
A natural question is whether a sequence of optimal solutions
of the latter, \(\{u^*_\epsi\}_\epsi\), converge in some sense
to an optimal solution of the former, \(u^*\), as \(\epsi\to0^+\).
This turns out to be an interesting mathematical question (see Remark~\ref{Q:opnebp}), which
we partially address in the following proposition.

\begin{proposition}[On the energies in \((\scrL\!\scrS)_{{TV\!}_{\omega_\epsi}}\)
as \(\epsi\to0^+\)]\label{thm:Gconvergence}
 Under the setup of the learning schemes   \((\scrL\!\scrS)_{{TV\!}_\omega}\)
and
  \((\scrL\!\scrS)_{{TV\!}_{\omega_\epsi}}\) above,  fix 
 \(\mathscr{L}\in \mathscr{P}\) and let \(E_\scrL:L^1(Q)\to[0,\infty]\)
and \(E_\scrL^\epsi:L^1(Q)\to[0,\infty]\) be  the functionals
defined for \(u\in L^1(Q)\) by  
\begin{equation*}
\begin{aligned}
&E_\scrL[u]:=\begin{cases}
\displaystyle\int_Q|u_\eta-u|^2\dd x +TV_{\omega_\mathscr{L}}(u,Q)
&\text{if } u\in
BV_{\omega_{\mathscr{L}}}(Q),\\
+\infty &\text{otherwise},
\end{cases}\\
&E_\scrL^\epsi[ u]:=\begin{cases}
\displaystyle\int_Q|u_\eta-u|^2\dd x +TV_{\omega^\epsi_{\mathscr{L}}}(u,Q)
&\text{if
} u\in
BV_{\omega^\epsi_{\mathscr{L}}}(Q),\\
+\infty &\text{otherwise}.
\end{cases}
\end{aligned}
\end{equation*}
If \eqref{eq:regweight} holds, then
\begin{equation}\label{eq:Glimsup}
\begin{aligned}
&\Gamma(L^1(Q))-\limsup_{\epsi\to0^+} E_\scrL^\epsi \leq E_\scrL.
\end{aligned}
\end{equation}

\end{proposition}

Inequality \eqref{eq:Glimsup} states, roughly speaking, that the asymptotic behavior of the functionals $E_\scrL^\epsi$ is bounded from above by $E_\scrL$, for it can be equivalently expressed as
$$\inf\left\{\limsup_{\epsi\to 0^+} E_\scrL^\epsi[u_{\epsi}]\colon\, u_{\epsi}\to u\quad\text{strongly in }L^1(Q)\right\}\leq E_\scrL[u]$$
for every $u\in L^1(Q)$. The proof of this proposition and an analytical discussion of
the learning scheme 
\((\scrL\!\scrS)_{{TV\!}_{\omega_\epsi}}\) in \eqref{lsTVomegae}
can be found in Section~\ref{sect:regTV}, while the corresponding
numerical scheme is detailed in Section~\ref{sect:numerics}.

Next, we study the weighted-fidelity learning scheme 
\((\scrL\!\scrS)_{{TV-Fid}_\omega}\)   motivated above. 
\begin{flalign}
& (\scrL\!\scrS)_{{TV-Fid}_\omega}  \textit{ \bf Weighted-fidelity learning scheme}&
\label{lsFido} \end{flalign}

\begin{itemize}[leftmargin=20mm]

\item[\textbf{Level 3.}] \text{(optimal local training
parameter)} Fix  \(\mathscr{L}\in \mathscr{P}\); for each \(L\in \mathscr{L}\),
find
\begin{equation}
\label{eq:alpha-Lf}
\alpha_L=\inf\left\{\arginf\left\{\int_L |u_c-u_{\alpha,L}|^2\,\dd x\!:\,\alpha\in
\RR^+\right\}\right\},
\end{equation}
where, for \(\alpha\in
\RR^+\),
\begin{equation}
\label{eq:ROFf}
\begin{aligned}
u_{\alpha,L}:= \argmin\left\{\int_{L}\frac1\alpha|u_\eta-u|^2\dd x+ TV(u,L)\!:\,u\in
BV(L)\right\}.
\end{aligned}
\end{equation}

\item[\textbf{Level 2.}] \text{(space-dependent image denoising)} For each
 \(\mathscr{L}\in \mathscr{P}\),
find%
\begin{equation*}
\begin{aligned}
u_{\mathscr{L}}:=\argmin\bigg\{\int_Q\frac{1}{\omega_\mathscr{L}}|u_\eta-u|^2\dd x +TV(u,Q)\!:\,u\in
BV_{\omega_{\mathscr{L}}}(Q)\bigg\},
\end{aligned}
\end{equation*}
where, similarly to \eqref{eq:weight},  
 \(\omega_{\mathscr{L}}\) is defined by %
\begin{equation*}
\begin{aligned}
\omega_{\mathscr{L}}(x):=\sum_{L\in \mathscr{L} }\alpha_L
\chi_L(x)  \quad \text{with \(\alpha_L\) given by Level~3}.
\end{aligned}
\end{equation*}

\item[\textbf{Level 1.}] \text{(optimal partition and 
image restoration)}
Find
\begin{equation*}
\begin{aligned}
u^*\in \argmin\left\{\int_Q|u_c-u_{\mathscr{L}}|^2\dd x\colon\,\mathscr{L}\in
\mathscr{P}\right\} \quad \text{with \(u_{\mathscr{L}}\) given by Level~2}.
\end{aligned}
\end{equation*}

\end{itemize}
Once more, similar results to those regarding the learning
scheme \((\scrL\!\scrS)_{{TV\!}_\omega}\)
in \eqref{lsTVomega}   hold for the learning scheme 
\((\scrL\!\scrS)_{{TV-Fid}_\omega}\) in \eqref{lsFido}.  In particular, the box constraint here is essential to guarantee that \textbf{Level~2} of the scheme is well posed. This analysis is undertaken in Section~\ref{sect:regTV}, while the corresponding numerical study is addressed in Section~\ref{sect:numerics}.

 The last theoretical result of this paper concerns replacing the $TV$ term in our space-dependent bilevel learning schemes with a higher-order regularizer. A well-known drawback of the ROF model is the possible occurrence of staircasing effects whenever two neighboring areas of an image are both smoothed out and an abrupt spurious discontinuity is produced in the denoising process. To counteract this effect a canonical solution (among others like the use of Huber-type smoother approximations of the total variation as in \cite{Bu15}) consists in resorting to higher-order derivatives in the regularizer (see, e.g., \cite{Ch00, Da09, Pa13, BrHo20}). We consider here the total generalized variation ($TGV$) model introduced in \cite{Br10}, which is considered to be one of the most effective image-reconstruction models among those involving mixed first- and higher-order terms, cf. \cite{Pa13,BrKuVa13,PoSc15,IgWa22} for some theoretical results about its solutions. 
 
 For a function $u\in BV(Q)$ and $\alpha=(\alpha_0,\alpha_1)\in \R^+\times \R^+$, the second-order $TGV$ functional is given by
\begin{equation}\label{eq:tgv}
TGV_{\alpha_0,\alpha_1}(u):={\min}\big\{\alpha_0|D u-v|(Q)+\alpha_1|\Ecal v|(Q): v\in BD(Q)\big\},
\end{equation}
 where, as before, $Du$ denotes the distributional gradient of $u$, $|\mu|(Q)$ is the total variation on \(Q\) of a Radon   measure \(\mu\), $\Ecal$ is the symmetric part of the distributional gradient, and $BD$ indicates the space of vector-valued functions with bounded deformation, cf. \cite{Te83}. In this setting, our learning scheme reads as follows.
\begin{flalign}
& (\scrL\!\scrS)_{{TGV\!}_\omega}  \textit{ \bf Weighted-TGV learning scheme}&
\label{lsTGVomega} \end{flalign}

\begin{itemize}[leftmargin=20mm]

\item[\textbf{Level 3.}] \text{(optimal local regularization parameter)} 
Fix \(\mathscr{L}\in \mathscr{P}\); for each \(L\in \mathscr{L}\),
find
\begin{equation}
\label{eq:alpha-Lg-TGV}
\begin{gathered}
\alpha_L=\big((\alpha_L)_0,(\alpha_L)_1\big):=\inf\left\{\arginf\left\{\int_L |u_c-u_{\alpha,L}|^2\,\dd x\!:\,\alpha=(\alpha_0,\alpha_1)\in
\RR^+\times \RR^+\right\}\right\},
\end{gathered}
\end{equation}
where, for \(\alpha=(\alpha_0,\alpha_1)\in
\RR^+\times\RR^+\),
\begin{equation}
\label{eq:TGV}
\begin{aligned}
u_{\alpha,L}:= \argmin\left\{ \int_{L}|u_\eta-u|^2\dd x+ TGV_{\alpha_0,\alpha_1}(u,L)\!:\,u\in
BV(L)\right\},
\end{aligned}
\end{equation}
and where the infimum in \eqref{eq:alpha-Lg-TGV} is meant with respect to the lexicographic order in $\R^2$.

\item[\textbf{Level 2.}] \text{(space-dependent TGV image denoising)} For each
 \(\mathscr{L}\in \mathscr{P}\),
find%
\begin{equation}\label{eq:minprS2g}
\begin{aligned}
u_{\mathscr{L}}:=\argmin\bigg\{\int_Q  |u_\eta-u|^2\dd x +TGV_{\omega_{\mathscr{L}}^0,\omega_{\mathscr{L}}^1}(u,Q)\!:\,u\in
BV_{\omega_{\mathscr{L}}^0}(Q)\bigg\},
\end{aligned}
\end{equation}
where, for $i\in\{0,1\}$, the weight
 \(\omega_{\mathscr{L}}^i\) is defined by %
\begin{equation*}
\begin{aligned}
\omega_{\mathscr{L}}^i(x):=\sum_{L\in \mathscr{L} }(\alpha_L)_i\,
\chi_L(x)  \quad \text{with \(\alpha_L\) given by Level~3}.
\end{aligned}
\end{equation*}
In the expression above, 
\begin{align}
\label{eq:TGVomega-def}
TGV_{\omega_{\mathscr{L}}^0,\omega_{\mathscr{L}}^1}(u,Q):=\inf_{v\in BD_{\omega_{\mathscr{L}}^1}(Q)}\left\{\scrV_{\omega_{\mathscr{L}}^0}(Du-v,Q)+{\scrV_{\omega_{\mathscr{L}}^1}}(\Ecal v,Q)\right\},
\end{align}
where the quantities $\scrV_{\omega_{\mathscr{L}}^0}$ and $\scrV_{\omega_{\mathscr{L}}^1}$ are weighted counterparts to the classical total variation of Radon measures. We refer to Sections~\ref{sect:glo} and \ref{sect:wTGV} for the precise definition and properties of these quantities. In particular,
we will prove that
\begin{equation}\label{eq:VarBDw1}
\scrV_{\omega_{\mathscr{L}}^0}(Du-v,Q)=\int_Q (\omega_{\mathscr{L}}^0)^{\textrm{sc}^-}\mathrm{d}|Du-v|,
\end{equation}
and 
\begin{equation}\label{eq:VarBDw2}
\scrV_{\omega_{\mathscr{L}}^1}(\Ecal v,Q)=\int_Q (\omega_{\mathscr{L}}^1)^{\textrm{sc}^-}\mathrm{d}|\Ecal v|,
\end{equation}
where
\(BV_{\omega_{\mathscr{L}}^0}\) is the space of  \(\omega_{\mathscr{L}}^0\)-weighted
$BV$-functions (see Subsection~\ref{sect:L2}) and $BD_{\omega_{\mathscr{L}}^1}$ is the space of \(\omega_{\mathscr{L}}^1\)-weighted $BD$-functions (see Section \ref{sect:wTGV}).

\item[\textbf{Level 1.}] \text{(optimal partition and 
image restoration)}
Find
\begin{equation*}
\begin{aligned}
u^*\in \argmin\left\{\int_Q|u_c-u_{\mathscr{L}}|^2\dd x:\,\mathscr{L}\in
\mathscr{P}\right\} \quad \text{with \(u_{\mathscr{L}}\) given by Level~2}.
\end{aligned}
\end{equation*}
\end{itemize}

Analogously to $(\scrL\!\scrS)_{{TV-Fid}_\omega}$, we can also consider a weighted-fidelity TGV scheme, which we  use in our numerical results and describe next.
\begin{flalign}
& (\scrL\!\scrS)_{{TGV-Fid}_\omega}  \textit{ \bf TGV weighted-fidelity learning scheme}&
\label{lsTGVfidomega} \end{flalign}
With $\alpha_0,\alpha_1 \in \RR^+$ fixed throughout:
\begin{itemize}[leftmargin=20mm]

\item[\textbf{Level 3.}] \text{(optimal local training
parameter)} Fix  \(\mathscr{L}\in \mathscr{P}\); for each \(L\in \mathscr{L}\), find
\begin{equation}
\label{eq:lambdatgv-Lg}
\lambda_L=\inf\left\{\arginf\left\{\int_L |u_c-u_{\lambda,L}|^2\,\dd x\!:\,\lambda\in
\RR^+\right\}\right\},
\end{equation}
where, for \(\lambda\in
\RR^+\),
\begin{equation*}
\begin{aligned}
u_{\lambda,L}:= \argmin\left\{\lambda \int_{L}|u_\eta-u|^2\dd x+ TGV_{\alpha_0, \alpha_1}(u,L)\!:\,u\in
BV(L)\right\}.
\end{aligned}
\end{equation*}

\item[\textbf{Level 2.}] \text{(space-dependent image denoising)} For each
 \(\mathscr{L}\in \mathscr{P}\),
find%
\begin{equation*}
\begin{aligned}
u_{\mathscr{L}}:=\argmin\bigg\{\int_Q \omega_\mathscr{L} |u_\eta-u|^2\dd x +TGV_{\alpha_0, \alpha_1}(u,Q)\!:\,u\in
BV_{\omega_{\mathscr{L}}}(Q)\bigg\},
\end{aligned}
\end{equation*}
where 
 \(\omega_{\mathscr{L}}\) is defined by %
\begin{equation*}
\begin{aligned}
\omega_{\mathscr{L}}(x):=\sum_{L\in \mathscr{L} }\lambda_L
\chi_L(x)  \quad \text{with \(\lambda_L\) given by Level~3}.
\end{aligned}
\end{equation*}

\item[\textbf{Level 1.}] \text{(optimal partition and 
image restoration)}
Find
\begin{equation*}
\begin{aligned}
u^*\in \argmin\left\{\int_Q|u_c-u_{\mathscr{L}}|^2\dd x:\,\mathscr{L}\in
\mathscr{P}\right\} \quad \text{with \(u_{\mathscr{L}}\) given by Level~2}.
\end{aligned}
\end{equation*}
\end{itemize}

As in the case of our learning schemes for the weighted total variation, the analysis of $(\mathscr{L}\mathscr{S})_{TGV_\omega}$ and $(\mathscr{L}\mathscr{S})_{TGV-{\rm Fid}_\omega}$ is performed under a box constraint assumption, which for the first case reads as
\begin{equation}
\label{eq:bc-TGV}
\alpha = (\alpha_0, \alpha_1) \in \left[c_0,\frac{1}{c_0}\right]\times \left[c_1,\frac{1}{c_1}\right].
\end{equation}

Our main result for the weighted-$TGV$ scheme is the following.
\begin{theorem}[Existence of solutions to \((\scrL\!\scrS)_{{TGV\!}_\omega}\)]
\label{thm:TGVomega}
There exists an optimal solution $u^\ast$ to the learning scheme \((\scrL\!\scrS)_{{TGV\!}_\omega}\) in \eqref{lsTGVomega} with the minimization in \eqref{eq:alpha-Lg-TGV} restricted by \eqref{eq:bc-TGV}.
\end{theorem}
Analogously,  we infer the ensuing theorem for the $TGV$ with weighted fidelity.
\begin{theorem}[Existence of solutions to \((\scrL\!\scrS)_{{TGV-Fid}_\omega}\)]\label{thm:FIDomegaTGV}
For every $c\in (0,1)$, there exists an optimal solution
\(u^*\) to the learning scheme \((\scrL\!\scrS)_{{TGV-Fid}_\omega}\)
in \eqref{lsTGVfidomega} with the minimization in \eqref{eq:lambdatgv-Lg} 
restricted by the box constraint $\lambda\in \left[c,\frac1c\right]$. 
\end{theorem}

Also in the case of weighted-$TGV$ learning schemes, we provide a connection between stopping criteria and existence of a box constraint. To be precise, we show that if \eqref{eq:bc-TGV} is imposed, then a stopping criterion can be naturally imposed on the schemes. Concerning the converse implication, we show that if a suitable stopping criterion is enforced, then  $(\alpha_L)_0$ and $(\alpha_L)_1$ are both always bounded from below by a positive constant, and that they cannot simultaneously blow up to infinity.
The weaker nature of this latter implication is due to  one main reason: the upper bound established on the optimal parameters for the weighted $TV$ scheme is hinged upon a suitable Poincar\'e inequality for the total variation functional, cf. Proposition \ref{prop:Ja257}; in the $TGV$ case, the analogous argument only provides a bound from above for the minimum between $(\alpha_L)_0$ and $(\alpha_L)_1$, and thus does not allow to conclude the existence of a uniform upper bound on either component, cf. Proposition \ref{prop:loc-affine}.
We refer to Subsection \ref{sub-l1tgv} for a discussion of this issue and for the details of this argument.  For completeness, we mention that a result related to Proposition~\ref{prop:loc-affine} has been proven in \cite[Proposition 6]{PaVa15}. In Proposition~\ref{prop:loc-affine}, we make this study quantitative and keep track of the dependence on the cell size through the Poincar\'e constant. 

The results we present suggest a number of possible directions and questions for future research. One possible avenue is the formulation of similar schemes with piecewise constant weights in the case of Mumford--Shah regularizations, relating to the Ambrosio--Tortorelli scheme of \cite{FoLi17} which explicitly allows for discontinuous weights. Another is an investigation of the relation and apparent discrepancy between our results concluding stopping of the refinement of partitions, in which parameter variations at very fine scales are not advantageous, and numerical results in the literature where wildly varying parameter maps appear in the optimization, such as in \cite{KoEtAl23}.

The paper is organized as follows: in Section \ref{sect:glo}, we collect some notation which will be employed throughout the paper. The focus of Sections \ref{sect:wTV} and \ref{sect:regTV} is on our weighted-$TV$ scheme, as well as on the two variants thereof, including a regularization of the weight and a weighted fidelity, respectively.
Section \ref{sect:wTGV} is devoted to the study of our weighted-$TGV$ learning scheme and of the corresponding $TGV$ scheme with weighted fidelity. Section \ref{sect:numerics} contains some numerical results for the various learning schemes presented in the paper and a comparison of their performances.

\section{Glossary}\label{sect:glo}
Here we collect some notation that will be used throughout the paper, and introduce some energy functionals that will be studied.

We start by addressing our admissible partitions of the unit cube \(Q=(0,1)^2\) into dyadic squares. For \(\kappa\in\NN_0\), let
\begin{equation*}
\begin{aligned}
Z_\kappa:=\left\{2^{-\kappa}z \in [0,1)^2\colon z\in \ZZ^2 \right\}.
\end{aligned}
\end{equation*}
For instance, \(Z_0=\{(0,0)\}\) and \(Z_1=\{(0,0),(0,\frac12) , (\frac12,0),(\frac12,\frac12)\}.\) Note that \(Z_k\) has cardinality \(2^k\times 2^k\), which allow us to write \(Z_\kappa = \cup_{\iota=1}^{4^\kappa}z_\iota^{(\kappa)}\), where  \(z_\iota^{(\kappa)}=2^{-\kappa}z_\iota\) for a convenient \( z_\iota\in \ZZ^2\). Then, for each  \(\kappa\in\NN_0\) and \(\iota\in\{1,...,4^\kappa\}\), we consider the dyadic square
\begin{equation*}
\begin{aligned}
Q_\iota^\kappa:=\left(z_\iota^{(\kappa)}+\left(0,\frac{1}{2^\kappa}\right]^2 \right) \cap Q.
\end{aligned}
\end{equation*}
For each \(\kappa\in\NN_0\) fixed, we have that  \(Q_{\iota_1}^\kappa \cap Q_{\iota_2}^\kappa=\emptyset\) for every \(\iota_1, \iota_2 \in\{1,...,4^\kappa\}
\) with  \(\iota_1\not=\iota_2\); moreover, \(Q=\cup_{\iota=1}^{4^\kappa} Q_\iota^\kappa\). In particular, \(\mathscr{L}:=\{Q_\iota^\kappa\colon \iota \in\{1,...,4^\kappa\}\}\) provides an example of an admissible partition of \(Q\). More generally, recalling that we denote by \(\mathscr{P}\)  the class of all 
admissible partitions \(\scrL\) of \(Q\) consisting of dyadic squares as above, then if  \(\scrL\in\mathscr{P}\) and \(L\in \scrL\) are arbitrary,  there exist  \(\kappa\in \NN_0\) and   \(\iota\in\{1,...,4^\kappa\}\) such that \(L=Q_\iota^\kappa\).

The setting of our work is a two-dimensional one, mainly due to  the scale invariance
of the constant in the two-dimensional Poincar\'e--Wirtinger
inequality
in \(BV\), as discussed in the proof of Proposition~3.1. This invariance   is crucial to prove existence of solutions for our schemes (see, for instance, Theorem~\ref{thm:solL1}). However, there are
some theoretical results concerning the weighted-\(BV\) and  weighted-\(TGV\) spaces that hold
in any dimension \(n\in\mathbb{N}\), for which reason we state such results
in \(\mathbb{R}^n\).

In what follows,  \(\Omega\subset \RR^n\) is an open and bounded  set and \(\XX\) stands for either \(\RR\), \(\RR^{n},\) or \(\RR^{n\times n}_{sym}\), where the latter is the space of  all \(n\times n\) symmetric matrices and \(n\in\NN\). We denote by \(\Mcal(\Omega;\XX)\) the space of all finite Radon measures in \(\Omega\) with values on \(\XX\), and by \(|\mu|\in\Mcal(\Omega;\RR^+_0) \) the total variation of \(\mu \in\Mcal(\Omega;\XX) \), which is defined for each measurable set \(B\subset \Omega\) by
\begin{equation*}
\begin{aligned}
|\mu|(B):= \sup\bigg\{\sum_{i=1}^\infty |\mu(B_i)|: \,\, \{B_i\}_{i\in\NN} \text{ is a partition of } B\bigg\}.
\end{aligned}
\end{equation*}
Using the Riesz
representation theorem, \(\Mcal(\Omega;\XX)\)  can be identified with the dual of \(C_0(\Omega;\XX')\),  the closure with respect to the supremum norm of the  set of all continuous functions on \(\Omega\) with compact support. In particular,  the total variation of a Radon measure  \(\mu\in \Mcal(\Omega;\XX)\) is alternatively given by
\begin{equation}\label{eq:var}
\begin{aligned}
|\mu|(B) = \sup \bigg \{ \int_B \varphi (x)\cdot\! \dd\mu(x):\ \varphi\in C_0(B;\XX'), \,\, \Vert \varphi\Vert_{L^\infty(B;\XX')} \leq 1 \bigg\}, \enspace B\subset \Omega \text{ measurable,}
\end{aligned}
\end{equation}
where \(\cdot\) represents the duality product between an element of \(\XX'\) and an element of \(\XX\). With the trivial identification of column vectors with row vectors, we will often write \(\XX\) in place of \(\XX'\).

In the case in which \(\mu = Du\in \Mcal(\Omega;\RR^n) \) for some \(u\in BV(\Omega)\), a density argument shows that \eqref{eq:var} is equivalent to
\begin{equation}\label{eq:varBV}
\begin{aligned}
|Du|(B) =\sup \bigg \{ \int_B u(x)\,  \Div \varphi (x) \dd x:\ \varphi\in \Lip_c(B;\RR^n),
\,\, \Vert \varphi\Vert_{L^\infty(B;\RR^n)}  \leq 1 \bigg\},
\end{aligned}
\end{equation}
and we often write \(TV(u,B)\) in place of \(|Du|(B)\).
In the preceding expression, and throughout this manuscript, \(\Lip_c(B;\XX)\) represents the space of all \(\XX\)-valued Lipschitz functions with compact support in \(B\).

Similarly, in the case in which \(\mu = \Ecal v\in \Mcal(\Omega;\RR^{n\times n}_{sym}) \) for some \(v\in BD(\Omega)\) and \(\Ecal\) the symmetrical part of the distributional derivative,
then  \eqref{eq:var}  is equivalent to
\begin{equation}\label{eq:varBD}
\begin{aligned}
|\Ecal v|(B) =\sup \bigg \{ \int_B v(x) \cdot \Div \varphi (x) \dd x:\ \varphi\in
\Lip_c(B;\RR^{n\times n}_{sym}),
\,\, \Vert \varphi\Vert_{L^\infty(B;\RR^{n\times n}_{sym})}  \leq 1 \bigg\},
\end{aligned}
\end{equation}  
where $({\mathrm{div}}\, \varphi)_j = \sum_{k=1}^n \frac{\partial \varphi_{jk}}{\partial
x_k}$  for each $j\in\{1,...,n\}$.

At the core of the present manuscript are weighted versions of the spaces of bounded variation and of bounded deformation. These weighted versions rely on a generalization of \eqref{eq:varBV} and \eqref{eq:varBD} that cannot be derived directly from the Riesz
representation theorem, and thus need a careful analysis to prove the variational identities stated in \eqref{eq:VarBVw} and \eqref{eq:VarBDw1}--\eqref{eq:VarBDw2},      addressed in Sections~\ref{sect:wTV} and \ref{sect:wTGV}, respectively.

 Given a Radon measure \(\mu \in\Mcal(\Omega;\XX) \) and a locally integrable function \(\omega:\Omega\to[0,\infty)\), we define the \(\omega\)-weighted variation of \(\mu\) on \(\Omega\), written \(\scrV_\omega (\mu,\Omega)\), by
\begin{equation}
\label{eq:VarWeight}
\begin{aligned}
\scrV_\omega (\mu,\Omega):=\sup \bigg \{ \int_{\Omega} \varphi (x)\cdot\! \dd\mu(x):\ \varphi\in \Lip_c(\Omega;\XX'),
\,\, | \varphi| \leq \omega \bigg\}.
\end{aligned}
\end{equation}
As before, if  \(\mu = Du\in \Mcal(\Omega;\RR^n) \) for some \(u\in BV(\Omega)\),
then
  \eqref{eq:VarWeight}  is equivalent to
\begin{equation*}
\begin{aligned}
\scrV_\omega (Du,\Omega) =\sup \bigg \{ \int_{\Omega} u(x)\,  \Div \varphi (x) \dd x:\ \varphi\in
\Lip_c(\Omega;\RR^n),
\,\, | \varphi| \leq \omega \bigg\},
\end{aligned}
\end{equation*}
which we often represent by \(TV_\omega(u,\Omega)\), and we define
\begin{equation*}
\begin{aligned}
BV_{\omega}(\Omega):=\left\{ u\colon
\Omega\to\R \text{ measurable:} \int_\Omega |u(x)|\, \omega(x)\dd x<\infty \,\text{ and } TV_{\omega}(u,\Omega)<\infty\right\}.
\end{aligned}
\end{equation*}
Also,  if  \(\mu  = Du-v:= Du-v\Lcal^n\lfloor\Omega \in \Mcal(\Omega;\RR^n) \) for some \(u\in BV(\Omega)\) and \(v\in L^1(\Omega;\RR^n)\),
then \eqref{eq:VarWeight}  is equivalent to
\begin{equation*}
\begin{aligned}
\scrV_\omega (Du-v,\Omega) =\sup \bigg \{ \int_{\Omega} \big(u(x)\,  \Div \varphi (x) +v(x) \cdot \varphi(x)\big) \dd x:\
\varphi\in
\Lip_c(\Omega;\RR^n),
\,\, | \varphi| \leq \omega \bigg\}.
\end{aligned}
\end{equation*}
Moreover, if  \(\mu = \Ecal v\in \Mcal(\Omega;\RR^{n\times n}_{sym})
\) for some \(v\in BD(\Omega)\), then
\eqref{eq:VarWeight}  is equivalent to%
\begin{equation*}
\begin{aligned}
\scrV_\omega (\Ecal v,\Omega) =\sup \bigg \{ \int_{\Omega} v(x) \cdot \Div \varphi (x) \dd x:\ \varphi\in
\Lip_c(\Omega;\RR^{n\times n}_{sym}),
\,\, | \varphi| \leq \omega \bigg\},
\end{aligned}
\end{equation*}
and we define
\begin{equation*}
\begin{aligned}
BD_{\omega}(\Omega):=\left\{ v\colon
\Omega\to\R \text{ measurable:} \int_\Omega |v(x)|\, \omega(x)\dd x<\infty
\,\text{ and } \scrV_\omega (\Ecal v,\Omega)<\infty\right\}.
\end{aligned}
\end{equation*}

 The energy functional associated with the analogue to the ROF's model, where we use a weighted-TV regularizer on \(\Omega\subset \RR^2\) instead of the total variation (TV), is denoted by (see Theorem \ref{thm:S2.1})
\begin{equation*}
\begin{aligned}
E[u]:= \int_\Omega|u_\eta-u|^2\dd x +TV_{\omega}(u,\Omega).
\end{aligned}
\end{equation*}
To highlight the dependence on a partition \(\scrL\) of \(Q\) made of dyadic cubes, the extension of the preceding functional (for a weight \(\omega_\scrL\) and \(\Omega=Q\)) to \(L^1(Q)\)
is represented by

\begin{equation*}
\begin{aligned}
&E_\scrL[u]:=\begin{cases}
\displaystyle\int_Q|u_\eta-u|^2\dd x +TV_{\omega_\mathscr{L}}(u,Q)
&\text{if } u\in
BV_{\omega_{\mathscr{L}}}(Q),\\
+\infty &\text{otherwise}.
\end{cases}
\end{aligned}
\end{equation*}
Moreover, for the \(\epsi\)-dependent regularized weight \(\omega^\epsi_\scrL\), introduced in \eqref{eq:regweight}, the  energy above is written as \begin{equation*}
\begin{aligned}
&E_\scrL^\epsi[ u]:=\begin{cases}
\displaystyle\int_Q|u_\eta-u|^2\dd x +TV_{\omega^\epsi_{\mathscr{L}}}(u,Q)
&\text{if
} u\in
BV_{\omega^\epsi_{\mathscr{L}}}(Q),\\
+\infty &\text{otherwise}.
\end{cases}
\end{aligned}
\end{equation*}
The two preceding functionals are introduced  in Proposition~\ref{thm:Gconvergence}, where we address the relationship between the weighted-TV and the regularized weighted-TV learning schemes in \eqref{lsTVomega} and \eqref{lsTVomegae}, respectively.

For a fixed image domain \(\Omega \subset \RR^2\), the optimal tuning parameter \(\alpha\) in Level~3 of any of the $TV$ learning schemes addressed here is found by minimizing  the cost function \(I:(0,\infty)\to\RR\) defined by
\begin{equation}\label{eq:glosI}
\begin{aligned}
&I(\alpha):=\int_\Omega
|u_c-u_\alpha|^2\dd x \text{ for } \alpha\in(0,+\infty),
\end{aligned}
\end{equation}
where \(u_c\) is the clean image and \(u_\alpha\) is the reconstructed image obtained as the minimizer of the denoising model in aforementioned Level~3. In our analysis, we make use of the extension \(\widehat I:[0,+\infty]\to[0,+\infty]\) of \(I\)  to the closed interval \([0,+\infty] \)  defined for \(\bar\alpha\in [0,+\infty]\) by\begin{equation}\label{eq:lscenvI}
\begin{aligned}
&\widehat I(\bar \alpha):= \inf\Big\{\liminf_{j\to\infty} I(\alpha_j)\!:\,
(\alpha_j)_{j\in\NN}\subset (0,+\infty), \, \alpha_j \to \bar\alpha
\text{ in } [0,+\infty]\Big\},
\end{aligned}
\end{equation}
which can be seen as the   lower-semicontinuous
envelope of \(I\)
on the closed interval \([0,+\infty]\). As it turns out, \(\widehat I\) is actually a continuous function on \(
[0,+\infty]\) (cf.~Corollary~\ref{cor:Icont}).
The study of existence of minimizers for \(I\) and the characterization of \(\widehat I\) for the  weighted-TV  learning scheme in \eqref{lsTVomega} is addressed in  Theorem~\ref{thm:onalpha}, Lemma~\ref{lem:ConvMinTV}, and Corollary~\ref{cor:Icont}. This study relies on the convergence of minimizers of the family, parametrized by \(\alpha\in(0,\infty)\), of energy functionals associated with ROF's  model,
\begin{equation*}
\begin{aligned}
&F_\alpha [u]:=\begin{cases}\displaystyle
\int_{\Omega}|u_\eta-u|^2\dd x+\alpha TV(u,\Omega) &\text{if
} u\in BV(\Omega),\\
+\infty &\text{otherwise.}
\end{cases}\quad 
\end{aligned}
\end{equation*}
In turn, this convergence analysis naturally involves the extreme points \(\bar \alpha = 0\) and   \(\bar\alpha=+\infty\), which are associated with the energies
\begin{equation*}
\begin{aligned}
F_0 [u]:=\begin{cases}\displaystyle
\int_{\Omega}|u_\eta-u|^2\dd x &\text{if
} u\in L^2(\Omega),\\
+\infty &\text{otherwise,}
\end{cases}
\quad \text{and}\quad
F_{\infty} [u]:=\begin{cases}\displaystyle
\int_{\Omega}|u_\eta-c|^2\dd x    &\text{if
} u\equiv c\in\RR, \text{ }\\
+\infty &\text{otherwise,}
\end{cases}
\end{aligned}
\end{equation*}
respectively (we remark that, since the local parameters in each dyadic square are constant, this analysis also applies for the weighted-TV  learning scheme in \eqref{lsTVomega}). 

Regarding the \(TGV\) case, 
 to obtain the existence of optimal parameters for Level~3 of the schemes \eqref{lsTGVomega} and \eqref{lsTGVfidomega}, stated in  Theorem~\ref{thm:onalpha-TGV},  we are led to study $\Gamma$-convergence of the family of functionals, parametrized by $\alpha=(\alpha_0,\alpha_1)\in(0,+\infty)^2$, defined as
\begin{equation*}
\begin{aligned}
&G_\alpha [u]:=\begin{cases}
\int_{\Omega}|u_\eta-u|^2\dd x+ TGV_{\alpha_0,\alpha_1}(u,\Omega) &\text{if
} u\in BV(\Omega),\\
+\infty &\text{otherwise.}
\end{cases}\quad 
\end{aligned}
\end{equation*}
In this case, the \(\Gamma\)-convergence result is more involved because it includes different combinations of $\bar{\alpha}_i = 0$, $\bar{\alpha}_i \in \R^+$, or $\bar{\alpha}_i = +\infty$ for $i=0$ and $i=1$.  The expressions for the ensuing limits can be found in the statement of Lemma \ref{lem:gammaTGV}. 

The characterization of the extension to the closed interval \([0,+\infty]^2\) of the $TGV$ analog of \eqref{eq:glosI}, denoted by $J(\alpha)$ for $\alpha=(\alpha_0, \alpha_1)$, is contained in Lemma \ref{lem:Ilsc-TGV}.

In the sequel, we use both the average of a function $u:\Omega \to \RR$ on a subdomain $L \subset \Omega$,
\begin{equation*}
\begin{aligned}
\lbrack u\rbrack_L:= \frac{1}{|L|}\int_L  u(x)\dd x,
\end{aligned}
\end{equation*}
and its projection onto affine functions $\langle u\rangle_L$, which is the unique solution to the minimum problem
$$\min\left\{\int_L |u-v|^2\dd x:\,v\text{ is affine in }L\right\},$$
where in both cases the subscript may be omitted when $L = \Omega$.

\section{Analysis of the Weighted-TV learning scheme  \texorpdfstring{\((\scrL\!\scrS)_{{TV\!}_\omega}\)}{TV-w}}
\label{sect:wTV}

Here, we prove existence of solutions to the weighted-TV learning scheme,  
\((\scrL\!\scrS)_{{TV\!}_\omega}\), introduced in \eqref{lsTVomega}. We analyze each level in the three subsequent subsections. In
particular, we prove Theorem~\ref{thm:TVomega} in Subsection~\ref{sect:L1}. Then,
in Subsection~\ref{sect:box}, we prove Theorem~\ref{thm:equiv} and we provide different examples of stopping criteria for
the refinement of the admissible partitions introduced in Definition~\ref{def:stop}. 
\subsection{On Level 3}\label{sect:TVL3}
In this section, we discuss the main features of Level 3, and variants thereof,  of the learning scheme 
\((\scrL\!\scrS)_{{TV\!}_\omega}\)  in \eqref{lsTVomega}.

As we mentioned in Remark~\ref{rmk:onTVw}, the parameter \(\alpha_L\) in \eqref{eq:alpha-L} is uniquely determined by definition, with \(\alpha_L\in [0,+\infty]\). Then, in view of Theorem~\ref{thm:onalpha} (see Subsection~\ref{sect:box}), if \(L\in\scrL\) is such that
\begin{equation}
\begin{aligned}\label{eq:dataonL}
TV(u_c,L) < TV(u_\eta,L) \quad \text{ and } \quad \Vert u_\eta - u_c\Vert^2_{L^2(L)}
<\Vert[  u_\eta]_L - u_c\Vert^2_{L^2(L)},
\end{aligned}
\end{equation}
then
\begin{equation*}
\begin{aligned}
\arginf\left\{\int_L |u_c-u_{\alpha,L}|^2\,\dd x\!:\,\alpha\in
\RR^+\right\} = \argmin \left\{\int_L |u_c-u_{\alpha,L}|^2\,\dd x\!:\,\alpha\in
\big[c_L,C_Q\Vert u_\eta\Vert_{L^2(L)}\big]\right\},
\end{aligned}
\end{equation*}
where \(c_L\) and \(C_Q\) are positive constants, with \(c_Q\)  depending only on \(Q\). In 
particular, we have that \(\alpha_L\in
\big[c_L,C_Q\Vert u_\eta\Vert_{L^2(L)}\big].\)
Furthermore, because each partition
$\mathscr{L\in \scrP}$ is finite, it follows that if \eqref{eq:dataonL} holds for all \(L\in \scrL\), then 
\begin{equation*}
\begin{aligned}
\displaystyle\alpha_L\in K_{\mathscr{L}}:=\Big[\min_{L\in\scrL} c_L, C_Q\max_{L\in\scrL}\Vert u_\eta\Vert_{L^2(L)}\Big] \subset (0,+\infty)
\end{aligned}
\end{equation*}
 for every $L\in \mathscr{L}$, which yields a natural box constraint for a fixed partition. Note, however, that the box constraint given by the compact
set $K_{\mathscr{L}}$  may vary according to the choice of
the partition $\mathscr{L}$.

Finally, if we consider Level~3 with \eqref{eq:alpha-L} 
replaced by \eqref{eq:box-min}, then the minimum
\begin{equation*}
\begin{aligned}
\min_{\alpha\in   [c_0,\frac{1}{c_0}]} \int_L |u_c-u_{\alpha,L}|^2\,\dd x
\end{aligned}
\end{equation*}
exists  as the minimum of a lower semicontinuous function (see Corollary~\ref{cor:Icont} in Subsection~\ref{sect:box}) on a compact set. In particular, \(\bar \alpha_L\)  is uniquely determined, with
\begin{equation*}
\begin{aligned}
\bar \alpha_L\in \Big[c_0,\frac{1}{c_0}\Big] \text{ for all \(L\in\scrL\) and \(\scrL\in\scrP\).} 
\end{aligned}
\end{equation*}

\subsection{On Level 2}
\label{sect:L2}

Here, we discuss existence and uniqueness of solutions to the minimization problem in \eqref{eq:minprS2a}. A key step in this discussion is the study of the space  \(BV_{\omega}(\Omega)\) of  \(\omega\)-weighted $BV$-functions in an open set \(\Omega\subset \RR^n\), where the  weight \(\omega:\Omega\to [0,\infty)\)  is assumed to be a  locally integrable function. We adopt the approach introduced  in \cite{Ba01}, and further analyzed in  \cite{Ca10,Ca08}.

Given a \(\omega\)-weighted locally integrable function in \(\Omega\),  \(u\in L^1_{\omega,\loc}(\Omega)\), where
\begin{equation}\label{eq:L1wloc}
\begin{aligned}
L^1_{\omega,\loc}(\Omega):= \bigg\{ v:
\Omega\to\R \text{ measurable:} \int_K |v(x)|\, \omega(x)\dd x<\infty  \text{ for all compact \(K\subset\Omega\)}\bigg\},
\end{aligned}
\end{equation}
we define its  \(\omega\)-weighted  
total variation in \(\Omega,\) \(TV_{\omega}(u,\Omega)\),  by  
\begin{equation}
\begin{aligned}\label{eq:TVwei}
TV_{\omega}(u,\Omega):= \sup\bigg\{ \int_\Omega u\,\diverg \ffi\dd x\!: \, \ffi
\in \Lip_c(\Omega;\RR^2), \, |\ffi| \leq \omega\bigg\}
\end{aligned}
\end{equation}
(see also Section~\ref{sect:glo}). Accordingly, we define the space  \(BV_{\omega}(\Omega)\) of  \(\omega\)-weighted
$BV$-functions in \(\Omega\)  by
\begin{equation*}
\begin{aligned}
BV_{\omega}(\Omega):=\big\{ u\in
L^1_\omega(\Omega)\!:\, TV_{\omega}(u,\Omega)<\infty\big\},
\end{aligned}
\end{equation*}
endowed with the semi-norm
\begin{equation}\label{eq:normBV}
\begin{aligned}
\Vert u\Vert_{BV_{\omega}(\Omega)}:=\Vert u\Vert_{L^1_{\omega}(\Omega)} +TV_{\omega}(u,\Omega), \quad \text{where }  \Vert u\Vert_{L^1_{\omega}(\Omega)} := \int_\Omega |u(x)|\, \omega(x)\dd x.
\end{aligned}
\end{equation}

Clearly, if \(\omega\equiv1\), then we recover the usual space \(BV\) of functions of bounded variation. Moreover,
if \(\omega>0\) (Lebesgue)-a.e.~in \(\Omega\)  and \(\omega \)
belongs to the global Muckenhoupt class \(A_1\), meaning that
there is \(c>0\) such that for (Lebesgue)-a.e.~ \(x\in\Omega\)
and for every ball \(B(x,r)\subset\Omega\), we have  %
\begin{equation}
\label{eq:Muck}
\begin{aligned}
\omega(x)\geq c\lbrack \omega\rbrack_{B(x,r)},
\end{aligned}
\end{equation}
then expression in \eqref{eq:normBV}  defines a norm in \(BV_{\omega}(\Omega)\).
Next, we collect some properties of   \(BV_{\omega}(\Omega)\), proved in \cite{Ba01,Ca10,Ca08}, that will be used in our analysis.

\begin{theorem}\label{thm:ppBVw}
Let \(\Omega\subset \RR^n\) be an open set and let  
\(\omega:\Omega\to [0,\infty)\)   be a  locally integrable function. Then, the following hold:
\begin{itemize}
\item[(i)]The map \(u\mapsto TV_{\omega}(u,\Omega)\) is lower-semicontinuous with respect to the (strong) convergence in \(L^1_{\omega,\loc}(\Omega)\).

\item[(ii)] Given  \(u\in
L^1_{\omega,\loc}(\Omega)\), we have that \(TV_{\omega}(u,\Omega)=TV_{\omega^{sc^-}}(u,\Omega)\), where  \(\omega^{sc^-}\) denotes the lower-semicontinuous envelope
of $\omega$.

\item[(iii)] Assume that \(\omega\) is lower-semicontinuous
and strictly positive everywhere in \(\Omega\). Then, we have that \(u\in
L^1_{\loc}(\Omega)\) and \(TV_{\omega}(u,\Omega)<\infty\) if and only if \(u\in BV_\loc(\Omega)\) and \(\omega\in L^1(\Omega;\vert Du\vert)\). If any of these two equivalent conditions hold, then we have
\begin{equation*}
\begin{aligned}
TV_{\omega}(u,B)= \int_B \omega(x)\dd|Du|(x) 
\end{aligned}
\end{equation*}
for every Borel set \(B\subset \Omega\).
\end{itemize}
\end{theorem}

\begin{proof}
The proof of $(i)$--$(iii)$ may be found in \cite{Ba01} under the additional assumption that \(\omega\) satisfies  a Muckenhoupt $A_1$ condition in \eqref{eq:Muck} (see \cite{Ba01} for the details). Without assuming this extra assumption on \(\omega\), the proof of $(i)$ may be found in \cite[Proposition~1.3.1 and Remark~1.3.2]{Ca08}; the proof of  $(ii)$ follows from  \cite[Proposition~2.1.1 and Theorem~2.1.2]{Ca08}; finally, $(iii)$ is shown in \cite[Theorem~2.1.5]{Ca08}. 
\end{proof}

The existence and uniqueness of  solutions of Level~2 of the learning scheme \((\scrL\!\scrS)_{{TV\!}_\omega}\)
in \eqref{lsTVomega} with  \eqref{eq:alpha-L} 
replaced by \eqref{eq:box-min} are hinged on the following theorem.
\begin{theorem}\label{thm:S2.1}
Let \(v\in L^2(\Omega)\) and let   
\(\omega:\Omega\to (0,\infty)\) be an \(L^\infty\) function with
\(0<\essinf_\Omega \omega \leq \esssup_\Omega \omega<\infty\).
 Then, there exists a unique \(\bar u \in 
BV_{\omega}(\Omega)\) satisfying
\begin{equation*}
\begin{aligned}
\int_\Omega|v-\bar u|^2\dd x +TV_{\omega}(\bar u,\Omega)=\min_{u\in
BV_{\omega}(\Omega)}\bigg\{\int_\Omega|v-u|^2\dd x +TV_{\omega}(u,\Omega) \bigg\}.
\end{aligned}
\end{equation*}
Moreover, denoting by  \(\omega^{sc^-}\)  the lower-semicontinuous envelope
of $\omega$, we have     \(\bar u\in BV_{\omega}(\Omega)\cap BV(\Omega) \cap 
BV_{\omega^{sc^-}}(\Omega) \) and 
\begin{equation*}
\begin{aligned}
TV_{\omega}(\bar u,\Omega)= \int_\Omega \omega^{sc^-}(x)\dd|D\bar u|(x). \end{aligned}
\end{equation*}
\end{theorem}

\begin{proof}
For \(u\in
BV_{\omega}(\Omega)\), set
\begin{equation*}
\begin{aligned}
E[u]:= \int_\Omega|v-u|^2\dd x +TV_{\omega}(u,\Omega),
\end{aligned}
\end{equation*}
and let 
\begin{equation*}
\begin{aligned}
\mathscr{m}:= \inf_{u\in
BV_{\omega}(\Omega)} E[u].
\end{aligned}
\end{equation*}
Note that \(0\leq \mathscr{m} \leq E[0] = \Vert v\Vert_{L^2(\Omega)}^2\), and consider \((u_n)_{n\in\NN} \subset BV_{\omega}(\Omega)\) such that
\begin{equation}\label{eq:minseqL2}
\begin{aligned}
\mathscr{m} = \lim_{n\to\infty} E[u_n].
\end{aligned}
\end{equation}

By hypothesis, there exist \(c_1\), \(c_2\in\RR^+\)  such that for  a.e.~\(x\in\Omega\), we
have
\begin{equation}\label{eq:bddomega}
\begin{aligned}
c_1 \leq \omega(x) \leq  c_2.
\end{aligned} 
\end{equation}
Consequently, for all \(x\in\Omega\),
\begin{equation}\label{eq:bddomegasc-}
\begin{aligned}
c_1 \leq \omega^{sc^-}(x) \leq  c_2.
\end{aligned} 
\end{equation}
Then, in view of \eqref{eq:minseqL2} and Theorem~\ref{thm:ppBVw}~$(ii)$--$(iii)$, for all \(n\in\NN\) sufficiently large, we have 
\begin{equation*}
\begin{aligned}
\mathscr{m} + 1 &\geq\int_\Omega|v-u_n|^2\dd x +TV_{\omega}(u_n,\Omega) = \int_\Omega|v-u_n|^2\dd x +TV_{\omega^{sc^-}}(u_n,\Omega)\\
 &= \int_\Omega|v-u_n|^2\dd x + \int_\Omega \omega^{sc^-}(x)\dd|D u_n|(x) \geq \int_\Omega|v-u_n|^2\dd x + c_1|Du_n|(\Omega).
\end{aligned}
\end{equation*}

Thus, extracting a subsequence if necessary (not relabeled), there exists \(\bar u\in BV(\Omega)\) such that
\begin{equation*}
\begin{aligned}
u_n \weaklystar \bar u \text{ in } BV(\Omega), \quad u_n \weakly \bar u \text{ in } L^2(\Omega), \quad  u_n \to \bar u \text{ in } L^1(\Omega).
\end{aligned}
\end{equation*}
Moreover, by \eqref{eq:bddomega}--\eqref{eq:bddomegasc-} and Theorem~\ref{thm:ppBVw}, we have also     \(\bar u\in BV(\Omega) \cap
BV_{\omega^{sc^-}}(\Omega) \), with
\begin{equation*}
\begin{aligned}
TV_{\omega}(\bar u,\Omega)= \int_\Omega \omega^{sc^-}(x)\dd|D\bar u|(x),
\end{aligned}
\end{equation*}
and    
\begin{equation*}
\begin{aligned}
\mathscr{m}&\leq E[\bar u] = \int_\Omega|v-\bar u|^2\dd x  +TV_{\omega}(\bar u,\Omega)
\\&\leq \liminf_{n\to\infty} \bigg( 
\int_\Omega|v-u_n|^2\dd x +TV_{\omega^{sc^-}}(u_n,\Omega) \bigg) 
 =\lim_{n\to\infty} E[u_n]=\mathscr{m}.
\end{aligned}
\end{equation*}
Because \(|\cdot|^2\) is strictly convex, \(\bar u\) is the unique minimizer of \(E[\cdot]\) over \(
BV_{\omega}(\Omega)\). 
\end{proof}

\begin{corollary}\label{cor:L2}
There exists a unique solution
\(u_\scrL\in BV_{\omega_\scrL}(\Omega)\cap
BV(\Omega) \cap 
BV_{\omega^{sc^-}_\scrL}(\Omega)\) to Level~2 of the learning scheme \((\scrL\!\scrS)_{{TV\!}_\omega}\)
in \eqref{lsTVomega} with  \eqref{eq:alpha-L} 
replaced by \eqref{eq:box-min}, where \(\omega_\mathscr{L}^{sc^-}\) denotes the lower-semicontinuous
envelope
of $\omega_\mathscr{L}$.
Moreover,
\begin{equation*}
\begin{aligned}
\min\bigg\{\int_Q|u_\eta-u|^2\dd x +TV_{\omega_\mathscr{L}}(u,Q)\!:\,u\in
BV_{\omega_{\mathscr{L}}}(Q)\bigg\}=\int_Q|u_\eta-u_\scrL|^2\dd x + \int_\Omega \omega^{sc^-}_\scrL(x)\dd|Du_\scrL|(x).
\end{aligned}
\end{equation*}
\end{corollary}

\begin{proof}
Using the analysis in Subsection~\ref{sect:TVL3}, the function \(\omega_\scrL\)
in \eqref{eq:weight} satisfies the bounds \(c_0 \leq\omega_\scrL\leq
\tfrac1{c_0} \)  in \(Q\), which, together with Theorem~\ref{thm:S2.1},
concludes the proof. \end{proof}

\begin{remark}\label{rmk:L2bc}Recalling once again the analysis in Subsection~\ref{sect:TVL3},
the previous corollary still holds if we assume that  \eqref{eq:dataonL}
holds for all \(L\in \scrL\)  instead of replacing   \eqref{eq:alpha-L} 
 by \eqref{eq:box-min}.\end{remark}

\subsection{On Level 1}\label{sect:L1}

Here, we prove that Level~1 of the learning
scheme \((\scrL\!\scrS)_{{TV\!}_\omega}\)
admits a solution provided we consider a stopping criterion as in Definition~\ref{def:stop}. 
We start by checking that the box constraint \eqref{eq:boxc} yields
such a stopping criterion, after which we establish the converse
statement. We then explore alternative stopping criteria. 
 
To prove that the box constraint \eqref{eq:boxc} yields
 a stopping criterion for the refinement of the admissible partitions, we first recall the existence of a smallness condition
on the tuning parameter under which the restored image given by the TV model is constant.  

\begin{proposition}\label{prop:Ja257}
There exists a positive constant, \(C_Q\), depending only on \(Q\),
such that for any dyadic cube \(L\subset Q\) and for all \(\alpha\geq
C_Q \Vert u_\eta\Vert_{L^2(L)}\), the solution \(u_{\alpha,L}\) of
\eqref{eq:ROF} is constant, with \(u_{\alpha,L} \equiv [ u_\eta]_L\).
\end{proposition}

\begin{proof}
The proof is a simple consequence of \cite[Proposition~2.5.7]{Ja12} combined with the scaling invariance of the constant in the 2-dimensional Poincar\'e--Wirtinger inequality
in \(BV\) (see \cite[Remark~3.50]{AmFuPa00}). 
\end{proof}

\begin{theorem}\label{thm:solL1} Consider the learning
scheme \((\scrL\!\scrS)_{{TV\!}_\omega}\)
in \eqref{lsTVomega} with  \eqref{eq:alpha-L} 
replaced by \eqref{eq:box-min}. 
Then, there exist \(\kappa\in \NN\) and \(\mathscr{L}_1, ...
, \mathscr{L}_\kappa \in \mathscr{P}\) such that
\begin{equation}\label{eq:finmin}
\begin{aligned}
\argmin\left\{\int_Q|u_c-u_{\mathscr{L}}|^2\dd x:\,\mathscr{L}\in
\mathscr{P}\right\} = \argmin\left\{\int_Q|u_c-u_{\mathscr{L}_i}|^2\dd x:\,i\in
\{1,...,\kappa\}\right\}.
\end{aligned}
\end{equation}
\end{theorem}

\begin{proof}
We use
Proposition~\ref{prop:Ja257} to prove that if a partition
 contains dyadic squares of side length smaller than a certain threshold, then it can be replaced by a partition of dyadic squares of side
length greater than that threshold without changing the minimizer
at Level~2.

 Let \(\bar\epsi\in(0,1)\) be such that
for every measurable set \(E\subset Q\) with \(|E|\leq \bar\epsi\),
we have
\begin{equation}\label{eq:equii}
\begin{aligned}
 \Vert u_\eta\Vert_{L^2(E)} \leq \frac{c_0}{C_Q},
\end{aligned}
\end{equation}
where \(c_0\) is the constant in  \eqref{eq:box-min} and \(C_Q\)
is the constant given by Proposition~\ref{prop:Ja257}. Set
\begin{equation*}
\begin{aligned}
\bar k:= \min \Big\{k\in\NN\!: \, \frac{1}{4^{k}}\leq \bar \epsi\Big\}\quad \text{and} \quad \bar{\mathscr{P}}:= 
\Big\{\mathscr{L}\in \mathscr{P}\!:\, |L|\geq\frac{1}{4^{\bar
k}} \text{ for all \(L\in
\mathscr{L}\)}\Big\}.
\end{aligned}
\end{equation*}
Note that \(\bar{\mathscr{P}}\) has finite cardinality. Finally, define 
\[\mathscr{P}^*:= \mathscr{P} \setminus \bar{\mathscr{P}}.\]

Fix \(\mathscr{L}^*\in
\mathscr{P}^*\), and let
\begin{equation*}
\begin{aligned}
\mathscr{L}^*_-:=\{L^*\in \mathscr{L}^*\!:\,  |\tilde L^*|\geq |L^*| \text{ for all \(\tilde L^*\in \mathscr{L}^*\}
\)}
\end{aligned}
\end{equation*}
be the collection of all  dyadic squares with
the smallest side length in \(\mathscr{L}^*\). Then, there exists
 \(k^*\in\NN\), with \(k^*> \bar
k\),  such that  \(|L^*| =
\frac{1}{4^{k^*}}\)   for all \(L^*\in\mathscr{L}^*_- \). 
Moreover, by construction of our admissible partitions,
we can write
\begin{equation*}
\begin{aligned}
\mathscr{L}^*_-= \medcup_{j=1}^\ell\{L^*_{j,i}\}_{i=1}^4 \enspace \text{
for some } \ell\in\NN,
\end{aligned}
\end{equation*}
where, for each \(j\in\{1,...,\ell\}\), 
\begin{equation*}
\begin{aligned}
\medcup_{i=1}^4L^*_{j,i} =:\bar   L^*_{j} \enspace  \text{ is a dyadic
square with } |\bar L^*_{j}|=\frac{1}{4^{k^*-1}}.
\end{aligned}
\end{equation*}

Note that \(k^*-1\geq \bar
k\). Then, for any \(\alpha\in
[c_0,1/c_0]\),  Proposition~\ref{prop:Ja257} and \eqref{eq:equii}
yield %
\begin{equation*}
\begin{aligned}
\int_{L^*_{j,i}} |u_c-u_{\alpha,L^*_{j,i}}|\dd x = \int_{L^*_{j,i}} |u_c- [ u_\eta]_{L^*_{j,i}}|\dd x
\quad \text{and} \quad
\int_{\bar L^*_{j}} |u_c-u_{\alpha,\bar L^*_{j}}|\dd x = \int_{\bar L^*_{j}} |u_c-
[ u_\eta]_{\bar L^*_{j}}|\dd x
\end{aligned}
\end{equation*}
for all \(j\in\{1,...,\ell\}\) and \(i\in\{1,...,4\}\). Thus,
by \eqref{eq:box-min},
\begin{equation*}
\begin{aligned}
\alpha_{L^*_{j,i}} = \alpha_{\bar L^*_{j}}=c_0
\end{aligned}
\end{equation*}
for all \(j\in\{1,...,\ell\}\) and \(i\in\{1,...,4\}\). Consequently
(see Figure~\ref{fig:partition}),
defining
\begin{equation*}
\begin{aligned}
\bar{\mathscr{L}^*}:=(\mathscr{L}^*\setminus\mathscr{L}^*_-)
\cup \medcup_{j=1}^\ell \bar L^*_{j},
\end{aligned}
\end{equation*}
we have \(\bar{\mathscr{L}^*} \in\mathscr{P} \) and, recalling
Level~2, %
\begin{equation*}
\begin{aligned}
\omega_{\bar{\mathscr{L}^*}} \equiv \omega_{{\mathscr{L}^*}}
\quad \text{and} \quad u_{\omega_{\bar{\mathscr{L}^*}}} \equiv u_{\omega_{{\mathscr{L}^*}}}.
\end{aligned}
\end{equation*}
\begin{figure}
      \includegraphics[scale=.28]{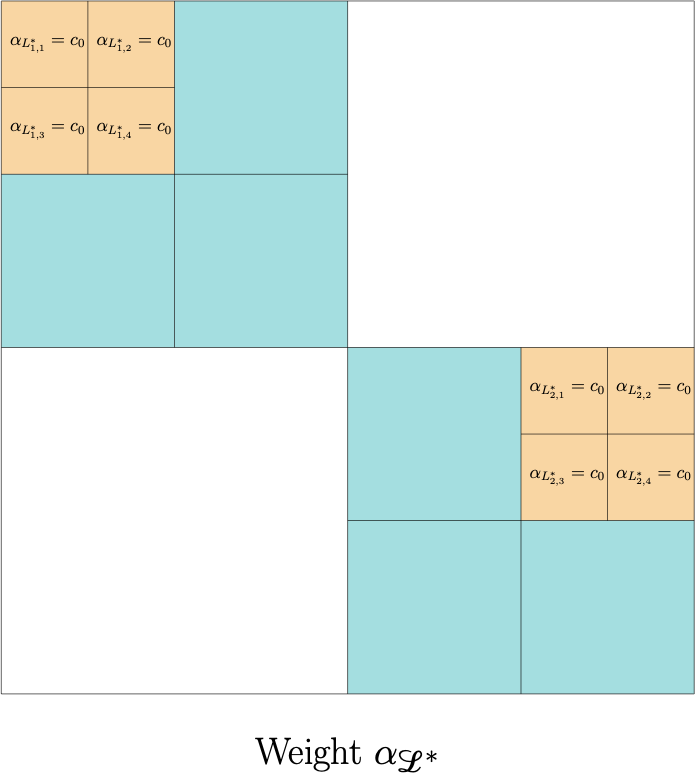}
         \hskip15mm\includegraphics[scale=.28]{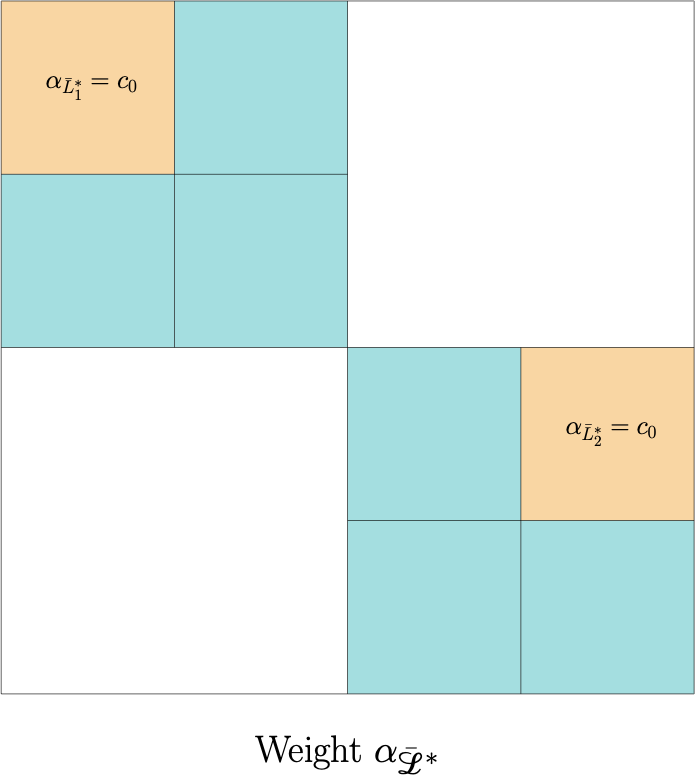}
\caption{Example of two partitions, \(\scrL^*\) and \(\bar\scrL^*\),
that yield the same solution at Level~2.}\label{fig:partition}
\end{figure}

Note also that \(|\bar L^*|\geq\frac{1}{4^{k^*-1}} \) for all
\(\bar L^*\in\bar{\mathscr{L}^*} \). If \(k^*-1 = \bar k\), we conclude that \(\bar{\mathscr{L}^*} \in \bar{\mathscr{P}} \).
Otherwise, if \(k^*-1 > \bar k\), we repeat the construction above
\(  k^*-1-\bar k \) times to obtain a partition  \(\hat{\mathscr{L}}^* \in \bar{\mathscr{P}} \) for which   \[u_{\omega_{\hat{\mathscr{L}}^*}} \equiv
u_{\omega_{{\mathscr{L}^*}}}.\]

Repeating this argument for each  \(\mathscr{L}^*\in
\mathscr{P}^*\), and recalling that  \(\bar{\mathscr{P}}\) has finite cardinality, we deduce \eqref{eq:finmin}.  
\end{proof}

\begin{remark}\label{rmk:boxtresh}
We have shown in the previous proof that the box constraint
condition yields a threshold on the minimum  side length of the dyadic squares of the possible optimal partitions \(\scrL\) of
\(Q\). In other words, the box constraint
condition yields the following stopping criterion for the refinement of the admissible partitions:

\smallskip
\((\scrS)\) There exists \(\kappa\in\NN\) such that \(|L|\geq
\frac1{4^\kappa}\) for all \(L\in\scrL\).\\
\smallskip
In the next subsection, we establish the converse of this implication (see the proof of Theorem~\ref{thm:equiv}).
\end{remark}

We conclude this section by proving Theorem~\ref{thm:TVomega}
that shows  the existence of an optimal solution to the learning
scheme \((\scrL\!\scrS)_{{TV\!}_\omega}\).

\begin{proof}[Proof of Theorem~\ref{thm:TVomega}]
This result is an immediate consequence of the results of Subsection~\ref{sect:TVL3}, Corollary~\ref{cor:L2},
and Theorem~\ref{thm:solL1}. 
\end{proof}

\subsection{Stopping Criteria and Box Constraint}\label{sect:box}
In this subsection, we provide different examples of 
 stopping criteria for
the refinement of the admissible partitions, which notion was
 introduced in Definition~\ref{def:stop}, and we prove Theorem~\ref{thm:equiv}.
The latter is based on the following theorem that yields a natural
box constraint for the optimal parameter \(\alpha\) associated
with the TV model, provided the training data satisfy some
mild conditions. The proof of \eqref{eq:minIalpha} in Theorem~\ref{thm:onalpha}
below uses arguments from \cite{DFKH22}
that are alternative to those in \cite{DeScVa16}. 

\begin{theorem}\label{thm:onalpha}
Let \(\Omega\subset \RR^2\) be a bounded, Lipschitz domain and, for each \(\alpha\in(0,+\infty)\), let \(u_\alpha\in BV(\Omega)\) be  given by \eqref{eq:ROF} with \(L\) replaced by \(\Omega\). Assume that the two following conditions on the training data
hold: 
\begin{itemize}
\item [\textit{i)}] $TV(u_c,\Omega) < TV(u_\eta,\Omega)$;

\item [\textit{ii)}] $\displaystyle \Vert u_\eta - u_c\Vert^2_{L^2(\Omega)}
<\Vert[  u_\eta]_\Omega - u_c\Vert^2_{L^2(\Omega)} $.
\end{itemize}
 Then, there exists \( \alpha^*_\Omega\in (0,+\infty)\) such that
\begin{equation}
\label{eq:minIalpha}
\begin{aligned}
I( \alpha^*_\Omega)=\min_{\alpha\in(0,+\infty)} I(\alpha) 
\quad \hbox{ where } \quad
I(\alpha):=\int_\Omega
|u_c-u_\alpha|^2\dd x \text{ for } \alpha\in(0,+\infty).
\end{aligned}
\end{equation}
Moreover, there exist positive constants \(c_\Omega\) and \(C_\Omega\), such that any
minimizer, \(\alpha^*_\Omega\), of \(I\) over \((0,+\infty)\) satisfies
\( c_\Omega\leq \alpha^*_\Omega < C_\Omega\Vert u_\eta\Vert_{L^2(\Omega)}\). Furthermore, if
 \(\Omega=L\) with   \(L\subset Q\)  a dyadic square,
then  there exists a positive constant \(c_L\) such that any
minimizer, \(\alpha^*_L\) of \(I\) over \((0,+\infty)\) satisfies
\( c_L\leq \alpha^*_L < C_Q\Vert u_\eta\Vert_{L^2(L)}\), where
\(C_Q\) is the constant given by Proposition~\ref{prop:Ja257}.
In particular, 
\(\alpha^*_L\to0\)  as \(|L|\to0\).
\end{theorem}

\begin{remark}\label{rmk:cts}
The constants \(C_\Omega\) and \(C_Q\) characterizing
the upper bound for the optimal parameters in Theorem~\ref{thm:onalpha}
depend only on the domains, \(\Omega\) and \(Q\), respectively
(cf.~Proposition~\ref{prop:Ja257}). On the other hand, the constants
\(c_\Omega\) and \(c_L\) providing a lower bound depend not only on the
corresponding domain, but also on \(u_c\) and \(u_\eta\).
\end{remark}
 
The proof of Theorem~\ref{thm:onalpha} is hinged on the next lemma of continuity with respect to the parameter in the ROF functional, including the limit cases where the parameter vanishes or tends to $+\infty$.

\begin{lemma}\label{lem:ConvMinTV}
Let \(\Omega\subset \RR^2\) be a bounded, Lipschitz domain and, for each
\(\alpha\in(0,+\infty)\), let \(u_\alpha\in BV(\Omega)\) be  given by \eqref{eq:ROF}
with \(L\) replaced by \(\Omega\). Consider the family of  functionals \((F_{\bar \alpha})_{\bar\alpha\in[0,+\infty]}\),
where \(F_{\bar \alpha}:L^2(\Omega)\to[0,+\infty]\) is defined by
\begin{equation*}
\begin{aligned}
&F_\alpha [u]:=\begin{cases}
\int_{\Omega}|u_\eta-u|^2\dd x+\alpha TV(u,\Omega) &\text{if
} u\in BV(\Omega),\\
+\infty &\text{otherwise,}
\end{cases}\quad \text{ for } \bar\alpha=\alpha\in(0,+\infty),\\
&F_0 [u]:= \int_{\Omega}|u_\eta-u|^2\dd x \quad \text{ for } \bar\alpha=0,\\
&F_{\infty} [u]:=\begin{cases}
\int_{\Omega}|u_\eta-c|^2\dd x    &\text{if
} u\equiv c\in\RR, \text{ }\\
+\infty &\text{otherwise,}
\end{cases}\quad \text{ for } \bar\alpha=+\infty,
\end{aligned}
\end{equation*}
and denote by \(u_{\bar\alpha}:= \argmin_{u\in
L^2(\Omega)}F_{\bar \alpha}[u]\) their unique minimizers, given by
\begin{equation}\label{eq:ubaralpha}
\begin{aligned}
u_{\bar \alpha} = \begin{cases}
u_\alpha &\text{if } \bar\alpha=\alpha,\\
u_\eta &\text{if } \bar\alpha=0,\\
[ u_\eta]_\Omega &\text{if } \bar\alpha=+\infty.\\
\end{cases}
\end{aligned}
\end{equation}
Let \((\alpha_j)_{j\in\NN} \subset (0,+\infty)\) and \(\bar
\alpha\in [0,\infty]\) be such that \(\alpha_j\to \bar \alpha\)
in \([0,+\infty]\). Then we have that $u_{\alpha_j} \to u_{\bar \alpha}$ strongly in $L^2(\Omega)$.
\end{lemma}
\begin{proof} We treat the cases $\bar \alpha \in (0, +\infty)$, $\bar \alpha = 0$, and $\bar \alpha = +\infty$ separately.

Let us first assume that $\bar\alpha \in (0,+\infty)$. The proof of this case essentially follows the computations in \cite[Thm.~2.4.20]{Ja12}, but since our notation and focus are different, we present a complete proof adapted to our setting. Being $u_{\alpha_j}$  a minimizer of $F_{\alpha_j}[u]$ and $u_{\bar \alpha}$  a minimizer of $F_{\bar \alpha}[u]$, we get that 
\begin{align*}
u_\eta - u_{\bar \alpha} &= \bar \alpha p_{\bar \alpha} \quad \text{ with } p_{\bar \alpha} \in \partial TV[u_{\bar \alpha}],\\
u_\eta - u_{\alpha_j} &= \alpha_j p_{\alpha_j} \quad \text{ with } p_{\alpha_j} \in \partial TV[u_{\alpha_j}],
\end{align*}
where $\partial TV$ denotes the subdifferential in $L^2(\Omega)$ of $TV$ (extended to be $+\infty$ on $L^2(\Omega)\setminus BV(\Omega)$). Multiplying the first equality by $\alpha_j / \bar \alpha$ and subtracting the second one from it, we obtain
\begin{align*}
\alpha_j (p_{\bar \alpha} - p_{\alpha_j}) & = \frac{\alpha_j}{\bar \alpha} (u_\eta - u_{\bar \alpha}) - (u_\eta - u_{\alpha_j})\\
& =\left(\frac{\alpha_j}{\bar\alpha}-1\right)(u_\eta - u_{\bar \alpha}) + u_{\alpha_j} - u_{\bar \alpha}.
\end{align*}
Multiplying the preceding identity by $u_{\bar \alpha} - u_{\alpha_j}$, integrating over \(\Omega\), and using the monotonicity of $\partial TV$, we obtain
\[
0 \leq \left(\frac{\alpha_j}{\bar\alpha}-1\right) \int_\Omega (u_\eta - u_{\bar \alpha})(u_{\bar \alpha} - u_{\alpha_j})\dd x - \|u_{\bar \alpha} - u_{\alpha_j}\|^2_{L^2(\Omega)}.
\]
Consequently, using Cauchy--Schwarz's inequality, and reorganizing the terms, it follows that
\[
\|u_{\bar \alpha} - u_{\alpha_j}\|_{L^2(\Omega)} \leq \frac{|\alpha_j - \bar{\alpha}|}{\bar\alpha} \|u_\eta - u_{\bar \alpha}\|_{L^2(\Omega)}.\]
On the other hand,  taking into account $u_{\bar \alpha} = \argmin_{L^2(\Omega)} F_{\bar \alpha}$, we have that 
\[\|u_{\bar \alpha} - u_\eta\|^2_{L^2(\Omega)} \leq \|u_{\bar \alpha} - u_\eta\|^2_{L^2(\Omega)} + \bar\alpha TV(u_{\bar \alpha},\Omega) =F_{\bar \alpha}[u_{\bar \alpha}] \leq F_{\bar \alpha}[0] = \|u_\eta\|^2_{L^2(\Omega)},\]
which, together with the preceding estimate, yields
\[\|u_{\bar \alpha} - u_{\alpha_j}\|_{L^2(\Omega)} \leq \frac{|\alpha_j - \bar{\alpha}|}{\bar\alpha} \|u_\eta\|_{L^2(\Omega)}.\]

We now consider the $\bar \alpha=0$ case. Because \(\Omega\) is a bounded, Lipschitz domain, we can find
a sequence \((\hat u_\kappa)_{\kappa\in\NN}\in C^\infty(\overline\Omega)\subset BV(\Omega)\)
such that \(\hat u_\kappa\to u_\eta\) in \(L^2(\Omega)\). Since \((\alpha_j)^{-\frac12}\to\infty\), we can modify \((\hat u_\kappa)_{\kappa\in\NN}\) by repeating each of its elements as (finitely) many times as necessary so that the resulting sequence, denoted by \((u_j)_{j\in\NN}\), satisfies \(TV(u_j,\Omega)\leq(\alpha_j)^{-\frac12}\) for all \(j\in\NN\) large enough. Thus,  \(u_j\to u_\eta\) in \(L^2(\Omega)\) and \( \lim_{j\to\infty}\alpha_jTV(u_j,\Omega)=0\). Using this sequence in the minimality of $u_{\alpha_j}$ results in
\[\|u_{\alpha_j} - u_\eta\|^2_{L^2(\Omega)} + \alpha_j TV(u_{\alpha_j}) \leq \|u_{j} - u_\eta\|^2_{L^2(\Omega)} + \alpha_j TV(u_{j}), \]
Because both terms on the right-hand side converge to zero, we conclude that $(u_{\alpha_j})_{j \in \N}$ converges to $u_\eta$ strongly in $L^2(\Omega)$, as well.

We are left to treat the $\bar \alpha = +\infty$ case. First, we claim that $[u_{\alpha_j}]_\Omega = [u_{\eta}]_\Omega$ for all $j\in\NN$. To see this, we use $u_{\alpha_j} = \argmin_{u \in BV(\Omega)} F_{\alpha_j}[u]$ to get for any $c \in \R$ that
\[\|u_{\alpha_j} - u_\eta\|^2_{L^2(\Omega)} + \alpha_j TV(u_{\alpha_j},\Omega) \leq \|u_{\alpha_j} - u_\eta - c\|^2_{L^2(\Omega)} + \alpha_j TV(u_{\alpha_j},\Omega).\]
Thus, \(\|u_{\alpha_j} - u_\eta\|^2_{L^2(\Omega)}  \leq\|u_{\alpha_j} - u_\eta - c\|^2_{L^2(\Omega)}  \). Moreover, we also know that
\[[u_{\alpha_j} - u_\eta]_\Omega = \argmin_{c \in \R} \|u_{\alpha_j} - u_\eta - c\|^2_{L^2(\Omega)}\]
with only one minimizer by strict convexity, which would lead to a contradiction with the previous inequality unless $[u_{\alpha_j} - u_\eta]_\Omega = 0$. In other words, we must have  $[u_{\alpha_j}] = [u_\eta]_\Omega$ for all $j\in\NN$. To conclude, we use the estimate $F_{\alpha_j}[u_{\alpha_j}] \leq \|u_\eta\|^2_{L^2(\Omega)}$ as above, which by the definition of $F_{\alpha_j}$ implies that
\[\lim_{j \to \infty} TV(u_{\alpha_j},\Omega) = 0.\]
Moreover, by the Poincar\'e inequality, we have that
\[\|u_{\alpha_j} - [u_\eta]_{\Omega}\|_{L^2(\Omega)} = \|u_{\alpha_j} - [u_{\alpha_j}]_{\Omega}\|_{L^2(\Omega)} \leq C \, TV(u_{\alpha_j},\Omega).\]
Thus, $(u_{\alpha_j})_{j \in \N}$ converges to $[u_\eta]_{\Omega}$ strongly in $L^2(\Omega)$.
\end{proof}

From the preceding lemma, we immediately deduce the following corollary.

\begin{corollary}\label{cor:Icont}
Let \(\Omega\subset \RR^2\) be a bounded, Lipschitz domain, and
let \(I:(0,+\infty)\to[0,+\infty)\) be the function defined
in \eqref{eq:minIalpha}. Then, $I$ can be extended continuously to a function \(\widehat I:[0,+\infty]\to[0,+\infty]\) defined for \(\bar\alpha\in
[0,+\infty]\) by
\begin{equation}
\label{eq:Icont}
\begin{aligned}
\widehat I (\bar\alpha) = \begin{cases}
I(\alpha)=\Vert u_\alpha - u_c\Vert^2_{L^2(\Omega)} & \text{if
} \bar \alpha= \alpha\in(0,+\infty),\\
\Vert u_\eta - u_c\Vert^2_{L^2(\Omega)} & \text{if } \bar\alpha=0,\\
\Vert [ u_\eta]_\Omega - u_c\Vert^2_{L^2(\Omega)}  & \text{if
} \bar\alpha=+\infty.
\end{cases}
\end{aligned}
\end{equation}
\end{corollary}

\begin{remark}

We observe that the only continuity condition on \(\widehat I\) needed for our analysis to hold is 
 that of lower semicontinuity of $\widehat I$, as given by \eqref{eq:lscenvI}. However, because it is not hard to prove continuity on the whole of $[0, +\infty]$ in the TV case, we have done so in the results above, which we believe to be of interest on their own.
\end{remark}

\begin{proof}[Proof of  Theorem~\ref{thm:onalpha}]
 We will proceed in three steps.

\medskip
{\it{\uline{Step~1.}} We prove that if condition \textit{i)} in  the statement holds (i.e.,  $ TV(u_\eta,\Omega)- TV(u_c,\Omega) >0$), then
there exists \(\alpha\in(0,+\infty)\) such that
\begin{equation}
\label{eq:notat0}
\begin{aligned}
\Vert u_\alpha - u_c\Vert^2_{L^2(\Omega)}<\Vert u_\eta - u_c\Vert^2_{L^2(\Omega)}.
\end{aligned}
\end{equation}
}

To show \eqref{eq:notat0}, we first recall (see \cite{ChLi97}) that for any \(\alpha\in(0,+\infty)\),
there exists a unique \(u_\alpha\in BV(\Omega) \subset L^2(\Omega)\) such that
\begin{equation}\label{eq:mincond}
\begin{aligned}
u_\alpha = \argmin_{u\in L^1(\Omega)} F_\alpha[u] = \argmin_{u\in L^2(\Omega)} F_\alpha[u],  
\end{aligned}
\end{equation}
which allow us to regard \(F_\alpha\) as a sum of two convex functionals
on \(L^2(\Omega)\) with values in \([0,+\infty]\). Precisely, %
\begin{equation*}
\begin{aligned}
F_\alpha[u]= F_\alpha^1[u] + F_\alpha^2[u],  
\end{aligned}
\end{equation*}
where, for \(u\in L^2(\Omega)\), 
\begin{equation*}
\begin{aligned}
F_\alpha^1[u] := \Vert u-u_\eta\Vert^2_{L^2(\Omega)} \quad \text{
and }\quad F_\alpha^2 [u]:=\begin{cases}
\alpha TV(u,\Omega) &\text{if
} u\in BV(\Omega),\\
+\infty &\text{otherwise.}
\end{cases}
\end{aligned}
\end{equation*}

Denoting by
\(\partial F(v) \in (L^2(\Omega))'\cong L^2(\Omega)\) the subdifferential of a  convex functional \(F:L^2(\Omega) \to [0,+\infty]\) at
\(v\in L^2(\Omega)\), we conclude from \eqref{eq:mincond} that
\begin{equation*}
\begin{aligned}
0\in \partial F_\alpha(u_\alpha)   \enspace \text{ or, equivalently,
}\enspace 2(u_\eta - u_\alpha) \in \partial F_\alpha^2(u_\alpha).
\end{aligned}
\end{equation*}
Consequently,
\begin{equation*}
\begin{aligned}
0&\geq F^2_\alpha[ u_\alpha] - F^2_\alpha[ u_c] +\int_\Omega 2(u_\eta - u_\alpha)(u_c
- u_\alpha)\dd x\\ 
& \geq F^2_\alpha[ u_\alpha]- F^2_\alpha[ u_c]+\int_\Omega 2(u_\eta
- u_\alpha)(u_c
- u_\alpha)\dd x -\Vert
u_\alpha - u_\eta\Vert^2_{L^2(\Omega)}  
 \\&= \alpha\big( TV(u_\alpha,\Omega) -  TV(u_c,\Omega)\big)+\Vert u_\alpha - u_c\Vert^2_{L^2(\Omega)} 
 - \Vert u_\eta - u_c\Vert^2_{L^2(\Omega)}.
\end{aligned}
\end{equation*}
Hence,
\begin{equation}\label{eq:notat0'}
\begin{aligned}
\Vert u_\eta - u_c\Vert^2_{L^2(\Omega)} -\Vert
u_\alpha - u_c\Vert^2_{L^2(\Omega)} \geq \alpha\big( TV(u_\alpha,\Omega) -  TV(u_c,\Omega)\big). 
\end{aligned}
\end{equation}

We claim that %
\begin{equation}
\label{eq:mon}
\begin{aligned}
TV(u_\alpha,\Omega)\nearrow\ TV(u_\eta,\Omega)
\enspace \text{ as } \enspace \alpha\searrow 0.
\end{aligned}
\end{equation}
Assuming that the preceding claim holds, the condition $TV(u_\eta,\Omega)- TV(u_c,\Omega)
>0$ allows us to  find \(\tilde \alpha \in
(0,+\infty)\) for which the left-hand side of \eqref{eq:notat0'} with \(\alpha=\tilde \alpha\) is strictly
positive.  Thus, \(\Vert u_\eta - u_c\Vert^2_{L^2(\Omega)} >\Vert
u_{\tilde\alpha} - u_c\Vert^2_{L^2(\Omega)}\), which proves \eqref{eq:notat0}. 

To conclude Step~1, we are left to prove \eqref{eq:mon}. Using
\eqref{eq:mincond}, for all \(\alpha\), \(\beta \in
(0,+\infty)\) with \(\alpha<\beta\),
we have that%
\begin{equation*}
\begin{aligned}
&\beta TV(u_\beta,\Omega) \leq F_\beta[ u_\beta] \leq F_\beta[ u_\eta]=
\beta TV(u_\eta,\Omega)
\end{aligned}
\end{equation*}
and
\begin{equation*}
\begin{aligned}
  \Vert u_\alpha-u_\eta\Vert^2_{L^2(\Omega) } + \alpha TV(u_\alpha,\Omega)
&\leq \Vert u_\beta-u_\eta\Vert^2_{L^2(\Omega) } + \alpha TV(u_\beta,\Omega)\\&= \Vert u_\beta-u_\eta\Vert^2_{L^2(\Omega) } + \beta TV(u_\beta,\Omega)
+ (\alpha - \beta)TV(u_\beta,\Omega)  \\&\leq\Vert u_\alpha-u_\eta\Vert^2_{L^2(\Omega) } + \beta TV(u_\alpha,\Omega) + (\alpha - \beta)TV(u_\beta,\Omega),
\end{aligned}
\end{equation*}
from which we get that 
\begin{equation*}
\begin{aligned}
\beta TV(u_\beta,\Omega) \leq \beta TV(u_\eta,\Omega) \quad \text{and} \quad (\alpha - \beta) TV(u_\alpha,\Omega) \leq(\alpha - \beta)TV(u_\beta,\Omega). \end{aligned}
\end{equation*}
Hence, recalling that \(\beta>0\) and  \(\alpha-\beta<0\), it follows that \(TV(u_\beta,\Omega)\leq  TV(u_\eta,\Omega)\) and \( TV(u_\alpha,\Omega)
\geq  TV(u_\beta,\Omega)\). Finally, using the first of these estimates and Lemma \ref{lem:ConvMinTV} with an arbitrary decreasing
sequence
\((\beta_j)_{j\in\NN}\) converging to 0, the lower-semicontinuity
of the total variation with respect to the strong convergence
in \(L^1\) yields%
\begin{equation*}
\begin{aligned}
TV(u_\eta,\Omega) \geq \limsup_{j\to\infty} TV(u_{\beta_j},\Omega)
\geq\liminf_{j\to\infty} TV(u_{\beta_j},\Omega)\geq TV(u_\eta,\Omega).
\end{aligned}
\end{equation*}
This concludes the proof of \eqref{eq:mon}.

\medskip
{\it{\uline{Step~2.}} We prove that if condition \textit{ii)}
in  the statement holds, (i.e., $ \Vert u_\eta - u_c\Vert^2_{L^2(\Omega)}
<\Vert[  u_\eta]_\Omega - u_c\Vert^2_{L^2(\Omega)}$), then
there exits \(\alpha\in(0,+\infty)\) such that
\begin{equation}
\label{eq:notati}
\begin{aligned}
\Vert u_\alpha - u_c\Vert^2_{L^2(\Omega)}<\Vert[  u_\eta ]_\Omega- u_c\Vert^2_{L^2(\Omega)}.
\end{aligned}
\end{equation}
}

Using Corollary~\ref{cor:Icont} with \(\bar\alpha=0\) together with \textit{ii)}, we obtain
\begin{equation*}
\begin{aligned}
\limsup_{j\to\infty} \Vert  u_{\alpha_j} - u_c\Vert_{L^2(\Omega)}
&\leq \limsup_{j\to\infty}\Big( \Vert  u_{\alpha_j} -  u_\eta
\Vert_{L^2(\Omega)} +\Vert  u_\eta
- u_c\Vert_{L^2(\Omega)} \Big) \\&= \Vert u_\eta - u_c\Vert_{L^2(\Omega)}
<\Vert[  u_\eta]_\Omega - u_c\Vert_{L^2(\Omega)}, \end{aligned}
\end{equation*}
from which \eqref{eq:notati} follows.

\medskip
{\it{\uline{Step~3.}} We conclude the proof of Theorem~\ref{thm:onalpha}.
}

We first show  \eqref{eq:minIalpha}. Because  \(\widehat I\) is a lower-semicontinuous function on the
compact set \([0,+\infty]\), \(\widehat I\) attains a minimum on
\([0,+\infty]\). By \eqref{eq:Icont}, \eqref{eq:notat0}, and \eqref{eq:notati},
we conclude
that  \(\widehat I\) attains its minimum
at some 
\(\alpha^*\in(0,+\infty)\). Thus,  using \eqref{eq:Icont} once more,
\begin{equation}\label{eq:minI'}
\begin{aligned}
I(\alpha^*)=\widehat I(\alpha^*)=\min_{\bar\alpha \in [0,+\infty]} \widehat I(\bar
\alpha)=\min_{\alpha \in (0,+\infty)}
\widehat I(\alpha)=\min_{\alpha \in (0,+\infty)}
I(\alpha),
\end{aligned}
\end{equation}
which yields  \eqref{eq:minIalpha}.

Next, to prove the existence of \(c_\Omega\) as stated, assume that there exist
\((\alpha^*_j)_{j\in\NN}\subset(0,+\infty)\) such that \(\alpha^*_j\to0\)
and \eqref{eq:minI'} holds with \(\alpha^*=\alpha^*_j\). Then,
using the lower semi-continuity of \(\widehat I\) on \([0,+\infty]\),%
\begin{equation*}
\begin{aligned}
\min_{\bar\alpha \in [0,+\infty]}
\widehat I(\bar \alpha) \leq \widehat I(0) \leq \liminf_{j\to\infty}
\widehat I( \alpha^*_j) =\min_{\bar\alpha \in [0,+\infty]}
\widehat I(\bar \alpha),
\end{aligned}
\end{equation*}
which is false by  \eqref{eq:notat0}. This establishes the existence of the  constant \(c_\Omega\).
 
 On the other hand,   as mentioned in the proof of Proposition~\ref{prop:Ja257},
 \cite[Proposition~2.5.7]{Ja12} yields a positive constant, \(C_\Omega\),  such that \(u_\alpha\equiv [ u_\eta]_\Omega\) for all \(\alpha\geq
C_\Omega \Vert u_\eta\Vert_{L^2(\Omega)}\). This fact,  
\eqref{eq:notati}, and \eqref{eq:minI'} show that we must have 
\(  \alpha^*_\Omega < C_\Omega\Vert u_\eta\Vert_{L^2(\Omega)}\).
Finally, the  \(\Omega=L\)  case follows from Proposition~\ref{prop:Ja257}. 
\end{proof}

Next, we  prove Theorem~\ref{thm:equiv}.

\begin{proof}[Proof of  Theorem~\ref{thm:equiv}]
In view of Theorem~\ref{thm:solL1} (also see Remark~\ref{rmk:boxtresh}), 
the statement in (a) follows. Conversely, the statement in (b)
can be proved arguing as in  Subsection~\ref{sect:TVL3}
and defining
\[c_0:= \min\bigg\{\min_{L\in \scrL \in \bar\scrP }
c_L, \big(c_Q\Vert u_\eta\Vert_{L^2(Q)}^2\big)^{-1}\bigg\},\]   where \(c_L\) and \(C_Q\) are the constants given by Theorem~\ref{thm:onalpha}.
\end{proof}

We conclude this section with some  examples of 
 stopping criteria for
the refinement of the admissible partitions as  defined in Definition~\ref{def:stop}.   

\begin{example}\label{ex:localsubdiv} Here, we give an example of a 
 stopping criterion that, heuristically, means that we only refine
a given dyadic square \(L\), if the distance of the restored image in
\(L\) to the clean image is greater than or equal to the sum of the distances of the restored
images in each of the
 subdivisions of \(L\) to the clean image, modulo a threshold
that is  determined by the user.

To make this idea precise, we introduce some
notation. Given  a dyadic square \(
L^{(1)}\subset Q\) of  side length
\(\frac1{2^{k+1}}\), we can find three other  dyadic squares,
which we denote by \( 
L^{(2)}\), \(L^{(3)}\),  and \(L^{(4)}\), of  side length
\(\frac1{2^{k+1}}\) and such that \(L:= \medcup_{i=1}^4
  L^{(i)}\) is a dyadic square of side length
\(\frac1{2^k}\). We observe further that \( 
L^{(2)}\), \( L^{(3)}\), and  \(L^{(4)}\) are uniquely determined
by the requirement that \(L\) is a dyadic square.
Using this notation, and setting $u_L=u_{\alpha_L}$ (see \eqref{eq:alpha-L}), we fix \(\delta>0\) and  set up an admissible criteria
as follows:
\begin{align}
(\scrS)\quad  (i)\enspace &\text{\(Q\) is admissible;}\notag\\
 (ii)\enspace &\text{If \(L\subset Q\) is an admissible dyadic square, then
each dyadic square \(L^{(i)}\subset L\),} \notag\\ &\text{with \(i\in\{1,...,4\}\) and \(\medcup_{i=1}^4
L^{(i)} = L\), is admissible if }\notag\\
& \hskip30mm \Vert u_c - u_{  L}\Vert_{L^2( 
L)}^2 \geq \sum_{i=1}^4 \Vert u_c - u_{L^{(i)}}\Vert_{L^2( 
L^{(i)})}^2  + \delta. \label{eq:scragf}
\end{align}

As we prove next, %
\begin{equation*}
\begin{aligned}
\bar\scrP:=\big\{\scrL\in\scrP\!:\, L \text{ satisfies \((\scrS)\)
for all \(L\in\scrL\)} \big\}
\end{aligned}
\end{equation*}
has finite cardinality, which shows that  \((\scrS)\) as above provides a
 stopping criteria  for
the refinement of the admissible partition.

To show that \(\bar \scrP\) has finite cardinality, we first
observe that if \(L\) satisfies \((\scrS)\), then we can find
\(k\) dyadic squares, \(L_1, ..., L_k \), where \(k\in\NN\) is
such that \(|L|=\frac1{4^{k}}\), satisfying
\begin{equation*}
\begin{aligned}
Q=L_1\supset .... \supset L_k \supset L, \qquad 
|L_k|=\frac1{4^{k-1}}, \qquad L_k \text{ satisfies }(\scrS).
\end{aligned}
\end{equation*}
Then, using \eqref{eq:scragf}, we conclude that 
\begin{equation*}
\begin{aligned}
\Vert u_c - u_{Q}\Vert_{L^2(Q)}^2 = \Vert u_c - u_{ L_1}\Vert_{L^2(
L_1)}^2 \geq c_k + k\delta
\end{aligned}
\end{equation*}
for some positive constant \(c_k\), which can only  hold true if  \(k\) is
small enough. In other words, there  exists \(k_\delta
\in\NN\) such that if  \(L\) satisfies \((\scrS)\), then  \(|L|\geq\frac1{4^{k_\delta}}\).
Hence,  \(\bar \scrP\) has finite cardinality. 
 \end{example}

\section{Analysis of the Regularized Weighted-TV and   Weighted-Fidelity
learning schemes \texorpdfstring{\((\scrL\!\scrS)_{{TV\!}_{\omega_\epsi}}\)}{TV-we} and    
 \texorpdfstring{\((\scrL\!\scrS)_{{TV-Fid}_\omega}\)}{TV-Fid-w}}
\label{sect:regTV}

The results proved in the preceding section for the weighted-TV learning scheme can be easily adapted to the case of the regularized weighted-TV and   the weighted-fidelity learning schemes,  \((\scrL\!\scrS)_{{TV\!}_{\omega_\epsi}}\) and    
\((\scrL\!\scrS)_{{TV-Fid}_\omega}\). For the former, we prove here only  Proposition~\ref{thm:Gconvergence} and  provide
an example of  a sequence of regularized weights satisfying 
the conditions assumed in this result. Moreover, we highlight a question that is intimately related to the convergence of the solutions to \((\scrL\!\scrS)_{{TV\!}_{\omega_\epsi}}\) as \(\epsi\to0^+\) (see Subsection~\ref{subs:reg} below).  Regarding
    \((\scrL\!\scrS)_{{TV-Fid}_\omega}\), and for completeness, we state the analogue existence and equivalence results for the weighted-fidelity
learning scheme (see Subsection~\ref{subs:wfid} below).
    
\subsection{The \texorpdfstring{\((\scrL\!\scrS)_{{TV\!}_{\omega_\epsi}}\)}{TV-we} learning scheme}\label{subs:reg}

Next, we prove Proposition~\ref{thm:Gconvergence} and  provide
an example of   a sequence \((\omega_\scrL^\epsi)_\epsi \)
as in \eqref{eq:regweight}.

\begin{proof}[Proof of Proposition~\ref{thm:Gconvergence}]
We  show that
\begin{equation}\label{eq:EepsileqE}
E_\scrL^\epsi[u] \leq E_\scrL[u]
\end{equation}
for all \(u\in L^1(Q)\), from which  \eqref{eq:Glimsup} 
 follows.

Let  \(u\in L^1(Q)\) be such that \(E_\scrL[u]< \infty\).  Then,
\(u\in BV_{\omega_\scrL}(Q)\) and recalling the definition
and properties of the space of weighted \(BV\)-function discussed
in Section~\ref{sect:L2}, we have that \(u\in BV_{\omega_\scrL^\epsi}(Q)\)
with \(TV_{\omega_\scrL^\epsi}(u,Q)
\leq   TV_{\omega_\scrL}(u,Q)\), using the estimate \(\omega_\scrL^\epsi \leq {\omega_\scrL}\)
a.e.~in \(Q\) in \eqref{eq:regweight}. Thus,  \eqref{eq:EepsileqE}
holds.
\end{proof}

\begin{example}\label{eg:onregweights}
An example of a sequence \((\omega_\scrL^\epsi)_\epsi \)
as in \eqref{eq:regweight} can be constructed combining a diagonalization
argument with a mollification of a Moreau--Yosida type approximation of
\(\omega^{sc^-}_\scrL\). Precisely, for each 
 \(k\in\NN\),
 let \(\omega_k:Q\to (0,\infty)\)   be given by
\begin{equation}\label{eq:defwk}
\begin{aligned}
\omega_k(x):=\inf \big\{\omega^{sc^-}_\scrL(y) + k|x-y|\!: \, y\in Q\big\}\quad \text{for
\(x\in Q\)}.
\end{aligned}
\end{equation}
We recall that each \(\omega_k\) is a \(k\)-Lipschitz function,
and we have (see \cite[Theorem~2.1.2]{Ca08} for instance)
\begin{equation}
\label{eq:monotwk}
\begin{aligned}
\omega_k\nearrow \omega^{sc^-}_\scrL \quad \text{pointwise everywhere
in \(Q\)}.
\end{aligned}
\end{equation}

Moreover, as we show next,
\begin{equation}\label{eq:cuMY}
\begin{aligned}
&\lim_{k\to\infty} \Vert\omega_k -\omega^{sc^-}_\scrL\Vert_{L^\infty(K)} = 0
\end{aligned}
\end{equation}
for any compact set \(K\) such that \(K\subset \inte(L)\), where \(L\in\scrL \) is arbitrary. 

In fact, let  \(L\in\scrL
\) and let   \(K\)  be a compact set such that \(K\subset \inte(L)\). Fix \(\tau>0\) and  set \(\delta:=\frac{\dist(K,\partial L)}{2}\). Note that $\delta>0$ and 
 \begin{equation}\label{eq:onwrlx}
 \omega^{sc^-}_\scrL(x) = \alpha_L\quad \text{for all \(x\in \inte(L)\)}
 \end{equation}
 because \(\omega_\scrL(x) = \alpha_L\) for all \(x\in L\). Moreover, using \eqref{eq:defwk}, given \(\bar x\in K\) we can find \( y_k\in Q\) such that
 \begin{equation}\label{eq:estYM}
 \omega_k(\bar x) + \tau \geq \omega^{sc^-}_\scrL( y_k) +\ k|\bar x-y_k|.
 \end{equation}
Hence, using \eqref{eq:monotwk} and nonnegativity of $\omega^{sc^-}$, we obtain
 \begin{equation*}
|\bar x-y_k|\leq \frac{\omega_k(\bar x) + \tau - \omega^{sc^-}_\scrL(y_k) }{k} \leq \frac{\Vert \omega^{sc^-}_\scrL\Vert_{L^\infty(Q)} + \tau}{k} <\delta    \end{equation*} 
 for all \(k\geq k_0\) and for some \(k_0\in\NN\) that is independent of \(\bar x\). Then, \(y_k\in \inte(L)\)  for all  \(k\geq k_0\). Consequently,   \eqref{eq:onwrlx}--\eqref{eq:estYM} then yield 
\begin{equation*}
\omega_k(\bar x) + \tau \geq \omega^{sc^-}_\scrL( y_k) = \alpha_L =\omega^{sc^-}_\scrL(\bar x) 
\end{equation*}
for all \(k\geq k_0\). Hence,
\begin{equation*}
0\leq \omega^{sc^-}_\scrL(\bar
x)-  \omega_k(\bar x) \leq  \tau 
\end{equation*}
for all \(k\geq k_0\). Taking the supremum on \(\bar x \in K\) in the preceding estimate yields \eqref{eq:cuMY}.

On the other hand, for each \(k\in\NN\), a standard mollification argument yields a sequence \((\omega^{(k)}_\epsi)_\epsi \subset C^\infty(\overline Q)\) such that
\begin{equation}\label{eq:cuMY1}
\begin{aligned}
&\lim_{\epsi\to0^+} \Vert \omega^{(k)}_\epsi -\omega_k\Vert_{L^\infty(Q)} = 0.
\end{aligned}
\end{equation}

Finally, denoting by $Q(x,\delta)$ the open square centered at $x\in\RR^2$ and side-length $\delta$, we  can write \(\scrL=\medcup_{i=1}^\ell
L_i\)  with  \(\inte (L_i)=Q(x_i,\delta_i)\), for  some  \(\ell\in\NN\), \(x_i\in L_i\), and  \(\delta_i >0\). Then, exploiting the countability of the family
\begin{equation}\label{eq:famQ}
\begin{aligned}
\mathcal{K}:=\medcup_{i=1}^\ell \{K_i:=\overline {Q(x_i, r_i)}\!: \, r_i \in \mathbb{Q}\cap (0,\delta_i)\}
\end{aligned}
\end{equation} 
and a diagonalization argument together with  \eqref{eq:cuMY} and \eqref{eq:cuMY1}, we can find a sequence
\((\omega_\scrL^\epsi)_\epsi \) such that
\begin{equation}\label{eq:cuMY2}
\begin{aligned}
&\lim_{\epsi\to 0^+} \Vert\omega_\scrL^\epsi -\omega^{sc^-}_\scrL\Vert_{L^\infty(K)} = 0
\end{aligned}
\end{equation}
for all compact set \(K\in \mathcal{K}\). From the definition
of \(\mathcal{K}\) in \eqref{eq:famQ}, we get that \eqref{eq:cuMY2} also
holds for all compact set \(K\subset \inte(L)\)  and for any   \(L\in\scrL
\). Furthermore, using the fact that mollification preserves
monotonicity, we deduce from \eqref{eq:monotwk} and \eqref{eq:cuMY1} that \(\omega_\scrL^\epsi \nearrow \omega^{sc^-}_\scrL\) everywhere
in \(Q\).

To conclude that \eqref{eq:regweight} also holds, it suffices to observe that \(\omega^{sc^-}_\scrL\leq \omega_\scrL\) in $Q$, \(\omega^{sc^-}_\scrL\equiv \omega_\scrL\) in $\medcup_{i=1}^\ell
\inte (L_i)$, $|Q\setminus \medcup_{i=1}^\ell
\inte (L_i)| =0$.

\end{example}

\begin{remark}[Existence of solutions to the learning scheme
 \((\scrL\!\scrS)_{{TV\!}_{\omega_\epsi}}\)]\label{rmk:solepsi}
For fixed \(\epsi\), we can apply the results proved in Section~\ref{sect:wTV}.
In particular, there exists an  optimal solution
\(u^*_\epsi\) to the learning scheme \((\scrL\!\scrS)_{{TV\!}_{\omega_\epsi}}\)
in \eqref{lsTVomegae} with  \eqref{eq:alpha-LL} 
replaced by \eqref{eq:box-min} (cf.~Theorem~\ref{thm:TVomega}).
 
\end{remark}

\begin{remark}\label{Q:opnebp} An interesting question is whether condition
 \eqref{eq:regweight} yields the   convergence 
\begin{equation}
\label{eq:openpb}
\begin{aligned}
\lim_{\epsi\to0^+} TV_{\omega_\scrL^\epsi}(u,Q) =    TV_{\omega_\scrL}(u,Q)
\end{aligned}
\end{equation}
for all 
\(u\in BV_{\omega_\scrL}(Q)\). 
Because sets of zero Lebesgue measure may  not have zero $|Du|$ measure, we do not expect \eqref{eq:openpb} to hold unless the almost everywhere pointwise convergence in \eqref{eq:regweight} is replaced by everywhere pointwise convergence.

To the best of our knowledge,
the closest result in this direction is \cite[Lemma~2.1.4]{Ca08},
which shows the following. If \(\tilde \omega\geq 0\) is lower semi-continuous
in \(Q\) and \(u:Q\to\RR\) is measurable, then we can find a
sequence of Lipschitz weights, \((\tilde \omega_k^{(u)})_{k\in\NN}\), depending
on \(u\), such 
that \(\tilde \omega^{(u)}_k\nearrow\tilde \omega\) pointwise everywhere
in \(Q\) and \eqref{eq:openpb} holds (with \(\omega_\scrL^\epsi\)
and \(\omega_\scrL\) replaced by \(\tilde \omega^{(u)}_k\) and
\(\tilde \omega\), respectively). 
\end{remark}

\subsection{The \texorpdfstring{\((\scrL\!\scrS)_{{TV-Fid}_\omega}\)}{TV-Fid-w} learning scheme}\label{subs:wfid}

Given a dyadic square \(L\subset Q\) and \(\alpha\in(0,\infty)\), we have 
\begin{equation*}
\begin{aligned}
 &\argmin\left\{\frac1\alpha\int_{L}|u_\eta-u|^2\dd x+ TV(u,L)\!:\,u\in
BV(L)\right\}\\ &\qquad = \argmin\left\{\int_{L}|u_\eta-u|^2\dd x+ \alpha TV(u,L)\!:\,u\in
BV(L)\right\}.
\end{aligned}
\end{equation*}
Consequently, Proposition~\ref{prop:Ja257} and Theorem~\ref{thm:onalpha} remain unchanged if we replace \eqref{eq:ROF} by \eqref{eq:ROFf}. These two results are the main tools to prove   Theorems~\ref{thm:equiv} and \ref{thm:TVomega}. Using this observation, the arguments used in Section~\ref{sect:wTV} can be reproduced here for the weighted-fidelity
learning scheme to conclude the two following theorems. 
 
\begin{theorem}[Existence of solutions to \((\scrL\!\scrS)_{{TV-Fid}_\omega}\)]\label{thm:FIDomega}
There exists an  optimal solution
\(u^*\) to the learning scheme \((\scrL\!\scrS)_{{TV-Fid}_\omega}\)
in \eqref{lsFido} with  \eqref{eq:alpha-Lf} 
replaced by \eqref{eq:box-min}.
\end{theorem}

As before, the previous existence theorem  holds true under
any stopping criterion   for
the refinement of the admissible partitions
provided that the training data satisfies suitable conditions, as stated in the next result.

\begin{theorem}[Equivalence between box constraint and
stopping criterion]\label{thm:equivF}
Consider the learning scheme  \((\scrL\!\scrS)_{{TV-Fid}_\omega}\)
in \eqref{lsFido}. The two following conditions hold:\begin{itemize}
\item[(a)] If we replace    \eqref{eq:alpha-Lf} 
 by \eqref{eq:box-min}, then there exists a
stopping criterion \((\scrS)\)  for
the refinement of the admissible partitions as in Definition~\ref{def:stop}.

\item[(b)] Assume that
there exists a
stopping criterion \((\scrS)\)  for
the refinement of the admissible partitions as in Definition~\ref{def:stop}
such that the training data
  satisfies for all \(L\in\medcup_{\scrL\in\bar\scrP} \scrL\), with
  \(\bar\scrP\) as in Definition~\ref{def:stop}, the   conditions
\begin{itemize}[leftmargin=12mm]
\item[(i)] $TV(u_c,L) < TV(u_\eta,L)$;

\item[(ii)] $\displaystyle \Vert u_\eta - u_c\Vert^2_{L^2(L)}
<\Vert[  u_\eta]_L - u_c\Vert^2_{L^2(L)} $.
\end{itemize}
Then, there exists \(c_0\in\RR^+\) such that the optimal
solution \(u^*\) provided by  \((\scrL\!\scrS)_{{TV-Fid}_\omega}\)
with  \(\scrP\) replaced by  \(\bar \scrP\)
coincides with the optimal solution
\(u^*\) provided by   \((\scrL\!\scrS)_{{TV-Fid}_\omega}\)
with     \eqref{eq:alpha-Lf} 
 replaced by \eqref{eq:box-min}.
   \end{itemize}

\end{theorem}


\section{Analysis of the Weighted-TGV learning scheme 
\texorpdfstring{\((\scrL\!\scrS)_{{TGV\!}_\omega}\)}{TGV-w}}
\label{sect:wTGV}
This section is devoted to proving the existence of solutions to the training scheme \((\scrL\!\scrS)_{{TGV\!}_\omega}\) described in \eqref{lsTGVomega}. 
We begin by providing the precise definition of the quantities $\scrV_{\omega_{\mathscr{L}}^0}$ and $\scrV_{\omega_{\mathscr{L}}^1}$ in \eqref{eq:TGVomega-def}, which are particular instances of the general definition of the  weighted variation of a Radon measure introduced in Section~\ref{sect:glo} (see \eqref{eq:VarWeight}). 
\begin{definition}
\label{def:omegaBD}
Let $\Omega$ be an open set in $\RR^n$ and  $\omega:\Omega\to [0,+\infty)$ a locally integrable function. Given $u\in L^1_{\omega,\rm loc}(\Omega)$ and $v\in L^1_{\omega,\rm loc}(\Omega;\RR^n)$ (see \eqref{eq:L1wloc}), we set
\begin{equation}
\label{eq:varBVvw}
\begin{aligned}
\scrV_{\omega}(Du-v,\Omega):=\sup\left\{\int_\Omega (u\, {\mathrm{div}}\,\varphi+v\cdot\varphi)\dd
x:\,\varphi\in {\mathrm{Lip}}_c(\Omega;\RR^n),\,|\varphi|\leq \omega\right\}
\end{aligned}
\end{equation}
and
\begin{equation}
\label{eq:varBDw}
\begin{aligned}
\scrV_{\omega}(\Ecal v,\Omega):=\sup\left\{\int_\Omega (v\cdot {\mathrm{div}}\,
\xi)\dd x:\,\xi\in {\mathrm{Lip}}_c(\Omega;\mathbb{R}^{n\times n}_{\textrm
sym}),\,|\xi|\leq \omega\right\},
\end{aligned}
\end{equation}
where $({\mathrm{div}}\, \xi)_j = \sum_{k=1}^n \frac{\partial \xi_{jk}}{\partial x_k}$  for each $j\in\{1,...,n\}$.

\end{definition}

\begin{remark}\label{rmk:onVar}
Recalling  \eqref{eq:VarWeight}, we are using an abuse of notation in the preceding definition
 as we are not requiring \(Du\) nor \(\Ecal v\) to be Radon measures. However, if \(u\in BV(\Omega) \), then \eqref{eq:varBVvw} is the \(\omega\)-weighted variation of the Radon measure \(Du-v:= Du-v\Lcal^n\lfloor\Omega
\in \Mcal(\Omega;\RR^n)\) in the sense of \eqref{eq:VarWeight}. Similarly, if \(v\in BD(\Omega)\), then \eqref{eq:varBDw} is 
the \(\omega\)-weighted variation of the
Radon measure \(\Ecal v\in \Mcal(\Omega;\mathbb{R}^{n\times n}_{\textrm
sym})\) in the sense of \eqref{eq:VarWeight}. 

\end{remark}

Analogously to the   \((\scrL\!\scrS)_{{TV\!}_\omega}\) case, we analyze each level of \((\scrL\!\scrS)_{{TGV\!}_\omega}\) in a dedicated subsection. 

To prove existence of a solution to the learning scheme \((\scrL\!\scrS)_{{TGV\!}_\omega}\) in \eqref{lsTGVomega}, we  argue by a box-constraint approach
in which we replace the requirement \(\alpha=(\alpha_0,\alpha_1)\in\RR^+\times \RR^+\) by the stricter condition \eqref{eq:bc-TGV}. 
In this case, we replace  \eqref{eq:alpha-Lg-TGV} 
by
\begin{equation}
\label{eq:box-min-TGV}
\bar \alpha_L=\inf\left\{\argmin\left\{\int_L |u_c-u_{\alpha,L}|^2\dd x:\,\alpha\in
\big[c_0,\tfrac{1}{c_0}\big] \times \big[c_1,\tfrac{1}{c_1}\big]\right\}\right\}.
\end{equation}

Throughout this section, for $u\in L^2(\Omega)$, we denote by $\langle u\rangle_\Omega$ the affine projection of $u$ given  by the unique solution to the minimization problem
\begin{equation}
\label{eq:profAff}
\begin{aligned}
\min\left\{\int_\Omega |u-v|^2\dd x:\,v\text{ is affine in }\Omega\right\},
\end{aligned}
\end{equation}
which will play an analogous role to the average $[u]_\Omega$ in the $TV$ case treated in Section \ref{sect:wTV}. Note that we have the orthogonality property
\begin{equation}
\label{eq:affine-proj}
\int_{\Omega} (u-\langle u\rangle_\Omega)\langle u\rangle_\Omega \dd x=0
\end{equation}
for every $u\in L^2(\Omega)$, since $\langle u\rangle_\Omega$ is the Hilbert projection of $u$ onto a finite dimensional subspace of $L^2(\Omega)$.

\subsection{On Level 3}\label{sect:TGVL3}
We provide here an analysis of Level 3, and minor variants thereof,  of the learning scheme 
\((\scrL\!\scrS)_{{TGV\!}_\omega}\)  in \eqref{lsTGVomega}.

As in the weighted $TV$ scheme case, the parameter \(\alpha_L\) in Level 3 of \((\scrL\!\scrS)_{{TGV\!}_\omega}\) (see \eqref{eq:alpha-Lg-TGV}) is uniquely determined by definition, and it satisfies \(\alpha_L\in [0,+\infty]^2\). In view of Theorem~\ref{thm:onalpha-TGV} (see Subsection~\ref{subs:stop-TGV}), if \(L\in\scrL\) is such that
\begin{equation}
\begin{aligned}\label{eq:dataonLTGV}
TGV_{\hat{\alpha}_0,\hat{\alpha}_1}(u_c,L) < TGV_{\hat{\alpha}_0,\hat{\alpha}_1}(u_\eta,L) \quad \text{ and } \quad \Vert u_\eta - u_c\Vert^2_{L^2(L)}
<\Vert\langle  u_\eta\rangle_L - u_c\Vert^2_{L^2(L)}
\end{aligned}
\end{equation}
 for some
$\hat{\alpha} = (\hat{\alpha}_0,\hat{\alpha}_1)$, then
\begin{equation*}
\begin{aligned}
\arginf &\left\{\int_L |u_c-u_{\alpha,L}|^2\,\dd x\!:\,\alpha\in
\RR^+ \times \RR^+\right\} = \\ &\argmin \left\{\int_L |u_c-u_{\alpha,L}|^2\,\dd x\!:\, \alpha\in
\RR^+ \times \RR^+ \text{ is s.t. } c_L\leq \min\{\alpha_0,\alpha_1\} < C_Q\Vert u_\eta\Vert_{L^2(L)} \right\},
\end{aligned}
\end{equation*}
where \(c_L\) and \(C_Q\) are positive constants, with \(c_Q\)  depending only on \(Q\). 
Furthermore, because each partition
$\mathscr{L\in \scrP}$ is finite, it follows that if \eqref{eq:dataonLTGV} holds for all \(L\in \scrL\), then 
\begin{equation}\label{eq:optcases-TGV}
\displaystyle\alpha_L\in \Big[\min_{L\in\scrL} c_L, +\infty\Big]\times \Big[\min_{L\in\scrL} c_L, +\infty\Big] \setminus \{(+\infty,+ \infty)\}.
\end{equation}
Moreover, if we consider Level~3 with \eqref{eq:alpha-Lg-TGV} 
replaced by \eqref{eq:box-min-TGV}, then the minimum
\begin{equation*}
\begin{aligned}
\min_{\alpha\in   [c_0,\frac{1}{c_0}] \times [c_1,\frac{1}{c_1}]} \int_L |u_c-u_{\alpha,L}|^2\,\dd x
\end{aligned}
\end{equation*}
exists  as the minimum of a lower semicontinuous function (see Lemma~\ref{lem:Ilsc-TGV} in Subsection~\ref{subs:stop-TGV}) on a compact set. In particular, \(\bar \alpha_L\) in \eqref{eq:box-min-TGV} is uniquely determined and belongs to the set in \eqref{eq:bc-TGV}.

\subsection{On Level 2}
\label{subs: Level 2-TGV}
In this subsection, we discuss the existence of solutions to \eqref{eq:minprS2g}. In what follows, let $\Omega\subset \RR^n$ be an open set and  $\omega:\Omega\to [0,\infty)$  a locally integrable function. Recalling the definition of $L^1_{\omega, {\rm loc}}(\Omega)$ and $\|\cdot\|_{L^1_\omega(\Omega)}$ in Subsection \ref{sect:L2}, as well as \eqref{eq:varBDw},  we define the space $BD_\omega(\Omega)$ of \(\omega\)-weighted $BD$ functions in $\Omega$ by
\begin{equation*}
BD_\omega(\Omega):=\big\{v\in L^1_{\omega}(\Omega;\RR^n):\,\scrV_{\omega}(\Ecal v,\Omega)<\infty \big\},
\end{equation*}
and we endow it with the semi-norm 
\begin{equation*}
\|v\|_{BD_\omega(\Omega)}:=\|v\|_{L^1_\omega(\Omega;\RR^n)}+\scrV_\omega(\Ecal v,\Omega).
\end{equation*}
Note that if $\essinf_\Omega\omega>0$, the semi-norm above is actually a norm, and that \(BD_\omega\) with $\omega\equiv 1$  coincides with the classical space of functions with bounded deformation, cf.~\cite{Te83} for instance. The instrumental properties of $BD_\omega$   for our analysis are collected in the ensuing result.

\begin{theorem}
\label{thm:ppBVw-tgv}
Let \(\Omega\subset \RR^n\) be an open set and   
\(\omega:\Omega\to [0,\infty)\)    a  locally integrable function. Then, the following statements hold:
\begin{itemize}
\item[(i)] If \(\inf_\Omega\omega>0\), then the map \(v\mapsto \scrV_{\omega}(\Ecal v,\Omega)\) is lower-semicontinuous with respect to the (strong) convergence in \(L^1_{\omega,\loc}(\Omega;\R^n)\).

\item[(ii)] Given  \(v\in
L^1_{\omega,\loc}(\Omega;\RR^n)\), we have \(\scrV_{\omega}(\Ecal v,\Omega)=\scrV_{\omega^{sc^-}}(\Ecal v,\Omega)\), where  \(\omega^{sc^-}\) denotes the lower-semicontinuous envelope
of $\omega$.

\item[(iii)] Assume \(\omega\) is lower-semicontinuous
and strictly positive. Then, we have  \(v\in
L^1_{\loc}(\Omega;\RR^n)\) and \(\scrV_{\omega}(\Ecal v,\Omega)<\infty\) if and only if \(v\in BD_{\rm loc}(\Omega)\) and \(\omega\in L^1(\Omega;\vert \Ecal v\vert)\). If any of these two equivalent conditions hold, we have
\begin{equation*}
\begin{aligned}
\scrV_{\omega}(\Ecal v,B)= \int_B \omega(x)\dd|\Ecal v|(x) 
\end{aligned}
\end{equation*}
for every Borel set \(B\subset \Omega\).
\item[(iv)] If \(\omega \in L^\infty_{\loc}(\Omega)\) is  lower-semicontinuous and strictly positive, then all bounded sequences in $BD_{\omega}(\Omega)$ are precompact in the strong $L^1_{\omega,{\rm loc}}$-topology. 
\end{itemize}
\end{theorem}
\begin{proof}
Accounting for the fact that test functions here take values in $\mathbb{R}^{n\times
n}_{\rm sym}$, the proof of \((i)\), \((ii)\), and \((iii)\) may be obtained by mimicking that of 
\cite[Proposition 1.3.1]{Ca08},  \cite[Proposition 2.1.1]{Ca08}, and  \cite[Theorem~2.1.5]{Ca08}, respectively.

 To prove $(iv)$, we observe that for each compact set \(K\subset \Omega\), there exists a positive constant \(c_K\) such that 0<  $\tfrac1{c_K} \leq \omega\leq c_K$ in $K$ because \(\omega \in L^\infty_{\loc}(\Omega)\) and strictly positive lower-semicontinuous
functions are locally bounded away from zero.   Then, using  $(iii)$, we have    for every $v\in BD_\omega(\Omega)$ that
\begin{align*}
&\scrV_{\omega}(\mathcal{E}v,K)=\int_K \omega(x)\dd|\mathcal{E}v|(x) 
\begin{cases}
\geq \tfrac1{c_K}  |\mathcal{E}v|(K),\\
\leq c_K |\mathcal{E}v|(K),
\end{cases}\\
&\|v\|_{L^1_{\omega}(K;\RR^n)}=\int_\Omega |v(x)|\,\omega(x)\dd x
\begin{cases}
\geq \tfrac1{c_K}  \|v\|_{L^1(K;\mathbb{R}^n)},\\
\leq c_K \|v\|_{L^1(K;\mathbb{R}^n)}.
\end{cases}
\end{align*}
The preceding estimates and    the compact embedding of $BD(K)$ into $L^1(K;\mathbb{R}^n)$ (cf. \cite{Te83}) yield \((iv)\).
\end{proof}

\begin{remark}\label{rmk:oniv} If \(\omega:\Omega\to (0,\infty)\) is a lower-semicontinuous function satisfying \(0<c\leq \inf_\Omega \omega \leq \sup_\Omega \omega\leq  c^{-1}\) for some positive constant \(c\), then the  arguments in the preceding proof show that Theorem~\ref{thm:ppBVw-tgv}~\((iv)\) holds globally in \(\Omega\). In other words,  bounded
sequences in $BD_{\omega}(\Omega)$ are precompact in the strong $L^1_{\omega}(\Omega;\RR^n)$-topology. \end{remark}

\begin{remark}
Differently from the  weighted-TV case  (cf.~Theorem~\ref{thm:ppBVw}), we need the weights $\omega$  in Theorem \ref{thm:ppBVw-tgv}   to be bounded from below away from zero for item $(i)$ to hold. This is because one cannot resort
to arguments based on coarea formulas in the  symmetrized gradient case, which prevents us to adapt  the arguments in \cite[Remark 1.3.2 and Theorem 3.1.13]{Ca08} to this framework.
\end{remark}
The next result collects some basic properties of the quantity $\scrV_{\omega}(Du - v,\Omega)$ given by \eqref{eq:varBVvw}.

\begin{theorem}
\label{lem:tv-cont}
Let \(\Omega\subset \RR^n\) be an open set and   
\(\omega:\Omega\to [0,\infty)\)    a  locally integrable function. Let  $u\in BV_\omega(\Omega)$. Then, the following statements hold:
\begin{itemize}
\item[(i)] The map $v\to \scrV_{\omega}(Du-v,\Omega)$ is lower semicontinuous with respect to the strong convergence in $L^1_{\omega,{\rm loc}} (\Omega;\RR^n)$.
\item[(ii)] Given  \(v\in
L^1_{\omega,\loc}(\Omega;\RR^n)\), we have \(\scrV_{\omega}(Du- v,\Omega)=\scrV_{\omega^{sc^-}}(Du- v,\Omega)\), where  \(\omega^{sc^-}\) denotes the lower-semicontinuous envelope
of $\omega$.
\item[(iii)] If \(v\in L^1_{\omega,{\rm loc}} (\Omega;\RR^n) \) and $\omega\in L^1(\Omega; |Du-v|)$ is  lower-semicontinuous
and strictly positive, then
\begin{equation}
\label{eq:intrepDu-v}
\begin{aligned}
\scrV_\omega(Du-v,B)=\int_B \omega(x)\dd|Du-v|(x)
\end{aligned}
\end{equation}
for every Borel set $B\subset\Omega$.
\end{itemize}
\end{theorem}
\begin{proof}
To prove $(i)$, let $(v_k)_{k\in \N}\subset L^1_{\omega, {\rm loc}}(\Omega;\RR^n)$ be a sequence such that $v_k\to v$ strongly in $L^1_{\omega,{\rm loc}}(\Omega;\R^n)$. Then, by Definition \ref{def:omegaBD}, 
$$\scrV_{\omega}(Du-v_k,\Omega)\geq \int_\Omega \big(u\,{\rm div}\,\varphi+v_k \cdot \varphi\big) \dd{x}$$
for every $\varphi\in {\rm Lip}_c(\Omega;\R^n)$ with $|\varphi|\leq \omega$ in $\Omega$. Moreover, for all such \(\varphi\),
\begin{equation*}
\begin{aligned}
\int_\Omega |v_k - v||\varphi|\dd x\leq \int_{\supp\varphi} |v_k - v|\,\omega \dd x \to 0 \text{ as } k\to+\infty.
\end{aligned}
\end{equation*}
Hence, 
$$\liminf_{k\to +\infty}\scrV_{\omega}(Du-v_k,\Omega)\geq \int_\Omega \big(u\,{\rm div}\,\varphi+v\cdot \varphi\big) \dd{x},$$
from which the conclusion follows by taking the supremum over all  test functions $\varphi\in {\rm Lip}_c(\Omega;\R^n)$ with $|\varphi|\leq \omega$ in $\Omega$. 
 
The proof of $(ii)$ follows  by Definition \ref{def:omegaBD}, observing that every map $\varphi\in {\rm Lip}_c(\Omega;\R^n)$ with $|\varphi|\leq \omega$ in $\Omega$ also satisfies $|\varphi|\leq (\omega^{\rm sc})^-$ in $\Omega$. 

As we discuss next, the proof of $(iii)$ is an adaptation of \cite[Theorem 2.1.5]{Ca08}. In fact, because $u\in BV_\omega(\Omega)$
and  strictly positive lower-semicontinuous
functions are locally bounded away from zero, we have \(u\in
BV_{\rm loc}(\Omega)\). Then,   for every  $\varphi\in {\rm Lip}_c(\Omega;\R^n)$
with $|\varphi|\leq \omega$ in $\Omega$, we have that $$\int_\Omega \big(u\,{\rm div}\,\varphi+v\cdot \varphi\big) \dd{x}\leq \int_\Omega \omega \dd|Du-v|; $$
 hence, $\scrV_\omega(Du-v,\Omega)\leq \int_\Omega \omega \dd|Du-v|$. 
Conversely, since $\omega\in L^1(\Omega;|Du-v|)$, we infer that
\begin{equation}
\label{eq:yyy}
\begin{aligned}
\int_{\Omega}\omega \dd|Du-v|=|\omega (Du-v)|(\Omega)=\sup\left\{\int_\Omega
\omega\,\psi\cdot \dd(Du-v):\,\psi\in {\rm Lip}_c(\Omega;\R^n),\,|\psi|\leq
1\right\}.
\end{aligned}
\end{equation}
Let \((\omega_k)_{k\in\NN}\) be an increasing sequence  of \(k\)-Lipschitz
functions converging to \(\omega\) in \(\Omega\) as in Example~\ref{eg:onregweights}
(see also  \cite[Theorem
2.1.2]{Ca08}). Then, for every $\psi\in {\rm Lip}_c(\Omega;\R^n)$
with $|\psi|\leq
1$ in $\Omega$, we have   $\omega_k \, \psi
 \in {\rm Lip}_c(\Omega;\R^n)$ with $|\omega_k \, \psi|\leq
\omega_k \leq \omega$ in $\Omega$; thus,  
  using the Lebesgue dominated convergence theorem and recalling
 \eqref{eq:varBVvw},
we find that
\begin{align*}\int_{\Omega} \omega\,\psi\cdot \dd(Du-v)&=\lim_{k\to\infty}
\int_{\Omega} \omega_k\,\psi\cdot \dd(Du-v)\\&=-\lim_{k\to\infty}\int_{\Omega} \big(u\,{\rm div}\,(\omega_k\,\psi)+\omega_k\,\psi\cdot v\big)\,\dd x\leq \scrV_\omega(Du-v,\Omega).\end{align*}
 From this estimate and  \eqref{eq:yyy}, we deduce that $ \int_\Omega \omega \dd|Du-v| \leq \scrV_\omega(Du-v,\Omega)$, which concludes
the proof of \eqref{eq:intrepDu-v} when \(B=\Omega\). The proof that this identity holds for every Borel set \(B\subset\Omega\) can be done exactly
as in 
\cite[Theorem 2.1.5]{Ca08}.
 \end{proof}

We proceed by showing that the infimum in  
\begin{equation}
\label{eqTGV-precise}
TGV_{\omega_0,\omega_1}(u,Q):=\inf_{v\in BD_{\omega_1}(Q)}\big\{\scrV_{\omega_0}(Du-v,Q)+{\scrV_{\omega_1}}(\Ecal
v,Q)\big\},
\end{equation}
where $\omega_0,\, \omega_1 :Q\to (0,+\infty)$ are bounded functions and \(u \in L^1_{\omega_0}(Q)\),  is actually a minimum, and that the contributions due to $\scrV_{\omega_0}$ and $\scrV_{\omega_1}$ can be expressed in a simplified way in terms of the lower semicontinuous envelopes of the weights $\omega_0$ and $\omega_1$. 
We begin with a technical lemma.

\begin{lemma}
\label{lem:bd-L1}
Let $c_0>0$ be a positive constant. For $i\in \{0,1\},$ let $\omega_i:Q\to (0,+\infty)$ be such that $c_0<\inf_Q \omega_i<\sup_Q \omega_i <\frac{1}{c_0}$, and  let $u\in L^1_{\omega_0,{\rm loc}}(Q)$. Then, for every $v\in L^1_{\omega_1}(Q;\R^n)$, we have
\begin{equation}
\label{eq:est-l1-w}
\|v\|_{L^1_{\omega_1}(Q;\RR^n)}\leq \frac{1}{c_0^2}
\big(\scrV_{\omega_0}(Du-v,Q)+TV_{\omega_0}(u,Q)\big).
\end{equation}
\end{lemma}
\begin{proof}
Fix $v\in L^1_{\omega_1}(Q;\R^2)$. Note that  the uniform bounds on \(\omega_1\) yield
\begin{equation}\label{eq:est-l1-w1}
\begin{aligned}
c_0\int_Q |v(x)| \dd x \leq  \int_Q \omega_1(x) |v(x)| \dd x = \|v\|_{L^1_{\omega_1}(Q;\RR^n)}   \leq \frac{1}{c_0}
\int_Q |v(x)| \dd x.
\end{aligned}
\end{equation}
In particular, 
  \(v \in L^1(Q;\RR^2)\); thus,%
\begin{equation}\label{eq:est-l1-w2}
\begin{aligned}
c_0\int_Q |v(x)| \dd x &= c_0\sup\bigg\{ \int_Q \psi(x) \cdot v(x)\dd x: \psi \in 
{\rm Lip}_{\rm c}(Q;\R^2), \,\, \Vert \psi\Vert_{L^\infty (Q;\RR^2)} \leq 1 \bigg\} \\
& =  \sup\bigg\{ \int_Q \tilde\psi(x) \cdot v(x)\dd x:
\tilde \psi \in 
{\rm Lip}_{\rm c}(Q;\R^2), \,\, \Vert\tilde \psi\Vert_{L^\infty (Q;\RR^2)} \leq
c_0 \bigg\} \\
& \leq \sup\bigg\{\int_Q \varphi(x) \cdot v(x)\dd x:\,\varphi\in
{\rm Lip}_{\rm c}(Q;\R^2),\,\,  |\varphi|\leq \omega_0\bigg\}\\
&\leq \scrV_{\omega_0}(Du-v,Q)+TV_{\omega_0}(u,Q),
\end{aligned}
\end{equation}
where we used Definition~\ref{def:omegaBD} together with the subadditivity of the supremum in the last estimate, and the bound \(c_0 \leq \inf_Q \omega_0 \) in the preceding one. We then obtain \eqref{eq:est-l1-w} by combining \eqref{eq:est-l1-w1} and \eqref{eq:est-l1-w2}.
\end{proof}

Under the same assumptions of Lemma \ref{lem:bd-L1}, the infimum problem in \eqref{eq:TGVomega-def} is actually a minimum.

\begin{lemma}
\label{lem:wTGV-min}
Let $c_0>0$ be a positive constant. For $i\in \{0,1\},$ let $\omega_i:Q\to
(0,+\infty)$ be such that $c_0<\inf_Q \omega_i<\sup_Q \omega_i <\frac{1}{c_0}$, and 
 let $u\in L^1(Q)$. Then, there exists $u^\ast\in BD_{\omega_1}(Q)$ such that
\begin{equation}
\label{eq:minTGV}
\begin{aligned}
TGV_{\omega_0,\omega_1}(u,Q)=\scrV_{\omega_0}(Du-u^\ast, Q)+\scrV_{\omega_1}(\mathcal{E}u^\ast,
Q).
\end{aligned}
\end{equation}
\end{lemma}
\begin{proof} We claim that  $TGV_{\omega_0,\omega_1}(u,Q)$ is finite if and only if $u\in
BV_{\omega_0}(Q)$. In fact, choosing $v=0$ as a competitor in \eqref{eqTGV-precise}, we infer that $TGV_{\omega_0,\omega_1}(u,Q)\leq TV_{\omega_0}(u,Q)$. On the other hand, recalling ~\eqref{eq:TVwei},
we have for any \( v\in BD_{\omega_1}(Q) \) that
\begin{align*}
TV_{\omega_0}(u,Q)&=\sup\left\{\int_Q (u \, {\mathrm{div}}\,\varphi+v \cdot \varphi
-v\cdot \varphi) \dd x:\,\varphi\in {\mathrm{Lip}}_c(\Omega;\mathbb{R}^2),\,|\varphi|\leq \omega_0\right\}\\
&\leq \scrV_{\omega_0}(Du-v, Q)+\|v\|_{L^1_{\omega_0}(Q;\mathbb{R}^2)}\\
&\leq \scrV_{\omega_0}(Du-v, Q)+\frac{1}{c_0^2}\|v\|_{L^1_{\omega_1}(Q;\mathbb{R}^2)},
\end{align*}
where  we used the subadditivity of the supremum combined with Definition \ref{def:omegaBD} in the first inequality, and the bounds on the two weights  in the second inequality. Thus, \(TV_{\omega_0}(u,Q) \leq \max\{1, c_0^{-2}\} \, TGV_{\omega_0,\omega_1}(u,Q)\), which concludes the proof of the claim.

To show \eqref{eq:minTGV}, we may assume without loss of generality that $TGV_{\omega_0,\omega_1}(u,Q)<\infty$,  in which case   $u\in BV_{\omega_0}(Q)$. Moreover, we may find a sequence   $(v_n) \subset BD_{\omega_1}(Q)$  such that
 \begin{equation}
 \label{eq:minimality-tgv}TGV_{\omega_0,\omega_1}(u,Q)=\lim_{n\to +\infty} \big\{\scrV_{\omega_0}(Du-v_n,Q)+\scrV_{\omega_1}(\Ecal v_n,Q)\big\} \leq C
 \end{equation}
 for some positive constant \(C\). From Lemma \ref{lem:bd-L1} and \eqref{eq:minimality-tgv} we infer  that \(\sup_{n\in \N}\|v_n\|_{BD_{\omega_1}(Q)}<+\infty\).
 Using the uniform bounds on \(\omega_1\), which are inherited by its lower semicontinuous envelope \((\omega_1)^{\mathrm{sc}^-}\),  and Theorem \ref{thm:ppBVw-tgv}~\((ii)\), also
\begin{equation*}
\begin{aligned}
\sup_{n\in \N}\|v_n\|_{BD_{(\omega_1)^{\mathrm{sc}^-}}(Q)}<+\infty.
\end{aligned}
\end{equation*}
 Moreover, by Theorem~\ref{thm:ppBVw-tgv}~\((i)\), \((ii)\), and \((iv)\) (also see Remark~\ref{rmk:oniv}),   there exists $u^\ast\in BD_{\omega_1}(Q) \cap BD_{(\omega_1)^{\mathrm{sc}^-}}(Q)$ such that
\begin{equation}\label{eq:xxx}
 \begin{aligned}
& v_n\to u^\ast\quad\text{strongly in }L^1_{(\omega_1)^{\mathrm{sc}^-}}(Q;\RR^2),\\
& \scrV_{\omega_1}(\Ecal u^*,\Omega)= \scrV_{(\omega_1)^{\mathrm{sc}^-}}(\Ecal u^*,\Omega)\leq \liminf_{n\to +\infty} \scrV_{(\omega_1)^{\mathrm{sc}^-}} (\Ecal v_n,\Omega) = \liminf_{n\to +\infty}\scrV_{\omega_1}(\Ecal v_n,\Omega).
 \end{aligned}
\end{equation}
 Using the uniform bounds on both weights once more, we also have \(v_n \to u^* \) strongly in $L^1_{\omega_0}(Q;\RR^2)$. The minimality of $u^\ast$ is then a direct consequence of Theorem~\ref{lem:tv-cont}~\((i)\), \eqref{eq:xxx},   and  \eqref{eq:minimality-tgv}.
\end{proof} 
The next result provides a characterization of the infimum problem in Level 2 of our learning scheme.
\begin{proposition}
\label{prop:L2tgv}
Let $\phi\in L^2(Q)$, and let $c_0>0$ be a positive constant.  For $i\in\{0,1\}$, let $\omega_i:Q\to [0,+\infty)$ be such that $c_0<\inf_Q \omega_i<\sup_Q \omega_i <\frac{1}{c_0}$. Then, there exists a unique $\bar{u}\in BV_{\omega_0}(Q)$ such that
\begin{align*}
\int_Q |\phi-\bar{u}|\,\dd x+ TGV_{\omega_0,\omega_1}(\bar{u},Q)=\min_{u\in BV_{\omega_0}(Q)}\left\{\int_Q |\phi-u|^2\,\dd x+TGV_{\omega_0,\omega_1}(u,Q)\right\}.
\end{align*}
Moreover, denoting by $(\omega_i)^{\mathrm{sc}^-}$ the lower semicontinuous envelope of $\omega_i$, $i\in\{0,1\}$, we have
$\bar{u}\in BV(Q)\cap BV_{(\omega_0)^{\mathrm{sc}^-}}(Q)$, and
\begin{align*}
TGV_{\omega_0,\omega_1}(\bar{u})=\int_Q (\omega_0)^{\mathrm{sc}^-}\dd |D\bar{u}-u^\ast|+\int_Q (\omega_1)^{\mathrm{sc}^-}\dd |\mathcal{E}u^\ast|,
\end{align*}
where $u^\ast \in BD_{\omega_1}(Q)\cap BD_{(\omega_1)^{\mathrm{sc}^-}}(Q)$ is a minimizer of \eqref{eqTGV-precise}
 associated to $\bar{u}$.
 \end{proposition}
 \begin{proof}
For $u\in BV_{\omega_0}(Q)$, we define
$$H[u]:= \int_Q |\phi-u|^2\dd x+TGV_{\omega_0,\omega_1}(u,Q),$$
and we set
$$\mu:=\inf_{u\in BV_{\omega_0}(Q)}H[u].$$
We have $0\leq \mu\leq F[0]=\|\phi\|_{L^2(Q)}^2,  $ and we may take a sequence $(u_n)_{n\in \N}\subset BV_{\omega_0}(Q)$ such that
$$\mu=\lim_{n\to +\infty}H[u_n].$$
Moreover, the boundedness assumptions on the weights $\omega_i$, $i\in\{0,1\}$, yield for all $x\in Q$ that$$c_0\leq (\omega_i)^{{\rm sc}^{-}}(x)\leq \frac{1}{c_0}.$$
Thus, by Lemma~\ref{lem:wTGV-min} and Theorems  \ref{thm:ppBVw-tgv} and \ref{lem:tv-cont}, we find for all \(n\in\NN\) large enough  that
\begin{align*}
&\mu+1\geq H[u_n]=\int_Q |\phi-u_n|^2\dd x+\scrV_{\omega_0}(Du_n-u_n^\ast,Q)+\scrV_{\omega_1}(\mathcal{E}u_n^\ast,Q)\\
&\quad=\int_Q |\phi-u_n|^2\dd x+\scrV_{(\omega_0)^{{\rm sc}^-}}(Du_n-u_n^\ast,Q)+\scrV_{(\omega_1)^{{\rm sc}^-}}(\mathcal{E}u_n^\ast,Q)\\
&\quad=\int_Q |\phi-u_n|^2\dd x+\int_Q (\omega_0)^{{\rm sc}^-}\dd|Du_n-u_n^\ast|+\int_Q (\omega_1)^{{\rm sc}^-}\dd|\mathcal{E}u_n^\ast|\\
&\quad\geq \int_Q |\phi-u_n|^2\dd x+{c_0}|Du_n-u_n^\ast|(Q)+{c_0}|\mathcal{E}u_n^\ast|(Q).
\end{align*}
An argument by contradiction as in the  classical $TGV$ case and variants thereof (see, e.g., \cite[Proposition~5.3]{DaFoLi23}) yields that the sequences $(u_n^\ast)_{n\in \N}$ and $(u_n)_{n\in \N}$ are uniformly bounded in $BD(Q)$ and $BV(Q)$, respectively. Thus, there exist $\bar{u}^\ast\in BD(Q)$ and $u\in BV(Q)$ such that, up to  extracting  a not relabelled subsequence, 
\begin{align*}
& u_n\weaklystar \bar{u}\quad\text{weakly* in }BV(Q),\\
& u_n^\ast\weaklystar \bar{u}^\ast\quad\text{weakly* in }BD(Q).
\end{align*}
By the bounds on the weights, and their lower-semicontinuous envelopes, and  Theorems \ref{thm:ppBVw-tgv} and \ref{lem:tv-cont},
we deduce that $\bar{u}\in BV_{(\omega_0)^{{\rm sc}^-}}(Q) \cap BV_{\omega_0}(Q) \cap BV(Q)$ and $\bar{u}^\ast\in BD_{(\omega_1)^{{\rm sc}^-}}(Q) \cap BD_{\omega_1}(Q)
\cap BD(Q)$, with
\begin{align}
\label{eq:ineq-mu}
&\mu\leq H[\bar{u}]\leq \int_Q |\phi-\bar{u}|^2\dd x+\scrV_{(\omega_0)^{{\rm sc}^-}}(D\bar{u}-\bar{u}^\ast,Q)+\scrV_{(\omega_1)^{{\rm sc}^-}}(\mathcal{E}\bar{u}^\ast,Q)\\
&\nonumber\quad \leq \lim_{n\to +\infty}H[u_n]=\mu.
\end{align}
Because of the strict convexity of the $L^2$-norm, we infer the uniqueness of $\bar{u}$. Finally, by \eqref{eq:ineq-mu}, 
\begin{align*}
&TGV_{\omega_0,\omega_1}(\bar{u}, Q)=\scrV_{(\omega_0)^{{\rm sc}^-}}(D\bar{u}-\bar{u}^\ast,Q)+\scrV_{(\omega_1)^{{\rm sc}^-}}(\mathcal{E}\bar{u}^\ast,Q).
\end{align*}
 The last part of the statement is then a consequence of Theorems 
\ref{thm:ppBVw-tgv} and \ref{lem:tv-cont}.
 \end{proof}

\subsection{On Level \texorpdfstring{$1$}{1}}
\label{sub-l1tgv}
As we address next, and similarly to the \((\scrL\!\scrS)_{{TV\!}_{\omega}}\) case, the box constraint provides a stopping criterion for the \(TGV\)-learning scheme. 

To proceed as in Theorem \ref{thm:solL1}, we need an analog to Proposition \ref{prop:Ja257}, which we now prove. Recalling that  $L$ represents a cell in a dyadic partition of $Q$,  
 we will use the Sobolev inequality   in $BV(L)$ yielding for every \(u\in BV(L)\) that 
\begin{equation}\label{eq:sobineqbv}
\|u - [u]_L\|_{L^2(L)} \leq C^{BV}_Q |Du|(L),
\end{equation}
where $[u]_L \in \R$ is the average of $u$ in $L$, and the constant $C^{BV}_Q$ depends only on the shape of $Q$ because of scale invariance of the embedding \(BV\) in \(L^2\) in dimension $d=2$. Moreover, we also have for any $w \in BD(L)$ that
\begin{equation}\label{eq:sobineqbd}
\|w - R_{M_w} - v_w\|_{L^2(L)} \leq C^{BD}_Q |\Ecal w|(L),
\end{equation}
where $v_w \in \R^2$, $M_w$ is a skew-symmetric matrix (that is, with $M^\top + M=0$, the set of which we denote by $\R^{2 \times 2}_{\mathrm{skew}}$), and $R_{M_w}$ denotes the function defined for $M_w \in \R^{2 \times 2}$ by $R_{M_w}(x) = M_w x$.
\begin{lemma}\label{lem:rotangle}
Let \(L\subset Q\) be a dyadic square. Then, there is a constant $C_Q^{rot} > 0$ such that for every $u \in BV(L)$  and for every skew-symmetric matrix \(M\in \RR^{2\times 2}_{\mathrm{skew}}\), we have
\begin{equation}\label{eq:rotineq}
C_Q^{rot} |Du|(L) \leq |Du - R_M|(L).
\end{equation}
\end{lemma}
\begin{proof}Suppose that \eqref{eq:rotineq} does not hold; then,  we may find functions $u_n \in BV(L)$ with $|Du_n|(L) = 1$ and skew-symmetric matrices $M_n \in
\RR^{2\times 2}_{\mathrm{skew}}$ for which
\begin{equation}\label{eq:notrotineq}
\frac{1}{n}=\frac{1}{n}|Du_n|(L) > |Du_n - R_{M_n}|(L).
\end{equation}
Then, in particular,  $\|R_{M_n}\|_{L^1(L)}\leq 2$; consequently, since $\R^{2 \times 2}_{\mathrm{skew}}$ is a finite-dimensional set, we can assume that $R_{M_n} \to R_{M_\infty}$
for some skew-symmetric matrix $M_\infty$, up to taking
a not relabelled subsequence.

On the other hand, recalling \eqref{eq:sobineqbv}, there are constants $c_n \in \R$ satisfying
\[\|u_n - c_n\|_{L^2(L)} \leq C^{BV}_Q  |Du_n|(L);\]
thus, up to taking a  not relabelled further subsequence, we have that $u_n - c_n \stackrel{\ast}{\rightharpoonup} u_\infty \in BV(L)$ for some \(u_\infty \in BV(L)\). Using \eqref{eq:notrotineq} once more, we must have $Du_\infty = R_{M_\infty}$. At this point, we can distinguish two cases,  $M_\infty = 0$ or \(M_\infty\not=0\). 

If    $M_\infty
= 0$ , then
\[\frac{1}{n}=\frac{1}{n}|Du_n|(L) > |Du_n - R_{M_n}|(L) \to 1,\]
which cannot be. 

If $M_\infty \neq 0$, then, using the antisymmetry of  $DR_{M_\infty}=M_{\infty}$, we again arrive at a contradiction, since
\[\curl Du_\infty = 0 \ \text{ but }\  |\curl R_{M_\infty}| = \sqrt{2} |M_\infty| >0.\]
To see that the last equality holds, just notice that in the two dimensional case under consideration we must have
\[M_\infty = \begin{pmatrix}0 & a \\ -a & 0\end{pmatrix}\text{ for some }a \neq 0, \text{ which implies } \curl R_{M_\infty} = -2a.\]

 Thus, we have proved that there is a constant $C_L$, possibly depending on $L$, such that 
\[C_L |Du|(L) \leq |Du - R_M|(L) \quad\text{ for all }M \in \R^{2 \times 2}_{\mathrm{skew}}.\]
To see that $C_L$ is independent of the size of $L$, we just notice that this inequality holds for all $M$ and that upon rescaling $x \mapsto r x$ it is enough to replace $M$ by $M/r$ to maintain the inequality.
\end{proof}

The next proposition guarantees that if a dyadic square $L\subset Q$ is small enough, then a solution $u_{\alpha_0,\alpha_1}$ of Level $3$  of our \(TGV\) learning scheme in \eqref{lsTGVomega} is affine for every $(\alpha_0,\alpha_1) \in \big[c_0,\frac{1}{c_0}\big]\times\big[c_1,\frac{1}{c_1}\big]$. Let us remark that a related result is contained in \cite[Proposition 6]{PaVa15}, which we make quantitative and with a scaling that enables us to draw conclusions on the cell size.

\begin{proposition}\label{prop:loc-affine} 
Fix \(c_0\), \(c_1>0\)  and \(L\subset Q\) a dyadic square. Let \(\bar \alpha_L \) be the optimal parameter  given by \eqref{eq:box-min-TGV}, where \(u_{\alpha,L}\) is defined by \eqref{eq:TGV} and \eqref{eq:tgv} (with \(Q\) replaced by \(L\)), and let  \(C_Q^{BV}\), \(C_Q^{BD}\), and \(C_Q^{rot}\) be the constants in
\eqref{eq:sobineqbv}, \eqref{eq:sobineqbd}, and \eqref{eq:rotineq}, respectively. If
 \begin{equation}\label{eq:l2boundforconst}
\|u_\eta\|_{L^2(L)} < \min \left( c_0, \frac{c_1}{C_Q^{BD}|L|^{1/2}}\right)\frac{C_Q^{rot}}{C_Q^{BV}},
\end{equation}
 then  $\bar \alpha_L:=(\overline{\alpha}_0, \overline{\alpha}_1 ) = (c_0, c_1)$ and
  $u_{\bar \alpha_L }:=u_{(\overline{\alpha}_0, \overline{\alpha}_1),L}$
is affine on $L$, with \(u_{\bar \alpha_L } = \langle u_\eta\rangle_L\).
\end{proposition}

\begin{proof} To simplify the notation in the proof, we omit the dependence of  $TGV_{\alpha_0,\alpha_1}$ and \( u_{{\alpha}_0, {\alpha}_1}\) on \(L\) by writing $TGV_{\alpha_0,\alpha_1}(\cdot)$ in place of $TGV_{\alpha_0,\alpha_1}(\cdot,L)$ and \( u_{({\alpha}_0, {\alpha}_1),L}\), respectively.

Fix  $(\alpha_0,\alpha_1)
\in \big[c_0,\frac{1}{c_0}\big]\times\big[c_1,\frac{1}{c_1}\big]$. The optimality condition for \eqref{eq:TGV} reads as
\[u_\eta - u_{\alpha_0, \alpha_1} \in \partial TGV_{\alpha_0, \alpha_1}(u_{\alpha_0, \alpha_1}).\]
Since $TGV_{\alpha_0,\alpha_1}$ is positively one-homogeneous, we have that
\[z \in \partial TGV_{\alpha_0,\alpha_1}(u_{\alpha_0, \alpha_1}) \text{ if and only if }z \in \partial TGV_{\alpha_0,\alpha_1}(0) \text{ and }\int_L z \,u_{\alpha_0, \alpha_1} \dd x = TGV_{\alpha_0, \alpha_1}(u_{\alpha_0, \alpha_1}).\]
Furthermore, by the definition of subgradient,
\[z \in \partial TGV_{\alpha_0, \alpha_1}(0) \text{ if and only if } \int_L z\, \overline{u} \dd x \leq TGV_{\alpha_0, \alpha_1}(\overline u) \text{ for all }\overline{u} \in L^2(L).\]

 Now, given $v \in \R^2$ and $c \in \R$, we denote by $A_{v,c}$ the affine function given by $A_{v,c}(x)=v\cdot x + c$. Because
\begin{equation}
\label{eq:TGVonAff}
\begin{aligned}
TGV_{\alpha_0, \alpha_1}(A_{v,c})=0, 
\end{aligned}
\end{equation}
we deduce from the above with \(z=u_\eta - u_{\alpha_0, \alpha_1}\) and  \(\overline{u}=\pm A_{v,c}\)  that \(\int_L(u_\eta - u_{\alpha_0, \alpha_1}) A_{v,c}  \dd x=0\) for any $v \in \R^2$ and $c \in \R$; moreover,
\begin{equation*}
\begin{aligned}
TGV_{\alpha_0,\alpha_1}(u_{\alpha_0, \alpha_1}) &= \int_L (u_\eta - u_{\alpha_0, \alpha_1}) u_{\alpha_0, \alpha_1} \dd x = \int_L (u_\eta - u_{\alpha_0, \alpha_1}) ( u_{\alpha_0, \alpha_1} - A_{v,c} ) \dd x \\& \leq \|u_\eta - u_{\alpha_0, \alpha_1}\|_{L^2(L)} \|u_{\alpha_0, \alpha_1} - A_{v,c}\|_{L^2(L)}.
\end{aligned}
\end{equation*}
Thus, taking the infimum over $v \in \R^2$ and $c \in \R$ and recalling \eqref{eq:profAff}, we conclude that
\begin{equation}
\label{eq:tgvupperbound}
\begin{aligned}
TGV_{\alpha_0,\alpha_1}(u_{\alpha_0, \alpha_1}) \leq \|u_\eta - u_{\alpha_0,
\alpha_1}\|_{L^2(L)}  \|u_{\alpha_0, \alpha_1}-\langle u_{\alpha_0, \alpha_1}\rangle_L\|_{L^2(L)}.
\end{aligned}
\end{equation}

On the other hand, since the infimum in the definition of $TGV_{\alpha_0, \alpha_1}$ is attained, there is a $w_u \in BD(L)$ for which
\begin{equation*}
\begin{aligned}
TGV_{\alpha_0,\alpha_1}(u_{\alpha_0, \alpha_1}) &= \inf_{w\in BD(L)} \Big\{\alpha_0 |Du_{\alpha_0, \alpha_1}-w|(L) + \alpha_1 |\Ecal w|(L)\Big\} \\ 
&= \alpha_0 |Du_{\alpha_0, \alpha_1}-w_u|(L) + \alpha_1 |\Ecal w_u|(L)\\
&\geq \alpha_0 |Du_{\alpha_0, \alpha_1}-w_u|(L) + \frac{\alpha_1}{C^{BD}_Q} \|w_u - R_{M_{w_u}} - v_{w_u}\|_{L^2(L)},
\end{aligned}
\end{equation*}
where we have used the inequality \eqref{eq:sobineqbd} for some skew-symmetric matrix $M_{w_u}\in\RR^{2\times 2}$ and vector $v_{w_u}\in\RR^2$. 
Setting $R_u := R_{M_{w_u}}$ and $v_u := v_{w_u}$, we get that
\begin{equation*}
\begin{aligned}
TGV_{\alpha_0,\alpha_1}(u_{\alpha_0, \alpha_1}) &\geq \alpha_0 |Du_{\alpha_0, \alpha_1}-w_u|(L) + \frac{\alpha_1}{C^{BD}_Q} \|w_u - R_u - v_u\|_{L^2(L)}\\
& \geq \alpha_0 |Du_{\alpha_0, \alpha_1}-w_u|(L) + \frac{\alpha_1}{C^{BD}_Q |L|^{1/2}} \|w_u - R_u - v_u\|_{L^1(L)} \\
&\geq \min\left( c_0, \frac{c_1}{C_Q^{BD}|L|^{1/2}}\right) \Big[ |Du_{\alpha_0, \alpha_1}-w_u|(L) +\|w_u - R_u - v_u\|_{L^1(L)}\Big]\\
&\geq\min \left( c_0, \frac{c_1}{C_Q^{BD}|L|^{1/2}}\right) |Du_{\alpha_0, \alpha_1}-R_u - v_u|(L)\\
&= \min\left( c_0, \frac{c_1}{C_Q^{BD}|L|^{1/2}}\right)  |D(u_{\alpha_0, \alpha_1}-A_{v_u, 0})-R_u|(L).
\end{aligned}
\end{equation*}
Now, we can apply Lemma \ref{lem:rotangle} to $u_{\alpha_0, \alpha_1}-A_{v_u, 0}$ and the Sobolev inequality \eqref{eq:sobineqbv} to obtain for some $c_u \in \R$ that
\begin{equation}\label{eq:tgvlowerbound3}
\begin{aligned}
TGV_{\alpha_0,\alpha_1}(u_{\alpha_0, \alpha_1}) &\geq \min\left( c_0, \frac{c_1}{C_Q^{BD}|L|^{1/2}}\right) C^{rot}_Q |D(u_{\alpha_0, \alpha_1}-A_{v_u, 0})|(L)\\
&\geq \min\left( c_0, \frac{c_1}{C_Q^{BD}|L|^{1/2}}\right) \frac{C^{rot}_Q}{C_Q^{BV}} \|u_{\alpha_0, \alpha_1}-A_{v_u, c_u}\|_{L^2(L)} \\
&\geq \min\left( c_0, \frac{c_1}{C_Q^{BD}|L|^{1/2}}\right)
\frac{C^{rot}_Q}{C_Q^{BV}} \|u_{\alpha_0, \alpha_1}-\langle u_{\alpha_0, \alpha_1}\rangle_L\|_{L^2(L)},
\end{aligned}
\end{equation}
where we used \eqref{eq:profAff} once more.  
Then, if $u_{\alpha_0, \alpha_1}$ was not affine, then \(\|u_{\alpha_0, \alpha_1}-\langle u_{\alpha_0, \alpha_1}\rangle_L\|_{L^2(L)}>0\), so
we could combine \eqref{eq:tgvlowerbound3} with the upper bound \eqref{eq:tgvupperbound} and minimality of $ u_{\alpha_0, \alpha_1}$ in \eqref{eq:TGV} to obtain
\[\min \left( c_0, \frac{c_1}{C_Q^{BD}|L|^{1/2}}\right)\frac{C^{rot}_Q}{C_Q^{BV}} \leq  \|u_\eta - u_{\alpha_0, \alpha_1}\|_{L^2(L)} \leq \|u_\eta\|_{L^2(L)},\]
which contradicts \eqref{eq:l2boundforconst}. Thus,  $u_{\alpha_0, \alpha_1}$ must be affine. 

Finally, using \eqref{eq:TGVonAff}, \eqref{eq:profAff}, and \(\langle u_\eta\rangle_L\) as a competitor in  \eqref{eq:TGV}, we conclude that \(u_{\alpha_0, \alpha_1} =\langle u_{\alpha_0, \alpha_1}\rangle_L =  \langle u_\eta\rangle_L\). Hence, \(\bar\alpha_L = (c_0,c_1)\), and this concludes the proof.  
\end{proof}
Owing to Proposition \ref{prop:loc-affine}, we are now in a position to reduce the minimum problem in Level $1$ of our training scheme to a minimization over a finite set of admissible partitions.
\begin{theorem}
\label{thm:finiteTGV}
Consider the learning scheme $(\scrL\!\scrS)_{{TGV\!}_{\omega}}$ in \eqref{lsTGVomega} with \eqref{eq:alpha-Lg-TGV} restricted by \eqref{eq:bc-TGV} (see \eqref{eq:box-min-TGV}). 
Then, there exist \(\kappa\in \NN\) and \(\mathscr{L}_1, ...
, \mathscr{L}_\kappa \in \mathscr{P}\) such that
\begin{equation*}
\begin{aligned}
\argmin\left\{\int_Q|u_c-u_{\mathscr{L}}|^2\dd x:\,\mathscr{L}\in
\mathscr{P}\right\} = \argmin\left\{\int_Q|u_c-u_{\mathscr{L}_i}|^2\dd x:\,i\in
\{1,...,\kappa\}\right\}.
\end{aligned}
\end{equation*}
\end{theorem}
\begin{proof}
The proof is analogous to that of Theorem \ref{thm:solL1}, so we only provide a sketch of the argument. The only difference  here is that  instead of being a constant, the solution
 \(u_{\alpha,L}\)  of Level $1$ is affine for any $\alpha:=(\alpha_0,\alpha_1)\in \big[c_0,\frac{1}{c_0}\big]\times\big[c_1,\frac{1}{c_1}\big]$ on squares $L$ on which \eqref{eq:l2boundforconst} holds,  due to Proposition \ref{prop:loc-affine}. Moreover, 
 $TGV_{\alpha_0,\alpha_1}(u_{\alpha,L},L)=0$ and, recalling \eqref{eq:box-min-TGV}, the optimal parameter
given by \eqref{eq:box-min-TGV} is $\bar \alpha_L=(c_0,c_1)$. As in the proof of Theorem \ref{thm:solL1}, this observation allows us to replace any partition $\mathscr{L}^\ast$ containing such small dyadic squares with another partition $\overline{\mathscr{L}}^\ast$ whose dyadic squares have all side length above the threshold provided by \eqref{eq:l2boundforconst} without affecting the minimizer of Level $2$. We refer to Figure \ref{fig:partition} for a graphical idea of the argument and to Theorem \ref{thm:solL1} for the details of the proof. 
\end{proof}

We conclude this section by proving existence of an optimal solution to the learning
scheme \((\scrL\!\scrS)_{{TGV\!}_\omega}\).

\begin{proof}[Proof of Theorem \ref{thm:TGVomega}]
The result follows directly by combining the analysis in Subsection~\ref{sect:TGVL3},  Proposition~\ref{prop:L2tgv},
and Theorem ~\ref{thm:finiteTGV}. 
\end{proof}

\subsection{Stopping criteria and box constraint for TGV}
\label{subs:stop-TGV}

In this subsection, we prove a $TGV$-counterpart to Theorem \ref{thm:onalpha}. Our result reads as follows. 

\begin{theorem}\label{thm:onalpha-TGV}
Let \(\Omega\subset \RR^2\) be a bounded, Lipschitz domain and, for each \(\alpha\in(0,+\infty)^2\), let \(u_\alpha\in BV(\Omega)\) be  given by \eqref{eq:TGV} with \(L\) replaced by \(\Omega\). Assume that the two following conditions on the training data
hold: 
\begin{itemize}
\item [\textit{i)}] There exists $\hat{\alpha}\in (0,+\infty)^2$ such that $TGV_{\hat{\alpha}_0,\hat{\alpha}_1}(u_c,\Omega) < TGV_{\hat{\alpha}_0,\hat{\alpha}_1}(u_\eta,\Omega)$;

\item [\textit{ii)}] $\displaystyle \Vert u_\eta - u_c\Vert^2_{L^2(\Omega)}
<\Vert \langle u_\eta \rangle - u_c\Vert^2_{L^2(\Omega)} $.
\end{itemize}
 Then, there exists
\begin{equation}\label{eq:optalpha-generalTGV}\alpha^*_\Omega \in (0,+\infty)^2 \cup \big(\{+\infty\} \times (0, +\infty)\big) \cup \big((0, +\infty) \times \{+\infty\}\big)\end{equation}
 such that
\begin{equation}
\label{eq:minIalpha-TGV}
\widehat J( \alpha^*_\Omega)=\min_{\alpha\in[0,+\infty]^2} \widehat J(\alpha),
\end{equation}
where $\widehat J$ is a (lower semicontinuous) extension on $[0, +\infty]^2$ (see \eqref{eq:repIsc-TGV} in Lemma \ref{lem:Ilsc-TGV} below) of the function \(J:(0,+\infty)^2\to[0,+\infty)\) defined by
\begin{equation}
\label{eq:defJ-TGV}J(\alpha):=\int_\Omega |u_c-u_\alpha|^2\dd x \quad\text{for } \alpha\in (0,+\infty)^2.\end{equation}

Additionally, there exist  positive constants, \(c_\Omega\) and \(C_\Omega\), such that any
minimizer, \(\alpha^*_\Omega\), of \(\widehat J\) over \([0,+\infty]^2\) satisfies
\( c_\Omega\leq \min\{(\alpha^*_\Omega)_0,(\alpha^*_\Omega)_1\} < C_\Omega\Vert u_\eta\Vert_{L^2(\Omega)}\). 

In particular, if
 \(\Omega=L\) with   \(L\subset Q\) is a dyadic square,
then  there exists a positive constant, \(c_L\), such that any
minimizer, \(\alpha^*_L\), of \(\widehat J\) over \([0,+\infty]^2\) satisfies
\( c_L\leq \min\{(\alpha^*_L)_0,(\alpha^*_L)_1\} < C_Q\Vert u_\eta\Vert_{L^2(L)}\), where
\(C_Q\) is a constant given by Proposition~\ref{prop:loc-affine}.\end{theorem}

Owing to the orthogonality property \eqref{eq:affine-proj}, condition $ii)$ in the statement of the theorem is equivalent to requiring that
\(\Vert u_c-\langle u_c \rangle-u_\eta+\langle u_\eta \rangle\Vert_{L^2(\Omega)}^2\leq \Vert u_c-\langle u_c \rangle\Vert_{L^2(\Omega)}^2.\)
In other words, $ii)$ is satisfied provided that the perturbation which the noise causes on the non-affine portion of $u_c$ is small in the $L^2$-sense compared to the original non-affine component of $u_c$. This is the case, for example, if $\eta=u_\eta-u_c$ and  $\eta-\langle \eta\rangle$ has a small $L^2$-norm, regardless of the $L^2$-norm of $\langle \eta \rangle$.
 
We remark that the conclusion of the theorem in the general case is slightly weaker than the corresponding result for the $TV$-setting. Indeed, while we can show that both entries of optimal parameters must be uniformly bounded away from zero, we can only prove that their minimum is uniformly bounded from above but cannot prevent that just one of the entries blows up to infinity. This is due to the fact that, without additional conditions, the maps $u_\alpha$ are not necessarily affine if just one of the entries of $\alpha$ becomes infinity, cf. also \cite[Proposition 6]{PaVa15} for comparison. 

However, as a direct consequence of our result, we find a complete characterization for the case in which the analysis of $TGV$ reduces to a one-dimensional problem.

\begin{corollary}\label{thm:onalpha-TGV1}
Under the same assumption and with the same notation of Theorem~\ref{thm:onalpha-TGV}, setting $u_{\lambda}:=u_{\lambda(\hat\alpha_0,\hat\alpha_1)}$ for every $\lambda\in [0,+\infty]$,
there exists \( \lambda^*_\Omega\in (0,+\infty)\) such that
\begin{equation*}
\begin{aligned}
J( \lambda^\ast_\Omega(\hat\alpha_0,\hat\alpha_1))=\min_{\lambda\in(0,+\infty)} J(\lambda(\hat\alpha_0,\hat\alpha_1)). 
\end{aligned}
\end{equation*}

Additionally, there exist  positive constants, \(c_\Omega\) and \(C_\Omega\), such that any
minimizer \(\lambda^*_\Omega\) satisfies
\( c_\Omega\leq \lambda^*_\Omega < C_\Omega\Vert u_\eta\Vert_{L^2(\Omega)}\). 

In particular, if
 \(\Omega=L\) with   \(L\subset Q\) is a dyadic square,
then  there exists a positive constant, \(c_L\), such that any
minimizer \(\lambda^*_L\) satisfies
\( c_L\leq \lambda^*_L < C_Q\Vert u_\eta\Vert_{L^2(L)}\), where
\(C_Q\) is a constant given by Proposition~\ref{prop:loc-affine}. 
\end{corollary}

As in the case of the total variation, we proceed by first studying the limiting behavior of the sum of fidelity and $TGV$-seminorm in the sense of $\Gamma$-convergence. To describe the situation in which the tuning coefficients approach $+\infty$, it is useful to recall that $\mathcal{M}_b(\Omega;\mathbb{R}^d)$ denotes the set of bounded Radon measures on $\Omega$ with values in $\mathbb{R}^d$ and 
${\mathrm{Ker}}\,\mathcal{E}\,(\Omega;\R^d)$ is the set of all maps $\Phi:\Omega\to \mathbb{R}^d$ such that $\mathcal{E}\Phi=0$. In particular, $\Phi\in {\mathrm{Ker}}\,\mathcal{E}\,(\Omega;\R^d)$ if and only if there exists $M\in \mathbb{R}^{d\times d}_{\mathrm{skew}}$ and $m\in \mathbb{R}^d$ such that $\Phi(x)=Mx+m$ for every $x\in \Omega$.

We also recall the function
$$m_{\mathcal{E}}:\mathcal{M}_b(\Omega;\mathbb{R}^d)\to {\mathrm{Ker}}\,\mathcal{E}\,(\Omega;\R^d),$$
introduced in \cite[Proposition 3]{PaVa15}, and
defined as the solution to the minimum problem
\begin{equation}
\label{def:m-E}
|\mu-m_{\mathcal E}(\mu)|(\Omega)=\min\{ |\mu-\phi|(\Omega):\,\phi\in {\mathrm{Ker}}\,\mathcal{E}\,(\Omega;\R^d)\},
\end{equation}
for every $\mu\in \mathcal{M}_b(\Omega;\mathbb{R}^d)$. Recall that $BH(\Omega)$ denotes the space of functions with bounded Hessian on $\Omega$, namely maps $u\in BV(\Omega)$ such that $D^2 u\in \mathcal{M}_b(\Omega;\mathbb{R}^{d\times d})$.

\begin{lemma}
\label{lem:gammaTGV}

Let \(\Omega\subset \RR^2\) be a bounded, Lipschitz domain and, for each
\(\alpha\in(0,+\infty)^2\), let \(u_\alpha\in BV(\Omega)\) be  given by \eqref{eq:TGV} and \eqref{eq:tgv},
with \(L\) and \(Q\) replaced by \(\Omega\). Consider the family of  functionals \((G_{\bar \alpha})_{\bar\alpha\in[0,+\infty]^2}\),
where \(G_{\bar \alpha}:L^1(\Omega)\to[0,+\infty]\) is defined by
\begin{align*}
&G_\alpha [u]:=\begin{cases}
\int_{\Omega}|u_\eta-u|^2\dd x+ TGV_{\alpha_0,\alpha_1}(u,\Omega) &\text{if
} u\in BV(\Omega),\\
+\infty &\text{otherwise,}
\end{cases}\quad \text{ for } \bar\alpha=:\alpha=(\alpha_0,\alpha_1)\in(0,+\infty)^2,\\
&G_{0,\bar\alpha_1} [u]:=\begin{cases}
\int_{\Omega}|u_\eta-u|^2\dd x &\text{if
} u\in L^2(\Omega),\\
+\infty &\text{otherwise,}
\end{cases}\quad \text{ for } \bar\alpha_0=0\text{ and }\bar\alpha_1\in[0,+\infty],\\
&G_{\infty,\alpha_1} [u]:=\begin{cases}
\int_{\Omega}|u_\eta-u|^2\dd x+\alpha_1|D^2 u|(\Omega) &\text{if
} u\in BH(\Omega),\\
+\infty &\text{otherwise,}
\end{cases}\quad \text{ for }\begin{array}{l} \bar\alpha_0=+\infty,\\\bar\alpha_1=:\alpha_1\in(0,+\infty),\end{array}\\
&G_{\bar{\alpha}_0,0} [u]:=\begin{cases}
\int_{\Omega}|u_\eta-u|^2\dd x &\text{if
} u\in BV(\Omega),\\
+\infty &\text{otherwise,}
\end{cases}\quad \text{ for } \bar\alpha_0\in(0,+\infty]\text{ and }\bar\alpha_1=0,\\
&G_{\alpha_0,\infty} [u]:=\begin{cases}
\int_{\Omega}|u_\eta-u|^2\dd x +\alpha_0  |Du-m_{\mathcal E}(Du)|(\Omega)&\text{if
} u\in BV(\Omega),\\
+\infty &\text{otherwise,}
\end{cases}\quad \text{ for }\begin{array}{l}\bar\alpha_0=:\alpha_0\in(0,\infty),\\ \bar\alpha_1=+\infty,\end{array}\\
&G_{\infty,\infty} [u]:=\begin{cases}
\int_{\Omega}|u_\eta-u|^2\dd x &\text{if
} Du\in {\mathrm{Ker}}\,\mathcal{E}(\Omega;\mathbb{R}^d),\\
+\infty &\text{otherwise,}
\end{cases}\quad \text{ for } \bar\alpha_0=\bar\alpha_1=+\infty.
\end{align*}
Let \((\alpha_j)_{j\in\NN} \subset (0,+\infty)^2\) and \(\bar
\alpha\in [0,\infty]^2\) be such that \(\alpha_j\to \bar \alpha\)
in \([0,+\infty]^2\). Then, \((G_{\alpha_j})_{j\in\NN}\)   \(\Gamma\)-converges
to \(G_{\bar\alpha}\) in \(L^1(\Omega)\).
\end{lemma}

\begin{proof}
We first prove that if \((u_j)_{j\in\NN}\subset L^1(\Omega)\)
and \(u\in L^1(\Omega)\) are such that \(u_j\to u\) in \(L^1(\Omega)\),
then
\begin{equation}
\label{eq:liminfE-TGV}
\begin{aligned}
G_{\bar \alpha} [u] \leq \liminf_{j\to\infty} G_{\alpha_j} [u_j].
\end{aligned}
\end{equation}
Without loss of generality, we work under the assumptions that
\begin{equation*}
\begin{aligned}
\liminf_{j\to\infty} G_{\alpha_j} [u_j] = \lim_{j\to\infty} G_{\alpha_j}
[u_j]<+\infty \quad \text{and} \quad \sup_{j\in\NN}  G_{\alpha_j}
[u_j] <+\infty.
\end{aligned}
\end{equation*}
Then, \(u_j\in BV(\Omega)\) for all \(j\in\NN\), \(\sup_{j\in\NN}  \int_{\Omega}|u_\eta-u_j|^2\dd
x
 <+\infty\) and \(\sup_{j\in\NN}   TGV_{(\alpha_j)_0,(\alpha_j)_1}(u_j,\Omega)
 <+\infty\). Hence, \(u\in L^2(\Omega)\) and \(u_j \weakly u\) weakly in
\(L^2(\Omega)\).
For each $j\in \mathbb{N}$, let $u^\ast_j\in BD(\Omega)$ be such that
\begin{equation}
\label{eq:ujstar}TGV_{(\alpha_j)_0,(\alpha_j)_1}(u_j)=(\alpha_j)_0|Du_j-u_j^\ast|(\Omega)+(\alpha_j)_1|\mathcal{E}u_j^\ast|(\Omega).
\end{equation}
We now consider each limiting behavior of the sequence $(\alpha_j)_{j\in \mathbb{N}}$ separately.
\begin{itemize}
\item[(i)] If \(\bar\alpha=\alpha\in (0,+\infty)^2\), then an argument by contradiction as the classical $TGV$ case and variants thereof (see, e.g., \cite[Proposition 5.3]{DaFoLi23}) yields uniform bounds for  sequences $(u_j)_{j\in \N}$ and $(u_j^\ast)_{j\in \N}$ in  $BV(\Omega)$ and $BD(\Omega)$, respectively.
Thus, \(u\in
BV(\Omega)\) and \(u_j\weakly
u\) weakly-\(\star\) in \(BV(\Omega)\). Additionally, there exists $u^\ast\in BD(\Omega)$ such that, up to extracting a further subsequence, \(u_j^\ast\weakly
u^\ast\) weakly-\(\star\) in \(BD(\Omega)\),  from which  
\eqref{eq:liminfE-TGV} follows.

\item[(ii)] If \(\bar\alpha_0=0\), then
\eqref{eq:liminfE-TGV} holds by the lower-semicontinuity of the \(L^2\)-norm
with respect to the weak convergence in \(L^2(\Omega)\).

\item[(iii)] If \(\bar\alpha_0=+\infty\) and $\bar{\alpha}_1\in (0,+\infty)$, then
\((u_j^\ast)_{j\in\NN}\) is bounded in \(BD(\Omega)\). Thus, there exists $u^\ast\in BD(\Omega)$ such that, up to extracting a further subsequence, \(u_j^\ast\weakly
u^\ast\) weakly-\(\star\) in \(BD(\Omega)\). Additionally,  \(\lim_{j\to\infty} |Du_j-u_j^\ast|(\Omega)=0\).
Thus,   \(u_j\to
u\) strongly in \(BV(\Omega)\), \(u\in
BH(\Omega)\), and 
\eqref{eq:liminfE-TGV} holds by the lower-semicontinuity of the \(L^2\)-norm
with respect to the strong convergence in \(BV(\Omega)\).

\item[(iv)] If $\bar\alpha_0\in(0,\infty]$ and $\bar\alpha_1=0$, then the situation is analogous to (ii).

\item[(v)] If $\bar\alpha_1=+\infty$ and  $\bar{\alpha}_0\in (0,+\infty)$, then there exists $u^\ast$ affine and such that
$u^\ast_j\to u^\ast$ strongly in $BD(\Omega)$ and $(u_j)_{j\in \mathbb{N}}$ is uniformly bounded in $BV(\Omega)$, so that $u_j\weaklystar u$ weakly-\(\star\) in $BV(\Omega)$. The statement follows from the lower semicontinuity of the total variation with respect to the weak-\(\star\) convergence of measures, as well as from \eqref{def:m-E}.

\item[(vi)] If $\bar\alpha_0=\bar\alpha_1=+\infty$, then there exists $u^\ast\in {\mathrm{Ker}}\,\mathcal{E}(\Omega;\mathbb{R}^d)$ such that
$u^\ast_j\to u^\ast$ strongly in $BD(\Omega)$ and $Du_j\to u^\ast$ strongly in $\mathcal{M}_b(\Omega;\mathbb{R}^d)$. Thus, $Du=u^\ast$ and the statement follows.
\end{itemize}

Next, we show that for any  \(u\in L^1(\Omega)\), there exists
  \((u_j)_{j\in\NN}\subset L^1(\Omega)\)
 such that \(u_j\to u\) in \(L^1(\Omega)\)
and
\begin{equation}
\label{eq:limsupE-TGV}
\begin{aligned}
G_{\bar \alpha} [u] \geq \limsup_{j\to\infty} G_{\alpha_j} [u_j].
\end{aligned}
\end{equation}
Again, we detail the argument in each case separately.
\begin{itemize}
\item[(i)] If $\bar\alpha=\alpha\in (0,+\infty)^2$ then we can assume, without loss of generality, that $u\in BV(\Omega)$. The conclusion follows then by a classical argument relying on the continuity of $TGV$ with respect to its tuning parameters (see, e.g., \cite[Theorem 4.2]{DaFoLi23}).
\item[(ii)] If $\bar\alpha_0=0$, then we consider for every $u\in L^2(\Omega)$ an approximating sequence $(u_k)_{k\in \mathbb{N}}\subset C^\infty_c(\Omega)$ such that $u_k\to u$ strongly in $L^2(\Omega)$. Choosing the null function as a competitor in the definition of $TGV$, we find that
$$TGV_{(\alpha_j)_0,(\alpha_j)_1}(u_k)\leq (\alpha_j)_0 |D u_k|(\Omega).$$
Thus,
$$\lim_{j\to +\infty}G_{\alpha_j}[u_k] = G_{0,\bar{\alpha}_1}[u_k]$$
 for every $\bar{\alpha}_1\in [0,+\infty]$ and every $k\in \mathbb{N}$. The thesis follows then by a classical diagonalization argument.
\item[(iii)] If $\bar\alpha_0=+\infty\text{ and }\bar\alpha_1=\alpha_1\in(0,+\infty)$ then we can assume, without loss of generality, that $u\in BH(\Omega)$. In particular, $\nabla u \in BD(\Omega)$ which we can then use as a competitor in the definition of $TGV$ to infer that$$TGV_{(\alpha_j)_0,(\alpha_j)_1}(u)\leq (\alpha_j)_1 |D^2u|(\Omega).$$
Thus, 
$$\limsup_{j\to +\infty}G_{\alpha_j}[u]\leq \limsup_{j\to +\infty}\bigg(\int_\Omega|u_\eta-u|^2\dd x+(\alpha_j)_1|D^2 u|(\Omega)\bigg)=G_{\infty,\alpha_1} [u].$$

\item[(iv)] If $\bar\alpha_0\in (0,+\infty]$ and $\bar\alpha_1=0$, arguing by approximation as in case (ii), we can assume without loss of generality that $u\in C^\infty_c(\Omega)$. Then, choosing $\nabla u$ as a competitor in the definition of $TGV$, we find that
$$TGV_{(\alpha_j)_0,(\alpha_j)_1}(u)\leq (\alpha_j)_1 |\nabla^2 u|(\Omega).$$
Hence, arguing as in case (ii) once more, yields \eqref{eq:limsupE-TGV}.
\item[(v)] If $\bar\alpha_0=\alpha_0\in (0,+\infty)$ and $\bar\alpha_1=+\infty$, then we can assume that $u\in BV(\Omega)$. Choosing $m_{\mathcal E}(Du)$ in the definition of $TGV$, yields
$$TGV_{(\alpha_j)_0,(\alpha_j)_1}(u)\leq (\alpha_j)_0 |D u-m_\mathcal E(Du)|(\Omega).$$
Hence, arguing as in case (ii), we infer \eqref{eq:limsupE-TGV}.
\item[(vi)] If $\bar\alpha_0=\bar\alpha_1=+\infty$, then, without loss of generality, we can assume that $Du\in {\mathrm{Ker}}\,\mathcal{E}(\Omega;\mathbb{R}^d)$. Choosing $Du$ as a competitor in the definition of $TGV$ shows that $TGV_{(\alpha_j)_0,(\alpha_j)_1}(u)=0$ for every $j\in \mathbb{N}$, from which \eqref{eq:limsupE-TGV} follows.
\end{itemize}
The $\Gamma$-convergence of   \((G_{\alpha_j})_{j\in\NN}\)  
to \(G_{\bar\alpha}\) in \(L^1(\Omega)\) is then a direct consequence of \eqref{eq:liminfE-TGV} and \eqref{eq:limsupE-TGV}.
\end{proof}

As a consequence of the previous result, we provide a characterization of the unique minimizer $u_{\bar\alpha}$ of $G_{\bar\alpha}$. 
\begin{corollary}
\label{cor:min-TGV}
Under the same assumptions of Lemma~\ref{lem:gammaTGV}, let \(u_{\bar \alpha}:=\argmin_{u\in L^1(\Omega)} G_{\bar\alpha}[u]\) for \(\bar \alpha\in [0,+\infty]^2\). Then,
\begin{equation}
\label{eq:minim-TGV}
u_{\bar\alpha}=\begin{cases}
u_{\alpha}&\text{if }\bar\alpha=\alpha\in(0,+\infty)^2\\
u_\eta&\text{if }\bar\alpha_0=0\text{ or }\bar\alpha_1=0,\\
\langle u_\eta \rangle_\Omega&\text{if }\bar\alpha_0=\bar\alpha_1=+\infty.
\end{cases}
\end{equation}
Additionally, when just one among $\bar\alpha_0$ and $\bar\alpha_1$ is infinite, then $\langle u_{\bar\alpha}\rangle_\Omega=\langle u_{\eta}\rangle_\Omega$. In these regimes, if additionally $u_\eta=\langle u_{\eta}\rangle_\Omega$, then $u_{\bar\alpha}=\langle u_{\bar\alpha}\rangle_\Omega$.
\end{corollary}
\begin{proof}
 The first claim follows directly from Lemma~\ref{lem:gammaTGV}. We show the second statement only in the case in which $\bar{\alpha}_0=\infty$ and $\bar{\alpha}_1$ is finite, being the case in which $\bar{\alpha}_1=\infty$  analogous. The characterization of minimizers is then a consequence of  the orthogonality property in \eqref{eq:affine-proj} which, in turn, yields $$G_{\infty,\bar{\alpha}_1}(u)=\int_\Omega |\langle u-u_\eta\rangle_\Omega|^2 \dd x +\int_\Omega \left[\left(u-\langle u\rangle_\Omega\right)-\left(u_\eta-\langle u_\eta\rangle_\Omega\right)\right]^2 \dd x+\bar{\alpha}_1|D^2 \left(u-\langle u\rangle_\Omega\right)|$$ for every $u\in BH(\Omega)$. 
\end{proof}

\begin{lemma}\label{lem:recseq-TGV}
Let \(\Omega\subset \RR^2\) be a bounded, Lipschitz domain and let
 \((G_{\bar \alpha})_{\bar\alpha\in[0,+\infty]}\) be
 the family of  functionals
 introduced in
Lemma~\ref{lem:gammaTGV}. Given   \(\bar
\alpha\in [0,\infty]^2\), set \(u_{\bar\alpha}:= \argmin_{u\in
L^1(\Omega)}  G_{\bar \alpha}[u]\).  Then, there exists a sequence of pairs 
of positive numbers,
 \((\alpha_j)_{j\in\NN} \subset (0,+\infty)^2\), such that \(\alpha_j\to
\bar \alpha\)
in \([0,+\infty]^2\)   as \(j\to\infty\) and 
\begin{equation}\label{eq:strongL2-TGV}
\begin{aligned}
\lim_{j\to\infty} \int_\Omega |u_{\alpha_j} - u_{\bar\alpha}|^2\dd
x=0,
\end{aligned}
\end{equation}
where   \(u_{\alpha_j}:= \argmin_{u\in
L^1(\Omega)}  G_{ \alpha_j}[u] \) for all \(j\in\NN\).\end{lemma}

\begin{proof}
With the same notation as in the proof of Lemma~\ref{lem:gammaTGV}, we detail the argument for each case separately.
\begin{itemize}
\item[(i)] If $\bar{\alpha}=\alpha\in (0,+\infty)^2$, then the statement follows directly by choosing $\alpha_j=\alpha$ for every $j$.
\item[(ii)]
 If \(\bar\alpha_0=0\), then $u_{\bar{\alpha}}=u_\eta$ and $G_{\bar{\alpha}}[u_{\bar{\alpha}}]=G_{\bar{\alpha}}[u_{\eta}]=0$.
 In view of Lemma~\ref{lem:gammaTGV}, there exists a sequence $(u_\eta^j)_{j\in\mathbb{N}}\subset L^1(\Omega)$ such that
 $$\limsup_{j\to+\infty}G_{\alpha_j}[u^j_\eta]\leq G_{\bar\alpha}[u_\eta].$$
 
 Hence, for any sequence $(\alpha_j)_{j\in \mathbb{N}}\subset (0,+\infty)^2$ satisfying $\alpha_j\to\bar{\alpha}$,  the minimality of $u_{\alpha_j}$ yields
 \begin{equation*}
\begin{aligned}
  \limsup_{j\to\infty}
  \bigg(
\int_\Omega |u_{\alpha_j} - u_{\eta}|^2\dd
x+ TGV_{(\alpha_j)_0,(\alpha_j)_1}(u_{\alpha_j},\Omega)\bigg)&=\limsup_{j\to\infty}
G_{\alpha_j}[u_{\alpha_j}]\\&\leq  \limsup_{j\to\infty} G_{\alpha_j}[u^j_\eta]\leq
G_{\bar\alpha}[u_\eta]=0.
\end{aligned}
\end{equation*}
Thus, we infer \eqref{eq:strongL2-TGV}.
\item[(iii)] If $\bar\alpha_0=+\infty$ and $\bar\alpha_1=\alpha_1\in(0,+\infty)$, then $u_{\bar\alpha}\in BH(\Omega)$. For every sequence $(\alpha_0^j)_{j\in\mathbb{N}}$ such that $\alpha_0^j\to+\infty$ as $j\to+\infty$, setting $\alpha_j:=(\alpha_0^j,\alpha_1)$, from the minimality of $u_{\alpha_j}$ and choosing $\nabla u_{\bar\alpha}$ as a competitor in the definition of $TGV$, we find
\begin{align*}
&G_{\alpha_j}[u_{\alpha_j}]\leq G_{\alpha_j}[u_{\bar\alpha}]\leq\int_\Omega|u_{\bar\alpha}-u_\eta|^2\dd x+\alpha_1|D^2 u_{\bar\alpha}|(\Omega)=G_{\bar\alpha}[u_{\bar\alpha}].
\end{align*}
By the fundamental theorem of $\Gamma$-convergence (see \cite[Corollary~7.20 and  Theorem~7.8]{Da93}), the equi-coerciveness of the functionals $G_{\alpha_j}$ together with the uniqueness of minimizers yields 
that $u_{\alpha_j}\weakly u_{\bar{\alpha}}$ weakly in $L^2(\Omega)$. Property
\eqref{eq:strongL2-TGV} follows then by arguing as in item $(iii)$ in the first part of the proof of Lemma~\ref{lem:gammaTGV} and using the continuous embedding $BV(\Omega) \subset L^2(\Omega)$. 
\item[(iv)] If $\bar\alpha_0\in(0,+\infty]$ and $\bar\alpha_1=0$, then $u_{\bar\alpha}=u_\eta$. Let $(u_\eta^k)_{k\in\mathbb{N}}\subset C^\infty_c(\Omega)$ be such that $u_\eta^k\to u_\eta$ strongly in $L^2(\Omega)$. For every sequence $(\alpha_j)_{j\in\mathbb{N}} \subset (0,+\infty)^2$ satisfying $\alpha_j\to\bar\alpha$, we obtain
from the minimality of $u_{\alpha_j}$  that
\begin{align*}
G_{\alpha_j}[u_{\alpha_j}]\leq G_{\alpha_j}[u_\eta^k]\leq \int_\Omega |u_\eta^k-u_\eta|^2\,\dd x+(\alpha_1)_j\int_\Omega |\nabla^2 u_\eta^k|\,\dd x,
\end{align*}
where the latter inequality follows by choosing $\nabla u_\eta^k$ as a competitor in the definition of $TGV$. Thus, 
$$\limsup_{j\to+\infty}G_{\alpha_j}[u_{\alpha_j}]\leq \int_{\Omega}|u_\eta-u_\eta^k|^2\,\dd x$$
for every $k\in\mathbb{N}$. Passing to the limit as $k\to+\infty$, we infer that
$$\limsup_{j\to+\infty}G_{\alpha_j}[u_{\alpha_j}]=0.$$
In turn, this implies \eqref{eq:strongL2-TGV}.
\item[(v)] If $\bar\alpha_0=\alpha_0\in(0,+\infty)$ and $\bar\alpha_1=+\infty$, then $u_{\bar\alpha}\in BV(\Omega)$. For $(\alpha_1^j)_{j\in\mathbb{N}}\subset (0,+\infty)$ such that $\alpha_1^j\to+\infty$, and setting $\alpha_j:=(\alpha_0,\alpha_1^j)$, we deduce that 
\begin{align*}
G_{\alpha_j}[u_{\alpha_j}]\leq G_{\alpha_j}[u_{\bar\alpha}]\leq \int_{\Omega}|u_{\bar\alpha}-u_\eta|^2\dd x+\alpha_0|D u_{\bar{\alpha}}-m_{\mathcal{E}}(Du_{\bar\alpha})|(\Omega)=G_{\bar\alpha}[u_{\bar\alpha}].
\end{align*}
By the fundamental theorem of $\Gamma$-convergence, we infer that
$u_{\alpha_j}\to u_{\bar\alpha}$ strongly in $L^1(\Omega)$ and that $G_{\alpha_j}[u_{\alpha_j}]\to G_{\bar\alpha}[u_{\bar\alpha}].$ On the other hand, letting $u^\ast_{\alpha_j}$ be defined as in \eqref{eq:ujstar} with $u_j$ replaced by $u_{\alpha_j}$,
the same argument as in item $(v)$ in the first part of the proof of Lemma~\ref{lem:gammaTGV} yields 
$u_{\alpha_j}\weaklystar u$ weakly-\(\star\) in $BV(\Omega)$, and $u^\ast_{\alpha_j}\to u^\ast$ strongly in $BD(\Omega)$ with $u^\ast$ affine. By combining the above convergences, we deduce 
\begin{align*}
G_{\bar\alpha}[u_{\bar\alpha}]&\leq \int_\Omega|u_{\bar{\alpha}}-u_\eta|^2\dd x+\alpha_0|Du_{\bar{\alpha}}-u^\ast|(\Omega)\\
&\leq \liminf_{j\to +\infty}\int_\Omega|u_{\alpha_j}-u_\eta|^2\dd x+\alpha_0|Du_{\alpha_j}-u^\ast_{\alpha_j}|(\Omega)\leq \lim_{j\to +\infty}G_{\alpha_j}[u_{\alpha_j}]=G_{\bar\alpha}[u_{\bar\alpha}],
\end{align*}
where the first inequality follows by the definition of $m_{\mathcal E}$, cf. \eqref{def:m-E}, whereas the second one is a consequence of the lower semicontinuity of the $L^2$-norm with respect to the weak $L^2$-convergence, as well as of the lower semicontinuity of the total variation with respect to the weak-\(\star\) convergence of measures.

\item[(vi)] If $\bar\alpha_0=\bar\alpha_1=+\infty$, then $u_{\bar\alpha}$ is affine. Thus, for every sequence $(\alpha_j)_{j\in\mathbb{N}}
\subset (0,+\infty)^2$ satisfying $\alpha_j\to\bar\alpha$,
$$G_{\alpha_j}[u_{\alpha_j}]\leq G_{\alpha_j}[u_{\bar\alpha}]=\int_\Omega|u_{\bar\alpha}-u_\eta|^2\dd x=G_{\bar\alpha}[u_{\bar\alpha}].$$
Property \eqref{eq:strongL2-TGV} is once again obtained arguing by the fundamental theorem of $\Gamma$-convergence, as in  (iii).
\end{itemize}
\end{proof}
In view of the lemmas above, we obtain the following characterization of the lower semicontinuous envelope of $J$.

\begin{lemma}\label{lem:Ilsc-TGV}
Let \(\Omega\subset \RR^2\) be a bounded, Lipschitz domain, and
let \(J:(0,+\infty)^2\to[0,+\infty)\) be the function defined
in \eqref{eq:defJ-TGV}. Then, the extension \(\widehat J:[0,+\infty]^2\to[0,+\infty]\)
of \(J\)  to the closed interval \([0,+\infty]^2 \)  defined for \(\bar\alpha\in [0,+\infty]^2\) by%
\begin{equation}
\label{eq:Isc-TGV}
\begin{aligned}
\widehat J(\bar\alpha):= \inf\Big\{\liminf_{j\to\infty} J(\alpha_j)\!:\,
(\alpha_j)_{j\in\NN}\subset (0,+\infty)^2, \, \alpha_j \to \bar\alpha
\text{ in } [0,+\infty]^2\Big\},
\end{aligned}
\end{equation}
satisfies 
\begin{equation}
\label{eq:repIsc-TGV}
\begin{aligned}
\widehat J(\bar\alpha) = \begin{cases}
J(\alpha)=\Vert u_\alpha - u_c\Vert^2_{L^2(\Omega)} & \text{if
} \bar \alpha= \alpha\in(0,+\infty)^2,\\
\Vert u_\eta - u_c\Vert^2_{L^2(\Omega)} & \text{if } \bar\alpha_0=0 \text{ or }\bar\alpha_1=0,\\
\Vert \langle u_\eta \rangle -u_c\Vert^2_{L^2(\Omega)}  & \text{if }\bar\alpha_0=\bar\alpha_1=+\infty,\\
\Vert u_{\bar{\alpha}}-u_c\Vert^2_{L^2(\Omega)}\text{ with }\langle u_{\bar\alpha}\rangle= \langle u_{\eta}\rangle & \text{otherwise},
\end{cases}
\end{aligned}
\end{equation}
where $u_{\bar{\alpha}}$ is the unique minimizer of $G_{\bar\alpha}$, cf. Corollary~\ref{cor:min-TGV}.
\end{lemma}
\begin{proof}
We first note that the function  \(\widehat J \) in \eqref{eq:Isc-TGV} is lower-semicontinuous on \([0,+\infty]^2\)
and  \(\widehat J \leq J\) in \((0,+\infty)^2\).
Next, we denote by \(\widetilde J\)  the function on \([0,+\infty]^2\)
defined by the right-hand
side of \eqref{eq:repIsc-TGV}, and observe  that
\begin{equation*}
\begin{aligned}
\widetilde J (\bar\alpha) = \Vert u_{\bar\alpha} - u_c\Vert^2_{L^2(\Omega)},
\end{aligned}
\end{equation*}
where  \(u_{\bar\alpha}:= \argmin_{u\in
L^1(\Omega)} G_{\bar \alpha}(u)\) is given by \eqref{eq:minim-TGV}.
We want to show that \(\widehat J \equiv \widetilde J\). By Lemma~\ref{lem:recseq-TGV}, for all $\bar \alpha \in [0,+\infty]^2$ there exists a sequence 
 $(\alpha_j)_{j\in\NN}\subset (0,+\infty)$ such that $\alpha_j \to \bar \alpha$
and for which we have
\begin{equation}\label{eq:byS2}
\begin{aligned}
\widetilde J (\bar \alpha) =\Vert u_{\bar\alpha} - u_c\Vert^2_{L^2(\Omega)}
= \lim_{j\to\infty} \Vert u_{\alpha_j} - u_c\Vert^2_{L^2(\Omega)}
= \lim_{j\to\infty} J(\alpha_j). 
\end{aligned}
\end{equation}
Thus, \(\widetilde J(\bar \alpha)\geq \widehat J (\bar \alpha)\) for all $\bar \alpha \in [0,+\infty]^2$. It remains to prove the opposite inequality. For this, we distinguish several cases as in the proofs of Lemma \ref{lem:recseq-TGV}:

\begin{itemize}
\item[(i)] If $\bar{\alpha}=\alpha\in (0,+\infty)^2$, let $(\alpha_j)_{j\in\NN}\subset (0,+\infty)$ be any sequence such that $\alpha_j
\to \alpha$. As argued before, we observe that
the uniform bounds in \(BV(\Omega)\) proved in Lemma~\ref{lem:gammaTGV} assert that
\((G_{\alpha_j})_{j\in\NN}\)
  is an equi-coercive
sequence in \(L^1(\Omega)\). Thus, as before, by well-known properties of
\(\Gamma\)-convergence  on the convergence of minimizing sequences and energies
(see \cite[Corollary~7.20 and  Theorem~7.8]{Da93}), together with the uniqueness of minimizers
of \(G_{\alpha_j}\) and \(G_\alpha\), we have that  \(u_{\alpha_j}
\weakly
u_\alpha \) weakly-$\star$ in $BV(\Omega)$
and \(\lim_{j\to\infty} G_{\alpha_j} [u_{\alpha_j}] = G_{\alpha}
[u_{\alpha}]\).
In particular,
 \(u_{\alpha_j} \weakly
u_{\alpha}\) weakly in $L^2(\Omega)$.
Hence,
\begin{equation*}
\begin{aligned}
\widetilde J(\alpha) =\Vert u_{\alpha} - u_c\Vert^2_{L^2(\Omega)}
\leq \liminf_{j\to\infty} \Vert u_{\alpha_j} - u_c\Vert^2_{L^2(\Omega)}
= \liminf_{j\to\infty} J(\alpha_j). 
\end{aligned}
\end{equation*}
Taking the infimum of all such sequences 
 $(\alpha_j)_{j\in\NN}\subset (0,+\infty)$, we conclude that
 \(\widetilde J(\alpha)\leq \widehat J (\alpha)\). 
 
\item[(ii)] If $\bar \alpha_0 = 0$, we obtain by the corresponding case of Lemma \ref{lem:recseq-TGV}  that for any sequence $(\alpha_j)_{j\in \mathbb{N}}\subset (0,+\infty)^2$ such that $\alpha_j \to \bar \alpha$, we have
\begin{equation}\label{eq:limGzero}0 \leq \limsup_{j\to+\infty}G_{\alpha_j}[u_{\alpha_j}]=0,\end{equation}
which implies $u_{\alpha_j} \to u_\eta$ strongly in $L^2(\Omega)$, and in turn $\lim_{j \to \infty} J(\alpha_j) = \widetilde J(\bar \alpha)$. Thus, taking the infimum over all such sequences, we conclude that
\(\widehat J (\bar\alpha) = \widetilde J(\bar\alpha)\).
\item[(iii)] If $\bar\alpha_0=+\infty$ and $\bar\alpha_1=\alpha_1\in(0,+\infty)$, the thesis follows by observing that the same argument as in (iii) of Lemma \ref{lem:recseq-TGV} still holds for any sequence $(\alpha_0^j,\alpha_1^j)_{j\in \N}$ with $\alpha_0^j\to +\infty$ and $\alpha_1^j\to \alpha_1$ as $j\to +\infty$.
\item[(iv)]  If $\bar\alpha_0\in(0,+\infty]$ and $\bar\alpha_1=0$, we can proceed exactly as in  (ii) to conclude that  for any sequence $(\alpha_j)_{j\in \mathbb{N}}\subset (0,+\infty)^2$ such that $\alpha_j \to \bar \alpha$, we again have \eqref{eq:limGzero} by the corresponding case of Lemma~\ref{lem:recseq-TGV}.
\item[(v)] Analogously to (iii), if $\bar\alpha_0=\alpha_0\in(0,+\infty)$ and $\bar\alpha_1=+\infty$, the statement is a consequence of the fact that the same argument as in (v) of Lemma \ref{lem:recseq-TGV} still holds for any sequence $(\alpha_0^j,\alpha_1^j)_{j\in \N}$ with $\alpha_0^j\to \alpha_0$ and $\alpha_1^j\to +\infty$ as $j\to +\infty$.
\item[(vi)] If $\bar\alpha_0=\bar\alpha_1=+\infty$, by the proof item (vi) of Lemma \ref{lem:recseq-TGV}, we have for any sequence $(\alpha_j)_{j\in \mathbb{N}}\subset (0,+\infty)^2$ with $\alpha_j \to \bar \alpha$ that $G_{\alpha_j}[u_{\alpha_j}] \leq G_{\bar \alpha}[u_{\bar \alpha}]$, which, analogously to item (iii) of Lemma~\ref{lem:recseq-TGV}, provides that \(u_{\alpha_j} \weakly
u_{\alpha}\) weakly in $L^2(\Omega)$, and this in turn allows us to conclude as in item (i).\qedhere
\end{itemize}
\end{proof}

We are now in a position to prove Theorem~\ref{thm:onalpha-TGV}.

\begin{proof}[Proof of Theorem ~\ref{thm:onalpha-TGV}]
The proof is subdivided into three steps.

\medskip

{\it{\uline{Step~1.}}} We prove that if condition \textit{i)} in  the statement holds, namely 
$$ TGV_{\hat\alpha_0,\hat\alpha_1}(u_\eta,\Omega)- TGV_{\hat\alpha_0,\hat\alpha_1}(u_c,\Omega) >0$$ for some $\hat\alpha\in (0,+\infty)^2$, then
there exists \(\bar\alpha\in(0,+\infty)^2\) such that
\begin{equation}
\label{eq:notat0-TGV}
\begin{aligned}
\Vert u_{\bar\alpha} - u_c\Vert^2_{L^2(\Omega)}<\Vert u_\eta - u_c\Vert^2_{L^2(\Omega)}.
\end{aligned}
\end{equation}
From the convexity of the $TGV$-seminorm, arguing as in the proof of \eqref{eq:notat0}, we infer that
\begin{align*}
\|u_\eta-u_c\|_{L^2(\Omega)}^2-\|u_{\alpha}-u_c\|_{L^2(\Omega)}^2
\leq TGV_{\alpha_0,\alpha_1}(u_\alpha,\Omega)-TGV_{\hat\alpha_0,\hat\alpha_1}(u_c,\Omega)
\end{align*}
for every $\alpha\in (0,+\infty)^2$. Choosing $\alpha=\lambda\hat\alpha$, and denoting $u_{\lambda(\hat\alpha)}$ by $u_\lambda$, for simplicity, we find that
\begin{align*}
\|u_\eta-u_c\|_{L^2(\Omega)}^2-\|u_{\lambda}-u_c\|_{L^2(\Omega)}^2\leq \lambda \left(TGV_{\hat\alpha_0,\hat\alpha_1}(u_\lambda,\Omega)-TGV_{\hat\alpha_0,\hat\alpha_1}(u_c,\Omega)\right)
\end{align*}
for every $\lambda\in (0,+\infty)$. By the proof of case (ii) of Lemma~\ref{lem:recseq-TGV} and by Corollary \ref{cor:min-TGV}, it follows that, up to (non-relabelled) subsequences, $u_\lambda\to u_\eta$ strongly in $L^2(\Omega)$ as $\lambda\to 0$. Fix $\varepsilon>0$; by the lower-semicontinuity of the $TGV$-seminorms with respect to the strong $L^2$-convergence, we conclude that 
\begin{align*}
TGV_{\hat\alpha_0,\hat\alpha_1}(u_\lambda,\Omega)\geq TGV_{\hat\alpha_0,\hat\alpha_1}(u_\eta,\Omega)-\varepsilon(TGV_{\hat\alpha_0,\hat\alpha_1(u_\eta,\Omega)}-TGV_{\hat\alpha_0,\hat\alpha_1}(u_c,\Omega))
\end{align*}
for $\lambda$ small enough. Thus, 
\begin{align*}
\|u_\eta-u_c\|_{L^2(\Omega)}^2-\|u_\lambda-u_c\|_{L^2(\Omega)}^2\geq \lambda(TGV_{\hat\alpha_0,\hat\alpha_1}(u_\eta,\Omega)-TGV_{\hat\alpha_0,\hat\alpha_1}(u_c,\Omega))(1-\varepsilon)
\end{align*}
for $\lambda$ small enough. This implies that there exists $\bar\lambda\in (0,+\infty)$ for which
$$\|u_\eta-u_c\|_{L^2(\Omega)}^2>\|u_\lambda-u_c\|_{L^2(\Omega)}^2.$$
The preceding estimate yields the thesis by choosing $\bar\alpha=\bar\lambda(\hat\alpha_0,\hat\alpha_1)$.

\medskip
{\it{\uline{Step~2.}}} We prove that if condition \textit{ii)}
in  the statement holds, (i.e., $ \Vert u_\eta - u_c\Vert^2_{L^2(\Omega)}
<\Vert\langle  u_\eta\rangle - u_c\Vert^2_{L^2(\Omega)}$), then
there exits \(\bar\alpha\in(0,+\infty)^2\) such that
\begin{equation}
\label{eq:notati-TGV}
\begin{aligned}
\Vert u_{\bar\alpha} - u_c\Vert^2_{L^2(\Omega)}<\Vert\langle  u_\eta \rangle- u_c\Vert^2_{L^2(\Omega)}.
\end{aligned}
\end{equation}
In view of Step 1, 
\begin{align*}
\lim_{\lambda\to 0}\|u_\lambda-u_c\|_{L^2(\Omega)}=\|u_\eta-u_c\|_{L^2(\Omega)}< \|\langle u_\eta\rangle-u_c\|_{L^2(\Omega)}.
\end{align*}
By the proof of case (vi) of Lemma~\ref{lem:recseq-TGV} and by Corollary \ref{cor:min-TGV}, we obtain the existence of $\bar\lambda\in (0,+\infty)$ for which 
$$\|u_{\bar\lambda}-u_c\|_{L^2(\Omega)}<\|\langle u_\eta\rangle-u_c\|_{L^2(\Omega)}.$$
The claim follows by choosing $\bar\alpha=\bar\lambda(\hat\alpha_0,\hat\alpha_1)$.

\medskip
{\it{\uline{Step~3.}}} 
We conclude the proof by establishing the bounds on the parameters stated in Theorem ~\ref{thm:onalpha-TGV}.  From the lower semicontinuity of $\widehat J$, we infer that there exists  $\alpha^\ast\in [0,+\infty]^2$ where the minimum value is attained. By Corollary~\ref{cor:min-TGV} and by the previous steps, $\alpha^\ast$ satisfies \eqref{eq:optalpha-generalTGV} and 
\begin{equation}\label{eq:minI'-TGV}
\begin{aligned}
\widehat J(\alpha^*)=\min_{\bar\alpha \in [0,+\infty]^2} \widehat J(\bar\alpha).
\end{aligned}
\end{equation}

To prove the existence of the lower bound \(c_\Omega\), we argue by contradiction. We first assume that there exists a sequence
\((\alpha^*_j)_{j\in\NN}\subset(0,+\infty)^2\) such that $\alpha^\ast_j\to 0$ as $j\to +\infty$,
and \eqref{eq:minI'-TGV} holds for \(\alpha^*=\alpha^*_j\) for all $j\in \mathbb{N}$. In view of the lower semi-continuity of \(\widehat J\) on \([0,+\infty]^2\),%
\begin{equation*}
\begin{aligned}
\min_{\bar\alpha \in [0,+\infty]^2}
\widehat J(\bar \alpha) \leq \widehat J(0) \leq \liminf_{j\to\infty}
\widehat J( \alpha^*_j) =\min_{\bar\alpha \in [0,+\infty]^2}
\widehat J(\bar \alpha),
\end{aligned}
\end{equation*}
which is false by  \eqref{eq:notat0-TGV}. This proves the existence of a constant $\hat c_\Omega$ such that $|\alpha^\ast|\geq \hat c_\Omega$ for every minimizer $\alpha^\ast$ of \(\widehat J\). 
The existence of the constant $c_\Omega$ as in the statement of the theorem follows by observing that the above argument can be repeated by considering sequences $(\alpha_j^\ast)_{j\in\mathbb{N}}$ for which just one of the entries converges to zero.

The bound from above on $\min\{\alpha^\ast_0,\alpha^\ast_1\}$ follows directly by  Proposition~\ref{prop:loc-affine}. In fact, from \eqref{eq:l2boundforconst}, we infer the existence of a constant $C_\Omega$ such that $u_{\alpha^\ast}$ is affine if $C_\Omega\|u_\eta\|_{L^2(\Omega)}<\min\{\alpha^\ast_0,\alpha^\ast_1\}$. Now, assume by contradiction that there exists a sequence \((\alpha^*_j)_{j\in\NN}\subset(0,+\infty)^2\) such that both entries of $\alpha^\ast_j$ blow up to infinity as $j\to +\infty$,
and \eqref{eq:minI'-TGV} holds for \(\alpha^*=\alpha^*_j\) for all $j\in \mathbb{N}$. 
Using, once again, the lower semi-continuity of \(\widehat J\) on \([0,+\infty]^2\),  
we find that
\begin{equation*}
\begin{aligned}
\min_{\bar\alpha \in [0,+\infty]^2}
\widehat J(\bar \alpha) \leq \widehat J(+\infty,+\infty) \leq \liminf_{j\to\infty}
\widehat J( \alpha^*_j) =\min_{\bar\alpha \in [0,+\infty]^2}
\widehat J(\bar \alpha),
\end{aligned}
\end{equation*}
which is false by
Corollary~\ref{cor:min-TGV} and \eqref{eq:notati-TGV}.
\end{proof}

\subsection{The \texorpdfstring{\((\scrL\!\scrS)_{{TGV-Fid}_\omega}\)}{TGV-Fid-w} learning scheme}\label{subs:wfidTGV}

Given a dyadic square \(L\subset Q\) and \(\lambda\in(0,\infty)\), we have 
\begin{equation*}
\begin{aligned}
 &\argmin\left\{\lambda\int_{L}|u_\eta-u|^2\dd x+ TGV_{1,1}(u,L)\!:\,u\in
BV(L)\right\}\\ &\qquad = \argmin\left\{\int_{L}|u_\eta-u|^2\dd x+  TGV_{\frac1\lambda,\frac1\lambda}(u,L)\!:\,u\in
BV(L)\right\}.
\end{aligned}
\end{equation*}
The analysis in Subsections~\ref{sect:TGVL3}--\ref{sub-l1tgv} applies also to the weighted-fidelity
learning scheme and yields Theorem \ref{thm:FIDomegaTGV}. As before, the previous existence theorem  holds true under
any stopping criterion  for the refinement of the admissible partitions
provided that the training data satisfies suitable conditions. We summarize the situation in the next result, which follows directly by the discussions in the previous subsection, in particular Corollary \ref{thm:onalpha-TGV1}.

\begin{theorem}[Equivalence between box constraint and
stopping criterion]\label{thm:equivFTGV}
Consider the learning scheme \((\scrL\!\scrS)_{{TGV-Fid}_\omega}\)
in \eqref{lsTGVfidomega}. The two following conditions hold:\begin{itemize}
\item[(a)] If we replace  \eqref{eq:alpha-Lg-TGV} 
 by \eqref{eq:box-min-TGV}, then there exists a
stopping criterion \((\scrS)\)  for
the refinement of the admissible partitions as in Definition~\ref{def:stop}.

\item[(b)] Assume that
there exists a
stopping criterion \((\scrS)\)  for
the refinement of the admissible partitions as in Definition~\ref{def:stop}
such that the training data
  satisfies for all \(L\in\medcup_{\scrL\in\bar\scrP} \scrL\), with
  \(\bar\scrP\) as in Definition~\ref{def:stop}, the   conditions
\begin{itemize}[leftmargin=12mm]
\item[(i)] $TGV_{\alpha_0, \alpha_1}(u_c,L) < TGV_{\alpha_0, \alpha_1}(u_\eta,L)$,

\item[(ii)] $\displaystyle \Vert u_\eta - u_c\Vert^2_{L^2(L)}
<\Vert\langle  u_\eta\rangle_L - u_c\Vert^2_{L^2(L)} $.
\end{itemize}
Then, there exist \(c_0, c_1\in\RR^+\) such that the optimal
solution \(u^*\) provided by  \((\scrL\!\scrS)_{{TGV-Fid}_\omega}\)
with  \(\scrP\) replaced by  \(\bar \scrP\)
coincides with the optimal solution
\(u^*\) provided by \((\scrL\!\scrS)_{{TGV-Fid}_\omega}\)
with \eqref{eq:alpha-Lg-TGV} 
 replaced by \eqref{eq:box-min-TGV}.
   \end{itemize}
\end{theorem}

\section{Numerical Treatment and Comparison of the learning schemes  \texorpdfstring{\((\scrL\!\scrS)_{{TV\!}_{\omega}}\)}{TV-w},  \texorpdfstring{\((\scrL\!\scrS)_{{TV\!}_{\omega_\epsi}}\)}{TV-we},
 \texorpdfstring{\((\scrL\!\scrS)_{{TV-Fid}_\omega}\)}{TV-Fid-w}, and  \texorpdfstring{\((\scrL\!\scrS)_{{TGV-Fid}_\omega}\)}{TGV-Fid-w}}
\label{sect:numerics}

\subsection{Common numerical framework for all schemes}\label{sect:commonnumerics}
The focus of our article is on the use of space-dependent weights and, from the numerical point of view, our schemes  require  addressing  weights that are piecewise constant on dyadic partitions. This stands in contrast to most previous approaches for optimizing space-dependent parameters, which in most cases hinge on $H^1$-type penalizations of the weights, as done in \cite{ChDeSc16, HiRaWuLan17} for TV,  \cite{HiPaRaSu22} for TGV and \cite{PaPaRaVi22} for some more general convex regularizers. The piecewise constant setting makes it possible to work in a modular fashion, building upon any numerical methods that are able to compute solutions to denoising with a weight (Level 2) and finding constant optimal regularization parameters (Level 3).

In our numerical examples, we have used a basic first-order finite difference discretization of the gradient and symmetrized gradient, on the regular grid arising from the discrete input images. For solving $TV$ regularized denoising, either with constant or varying weights, we have opted for the standard primal-dual hybrid gradient (PDHG) descent scheme of \cite{ChPo11}. The optimization for optimal constant parameters $\alpha$ of Level 3 is done with the `piggyback' version of the same algorithm, which has been proposed in \cite{ChPo21} to learn finite difference discretizations of $TV$ with a high degree of isotropy, and further analyzed under smoothness assumptions on the energies in \cite{BoChPo22}. Essentially, it consists in evolving an adjoint state along with the main variables, to keep track of the sensitivity of the solution with respect to parameters. We remark that such sensitivity analysis in principle requires not just first but second derivatives of the lower-level cost functions involved, in our case TV or TGV denoising involving weighted $\ell^1$ norms and their Fenchel conjugates, which are only componentwise piecewise smooth. In any case, as already observed in \cite[Appendix A]{ChPo21}, we do achieve an adequate performance in practice. It is worth mentioning that other methods to handle the bilevel optimization problems of Level~3 in a nonsmooth setting have been introduced in \cite{De23, DeVi22, BrChHo22}. One could also use these in our subdivision scheme within Algorithm \ref{alg:subdivision} below, and in fact the authors of the cited papers optimize for adaptive weights on regular dyadic grids refined uniformly. In contrast, our focus here is on the adaptive subdivision scheme.

These PDHG methods are based on considering the discrete optimization problems
\[\min_{x \in \Xcal}\, \Gcal(x) + \Fcal(Ky)\]
through their corresponding saddle point formulation
\[\max_{y \in \Ycal} \min_{x \in \Xcal}\, \langle y, Kx \rangle_{\ell^2} + \Gcal(x) - \Fcal^\ast(y),\]
with $\Gcal$ representing the differentiable fidelity term and $\Fcal^\ast$ being the projection onto a convex set, arising as the Fenchel conjugate of an $\ell^1$-type norm. Denoting by $W=\R^{nm}$ the space of discrete scalar-valued functions,  these read
in the $TV$ case as\begin{gather*}\Xcal=W,\ \Ycal=W^2,\ K=\nabla,\ \Gcal(u)=\lambda \sum_{ij} \big( u^{ij} - u^{ij}_\eta \big)^2,\ \text{ and }\\ \Fcal^\ast(p)=\Ical_{Q_{TV}} \text{ with }Q_{TV} = \big\{ p \in \Ycal \,|\, (p^{ij}_1)^2 + (p^{ij}_2)^2 \leq \alpha \text{ for all }i,j\big\}.\end{gather*}
For the TGV case, following the approach used in \cite{Br14} and \cite{BrHo15}, we have used
\begin{gather*}\Xcal=W \times W^2,\ \Ycal=W^2 \times W^3,\ K=\begin{pmatrix} \nabla u & - \mathrm{Id} \\ 0 &  \Ecal \end{pmatrix},\\ \Gcal(u,v)=\lambda \sum_{ij} \big( u^{ij} - u^{ij}_\eta \big)^2,\ \text{ and } \Fcal^\ast(p)=\Ical_{Q_{TGV}}, \text{ where }\\ Q_{TGV} = \big\{ (p,q) \in \Ycal \,|\, (p^{ij}_1)^2 + (p^{ij}_2)^2 \leq \alpha_0, (q^{ij}_{11})^2 + 2(q^{ij}_{12})^2 + (q^{ij}_{22})^2 \leq \alpha_1\text{ for all }i,j\big\}.\end{gather*}
With this notation and denoting the subgradient by  $\partial$, the PDHG algorithm \cite[Algorithm 1]{ChPo11} can be written as
\begin{equation}\label{eq:primaldual}\begin{cases}y^{k+1}= (\mathrm{Id} + \sigma \partial \Fcal^\ast)^{-1}(y^k + \sigma K \bar{x}^{k}),\\ x^{k+1}=(\mathrm{Id} + \tau \partial \Gcal)^{-1}(x^k - \tau K^\ast y^{k+1}),\\ \bar{x}^{k+1} = x^{k+1} + \theta( x^{k+1} - x^k ),\end{cases}\end{equation}
where the descent parameters satisfy $\sigma \tau \|K\| \leq 1$. In the  $TV$ case, this operator norm of $\nabla$ can be bounded by $\sqrt{8}$  (cf. \cite[Theorem~3.1]{Ch04}), while in the $TGV$ case, we have $\|K\|^2 \leq (17+\sqrt{33})/2$ (cf. \cite[Section 3.2]{Br14}). The piggyback algorithm of \cite{ChPo21, BoChPo22} introduces one adjoint variable for each primal and dual variable above (denoted by $X \in \Xcal$, $Y \in \Ycal$, $U \in W$, $P \in W^2$, $Q \in W^3$) and performs the same kind of updates also on these new variables to optimize the values of a loss function $L$, resulting in
\begin{equation}\label{eq:piggyback}\begin{cases}Y^{k+1}= D\prox_{\sigma \Fcal^\ast}(y^k + \sigma K \bar{x}^{k})\cdot\big[Y^k + \sigma K \big(\bar{X}^{k}+ D_x L(x^k, y^k) \big)\big],\\ 
X^{k+1}=D\prox_{\tau \Gcal}(x^k - \tau K^\ast y^{k+1})\cdot\big[X^k - \tau K^\ast \big(Y^{k+1} + D_y L(x^k, y^k) \big)\big],\\ \bar{X}^{k+1} = X^{k+1} + \theta( X^{k+1} - X^k ),\end{cases}\end{equation}
where $\prox_{\tau \Gcal}=(\mathrm{Id} + \tau \partial \Gcal)^{-1}$ and $\prox_{\sigma F^\ast}=(\mathrm{Id} + \tau \partial F^\ast)^{-1}$ as appearing in \eqref{eq:primaldual}; the latter corresponds to a projection onto $Q_{TV}$ or $Q_{TGV}$ which, as already remarked, is not differentiable on the boundary of these sets.

In our case, we optimize the squared $L^2$ distance to $u_c$ by varying the fidelity parameter $\lambda = 1/\alpha$, so that
\begin{equation}\label{eq:trainingloss}L(u) = \frac12 \sum_{ij} (u^{ij} - u^{ij}_c)^2\ \text{ and }\ D_\lambda \Lcal(\lambda) = \lambda \sum_{ij} \hat{U}^{ij} (\hat{u}^{ij} - u^{ij}_\eta)\ \text{ for }\ \Lcal(\lambda) = L(\hat{u}(\lambda)),\end{equation}
where $\hat{u}$, $\hat{U}$ are the optimal image variable and corresponding adjoint obtained after convergence of \eqref{eq:primaldual} and \eqref{eq:piggyback}. We have then used the derivative $D_\lambda \Lcal$ to update $\lambda$ with gradient descent steps. We have chosen to not use line search, since with the piggyback algorithm evaluations of energy and of gradient for the solution of the lower level problem require a comparable amount of computational effort, that is, either performing \eqref{eq:primaldual} alone or together with \eqref{eq:piggyback} for the same number of lower level steps. We summarize this basic approach in Algorithm \ref{alg:piggybackgd}.

\begin{algorithm}[ht]
\begin{flushleft}
\hspace*{\algorithmicindent} \textbf{Input:} Restrictions of noisy image $u_\eta$ and clean (training) image $u_c$ to a dyadic square,\\ \qquad\qquad\quad initial parameter $\lambda_0$, initial timestep $\zeta$, damping factor $\nu \leq 1$, tolerance $\mathrm{Tol}$.
\\
\end{flushleft}
\begin{algorithmic}
\STATE 1. Set $k=0$.
\WHILE {\big($|\lambda_k - \lambda_{k-1}| > \mathrm{Tol}$ \textbf{or} $k=0$\big)}
\STATE 2. Set $k=k+1$.
\STATE 3. Compute $\hat{u}, \hat{U}$ by running the Piggyback PDHG iterations \eqref{eq:primaldual}-\eqref{eq:piggyback} to convergence.
\STATE 4. Update $\lambda_k = \lambda_{k-1} - \zeta D_\lambda \Lcal(\lambda_k)$, with $D_\lambda \Lcal$ from \eqref{eq:trainingloss}.
\STATE 5. Set $\zeta=\nu \zeta$.
\ENDWHILE
\end{algorithmic}
\caption{Numerical approach to Level 3 of \((\scrL\!\scrS)_{{TV-Fid}_\omega}\) and \((\scrL\!\scrS)_{{TGV-Fid}_\omega}\)}
\label{alg:piggybackgd}
\end{algorithm}

It is worth noting that we are optimizing only on the parameter $\lambda$ in front of the fidelity term. For the TV case and since this algorithm is applied to Level 3 with constant parameters, only the balance between the two energy terms is relevant and finding an optimal $\lambda_L$ is equivalent to finding an optimal $\alpha_L = 1/\lambda_L$, which can then be assembled over all $L$ into a weight $\omega$ for Level 2 of either \((\scrL\!\scrS)_{{TV-Fid}_\omega}\) or \((\scrL\!\scrS)_{{TV}_\omega}\). In the TGV setting, optimizing only over one parameter imposes a restriction, but we have chosen to do so to keep the simple approach of Algorithm \ref{alg:piggybackgd} and avoid more complicated behaviors of the costs when varying both $\alpha_0$ and $\alpha_1$ (or, equivalently, $\lambda$ together with either $\alpha_0$ and $\alpha_1$).

\subsection{Effect of parameter discontinuities in Level 2 of  \texorpdfstring{\((\scrL\!\scrS)_{{TV\!}_{\omega}}\)}{TV-w}, \texorpdfstring{\((\scrL\!\scrS)_{{TV\!}_{\omega_\epsi}}\)}{TV-we} and \texorpdfstring{\((\scrL\!\scrS)_{{TV-Fid}_\omega}\)}{TV-Fid-w}}\mbox{}\\

In Figure \ref{fig:parameterdiscont}, we present an example using large regularization parameters and a symmetric input image to demonstrate the effect of parameter discontinuities in Level 2 of the schemes \((\scrL\!\scrS)_{{TV\!}_{\omega}}\), \((\scrL\!\scrS)_{{TV\!}_{\omega_\epsi}}\) and \((\scrL\!\scrS)_{{TV-Fid}_\omega}\). In the weighted-$TV$ result, a jump in the weight results in a spurious discontinuity in the resulting image. Mollifying the weight smooths the transition slightly, and it shifts it to the side with lower weight. Using a weighted fidelity term does not introduce discontinuities besides those present in the input, but still creates visible artifacts near them.

\begin{figure}[htb]
     \begin{center}
         \mbox{} 
     \hfill
         \raisebox{-0.5\height}{\includegraphics[width=0.315\textwidth]{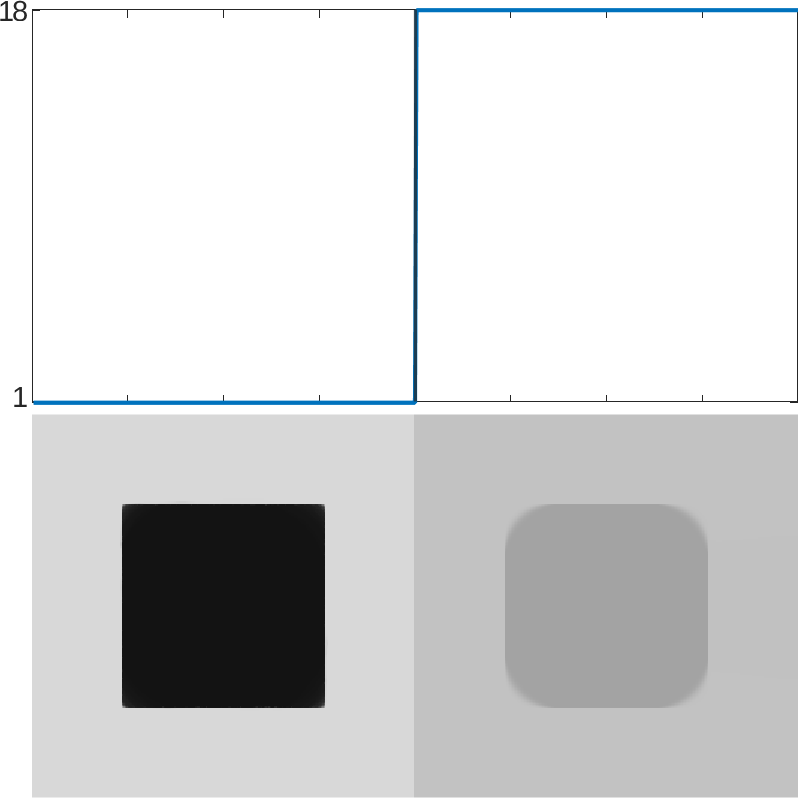}}
     \hspace{0.05cm}
         \raisebox{-0.5\height}{\includegraphics[width=0.315\textwidth]{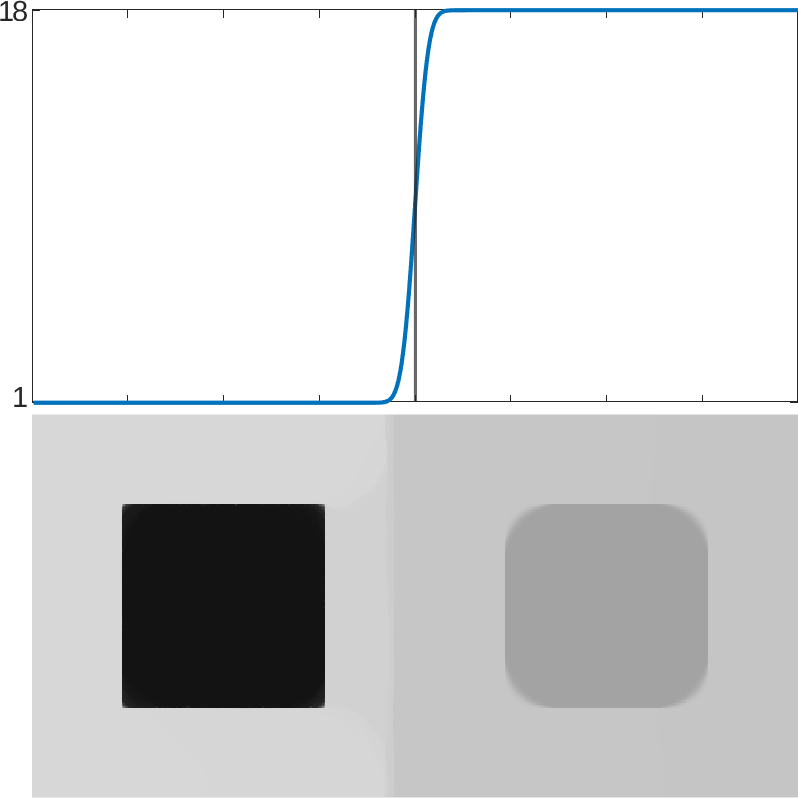}}
     \hspace{0.05cm}
         \raisebox{-0.5\height}{\includegraphics[width=0.315\textwidth]{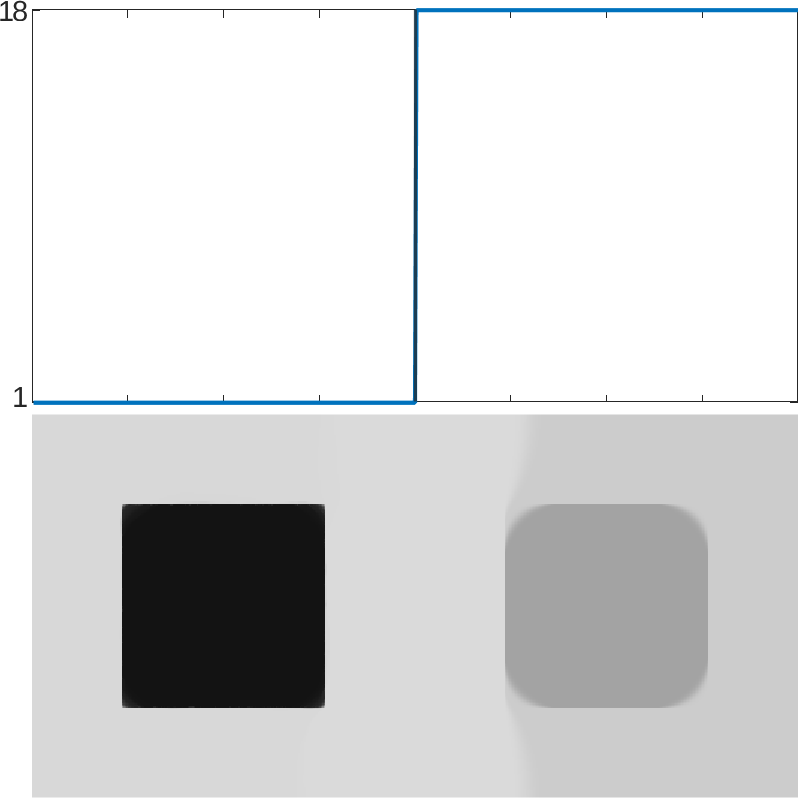}}
         \hfill
         \mbox{}
     \end{center}
     \caption{Oversmoothed denoising with a sharp change of weight and schemes corresponding to Level 2 of \((\scrL\!\scrS)_{{TV\!}_{\omega}}\), \((\scrL\!\scrS)_{{TV\!}_{\omega_\epsi}}\), and
\((\scrL\!\scrS)_{{TV-Fid}_\omega}\), from left to right. Top row: weights $\omega(x)=\omega(x_1)$. Bottom row: results with each denoising scheme and the corresponding (not optimal) weight.}\label{fig:parameterdiscont}
\end{figure}
\subsection{Dyadic subdivision approach to Level 1}\mbox{}\\
\begin{algorithm}
\begin{flushleft}
\hspace*{\algorithmicindent} \textbf{Input:} Noisy image $u_\eta$, clean (training) image $u_c$, subdivision tolerance $\rho$.
\\
\end{flushleft}
\begin{algorithmic}
\STATE 1. Set $\ell=0$ and $\scrL=\{(0,1)^2\}$.
\STATE 2. Compute the constant optimal $\omega_{(0,1)^2}\equiv \lambda_{(0,1)^2}$ or $\omega_{(0,1)^2}\equiv 1/\lambda_{(0,1)^2}$ using the numerical approach to Level 3 described in Section \ref{sect:commonnumerics}. Store the cost at the minimum as $c_{(0,1)^2}$.
\WHILE {$\ell < \ell_{\max}$}
\FOR {all $L \in \scrL$ with $\mathrm{side}(L)=2^{-\ell}$}
\STATE 2. Denote by $L_i$ for $i=1,\ldots,4$ the cells obtained by one dyadic subdivision of $L$.
\FOR {$i=1,\ldots,4$}
\STATE 3. Compute $\lambda_{L_i}$ with the approach to Level 3 of Algorithm \ref{alg:piggybackgd}, store local minimal training cost as $C_{L_i}:=\|u_c - u_{L_i}\|^2_{L^2(L_i)}$.
\vspace*{0.3em}
\ENDFOR
\IF{$C_{L_1}+C_{L_2}+C_{L_3}+C_{L_4} < \rho \,C_L$ (cf. \eqref{eq:scragf})}
\vspace*{0.3em}
\STATE 4. Replace $\scrL$ by $\big( \scrL \setminus \{L\} \big) \bigcup_{i=1}^4 \{L_i\}$.
\ENDIF
\ENDFOR
\STATE 5. Set $\omega_{\scrL}$ to be $\omega_{\scrL}=\lambda_L$ or $\omega_{\scrL}=1/\lambda_L$ on each $L \in \scrL$.
\STATE 6. Set $\ell = \ell + 1$.
\ENDWHILE
\STATE 7. Compute $u_{\scrL}$ with $\omega_{\scrL}$ and the numerical approach of Section \ref{sect:commonnumerics} to Level 2 of \((\scrL\!\scrS)_{{TGV-Fid}_{\omega}}\) or \((\scrL\!\scrS)_{{TV-Fid}_\omega}\).
\end{algorithmic}
\caption{Numerical approach to Level 1}
\label{alg:subdivision}
\end{algorithm}

In Algorithm \ref{alg:subdivision}, we summarize our approach to numerically treat Level 1. We remark that in comparison with the original formulations \((\scrL\!\scrS)_{{TV\!}_{\omega}}\) and \((\scrL\!\scrS)_{{TGV\!}_{\omega}}\) as formulated in the introduction, we do not search the entire space of partitions (which would be numerically intractable) and instead work by subdivision as in Example \ref{ex:localsubdiv}. This means that for any given cell $L$, we make a local decision whether to subdivide it or not, based on the training costs arising from it before and after subdividing it in four new cells. When performing this subdivision, the parameter from the original cell is used as initialization for the optimization on the newly created ones. Even though this approach strongly restricts the number of possible partitions considered, it still manages to achieve reasonable performance in practice. On a heuristic level, this indicates that if splitting one dyadic square once to add more detail on the parameter does not lead to better performance, then in most cases it is also not advantageous to consider further finer subdivisions of the same square.

\subsection{Numerical examples with the complete schemes \texorpdfstring{\((\scrL\!\scrS)_{{TV-Fid}_\omega}\)}{TV-Fid-w} and  \texorpdfstring{\((\scrL\!\scrS)_{{TGV-Fid}_{\omega}}\)}{TGV-Fid-w}}\mbox{}\\

In Figures \ref{fig:synthetic}, \ref{fig:lighthouse}, \ref{fig:cameraman}, and \ref{fig:parrot}, we present some illustrative examples resulting from the application of Algorithm \ref{alg:subdivision} with $\ell_{\max} = 4$ to several images, for both $TV$ and $TGV$ regularization and optimizing for one adaptive parameter in the fidelity term, which is also shown along with the partitions overlaid on the noisy input images. In these, we generally see that the adapted fidelity parameter $\lambda$ is higher in areas with finer details. Peak signal to noise ratios and SSIM values for each case are summarized in Table \ref{table:metrics}. In all cases, $TGV$ with adaptive fidelity produces the best results by these metrics, but there are several instances where the gains are very marginal or there are even ties with the corresponding adaptive $TV$ results. Nevertheless, it may be argued that even in these cases the $TGV$ results are more visually appealing due to reduced staircasing.

For the simple example of Figure \ref{fig:synthetic}, some more direct observations can be made. In it, we see that the spatially adaptive results manage to better preserve the fine structures inside the main object, while TGV greatly diminishes staircasing in regions where the original image is nearly linear. Observe that unlike the fine structures, the boundaries of the main object consisting of a sharp discontinuity along an interface with low curvature do not necessarily force further subdivision, as expected for $TV$ or $TGV$ regularization.

The synthetic image used in Figure \ref{fig:synthetic} was created by the authors for this article. The lighthouse and parrot examples in Figures \ref{fig:lighthouse} and \ref{fig:parrot} have been cropped and converted to grayscale from images in the Kodak Lossless Image Suite. The cameraman image of Figure \ref{fig:cameraman} is very widely used, but to our knowledge its origin is not quite clear.
\begin{figure}[htb]
     \begin{center}
     \mbox{} 
     \hfill
         \raisebox{-0.5\height}{\includegraphics[height=0.157\textheight]{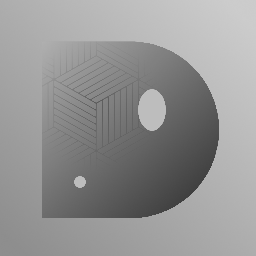}}
         \raisebox{-0.5\height}{\includegraphics[height=0.157\textheight]{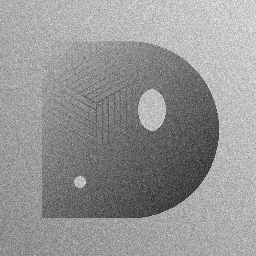}}
     \hfill 
     \mbox{}
     \\\vspace{0.09cm}
     \mbox{} 
     \hfill
         \raisebox{-0.5\height}{\includegraphics[height=0.157\textheight]{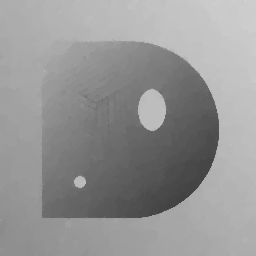}}
         \raisebox{-0.5\height}{\includegraphics[height=0.157\textheight]{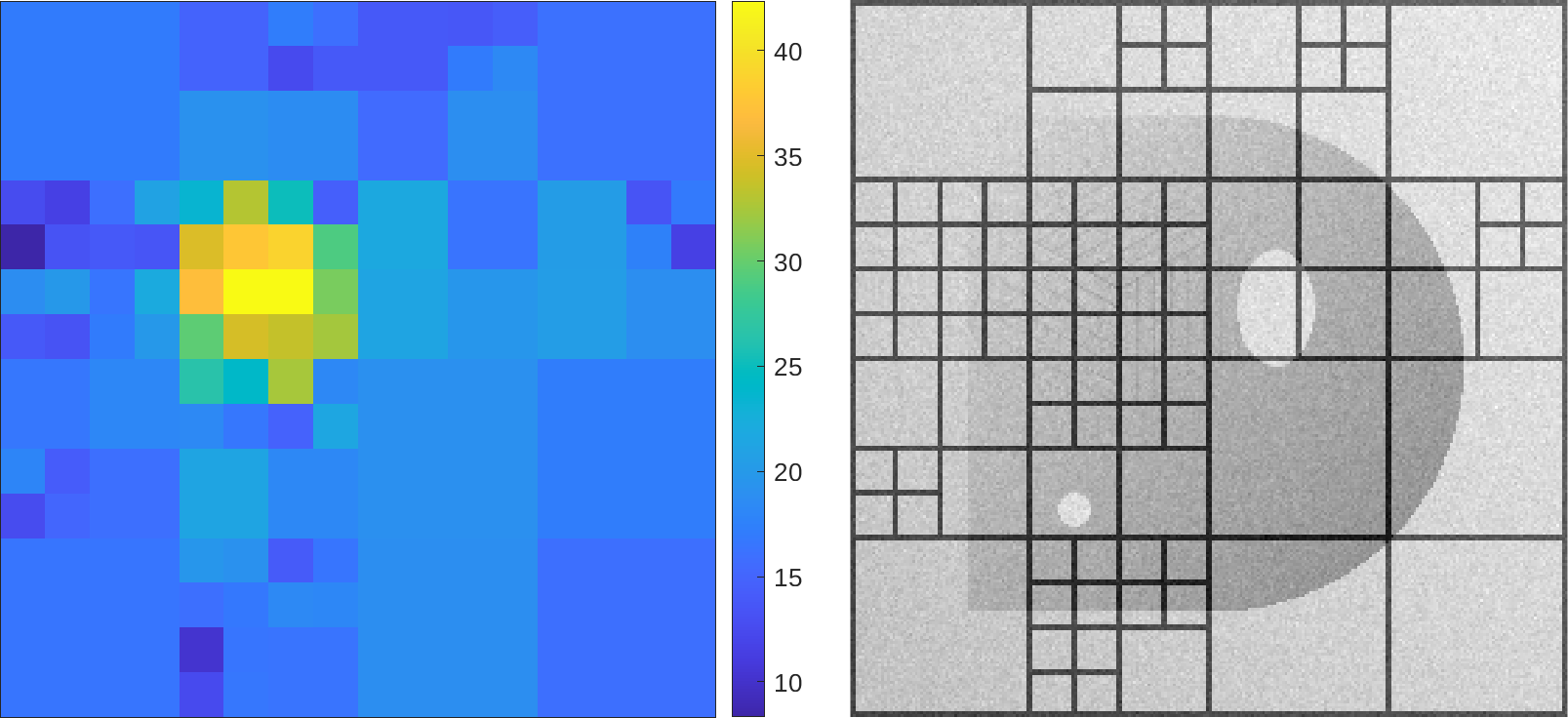}}
         \raisebox{-0.5\height}{\includegraphics[height=0.157\textheight]{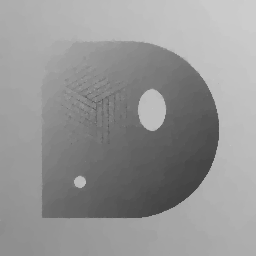}}
     \hfill 
     \mbox{}
     \\\vspace{0.09cm}
         \mbox{}          
     \hfill
         \raisebox{-0.5\height}{\includegraphics[height=0.157\textheight]{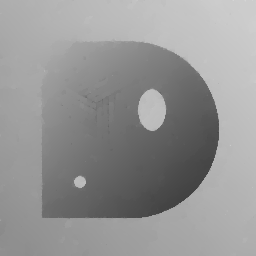}}
         \raisebox{-0.5\height}{\includegraphics[height=0.157\textheight]{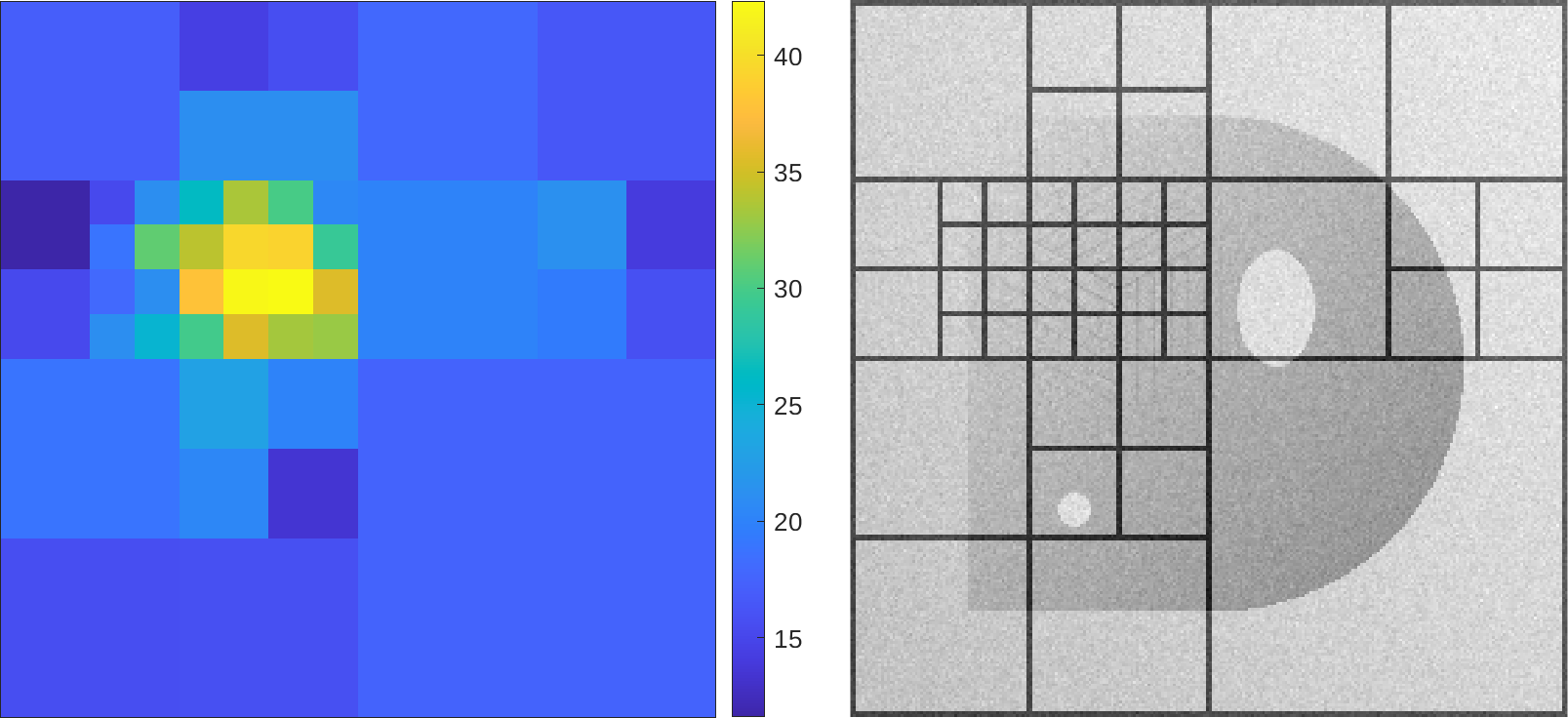}}
         \raisebox{-0.5\height}{\includegraphics[height=0.157\textheight]{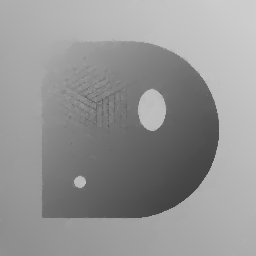}}
     \hfill
     \mbox{}
     \end{center}
     \caption{Synthetic example. Top row: Clean and noisy images $u_c, u_\eta$. Middle row, left to right: TV result with global parameter, partition and spatially-dependent $\lambda$ arising from Algorithm \ref{alg:subdivision}, and corresponding result with weighted fidelity. Bottom row: TGV results, same order as in the middle row and with $\alpha_0=1, \alpha_1 = 10$.}\label{fig:synthetic}
\end{figure}

\begin{figure}[htb]
     \begin{center}
     \mbox{} 
     \hfill
         \raisebox{-0.5\height}{\includegraphics[height=0.157\textheight]{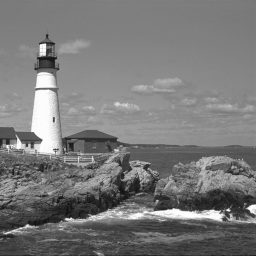}}
         \raisebox{-0.5\height}{\includegraphics[height=0.157\textheight]{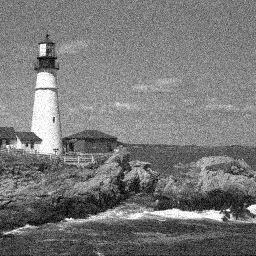}}
     \hfill 
     \mbox{}
     \\\vspace{0.09cm}
     \mbox{} 
     \hfill
         \raisebox{-0.5\height}{\includegraphics[height=0.157\textheight]{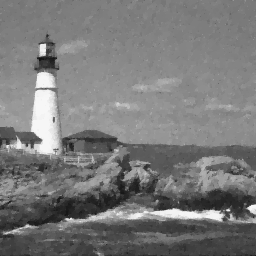}}
         \raisebox{-0.5\height}{\includegraphics[height=0.157\textheight]{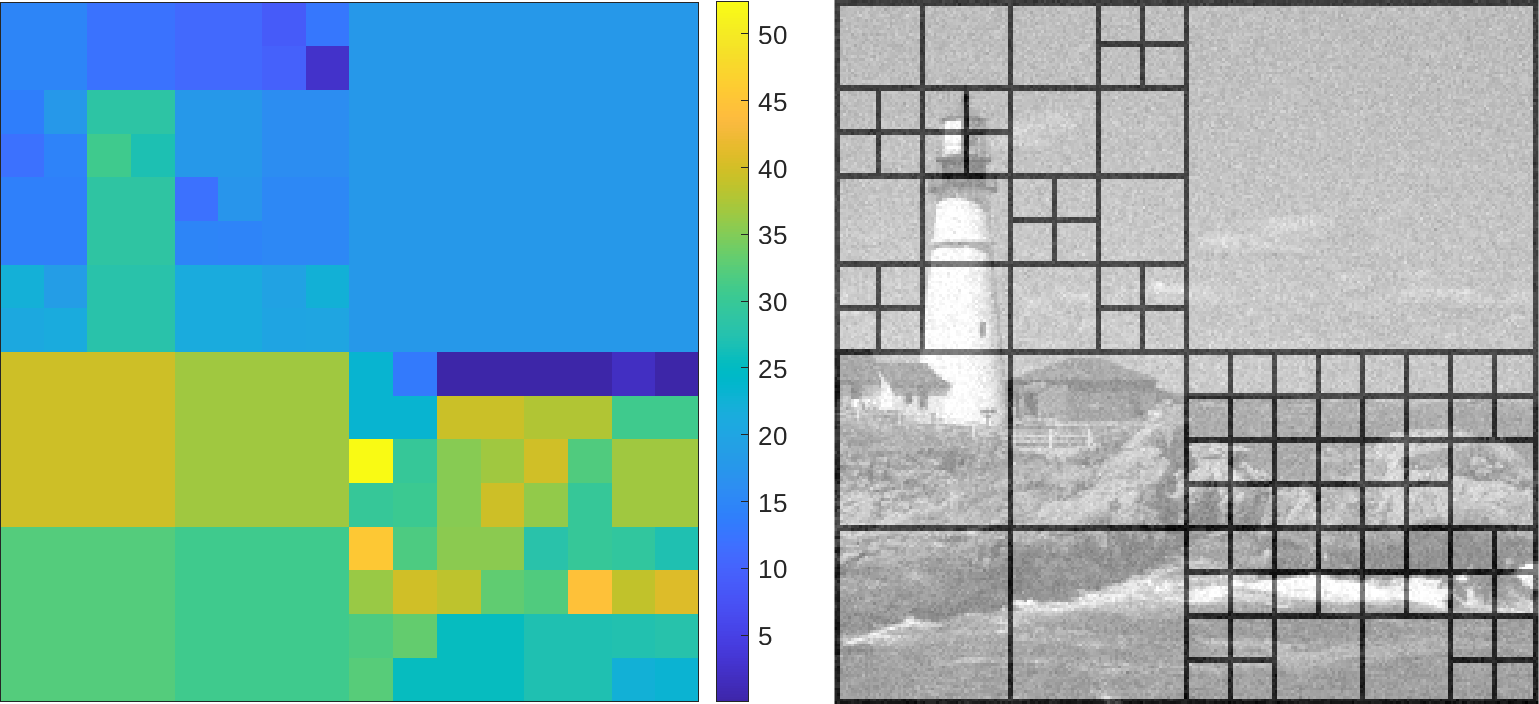}}
         \raisebox{-0.5\height}{\includegraphics[height=0.157\textheight]{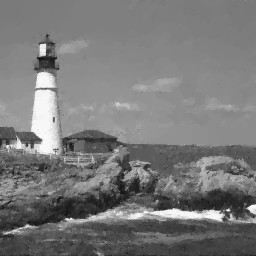}}
     \hfill 
     \mbox{}
     \\\vspace{0.09cm}
         \mbox{}          
     \hfill
         \raisebox{-0.5\height}{\includegraphics[height=0.157\textheight]{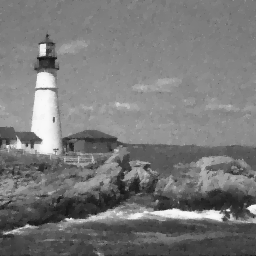}}
         \raisebox{-0.5\height}{\includegraphics[height=0.157\textheight]{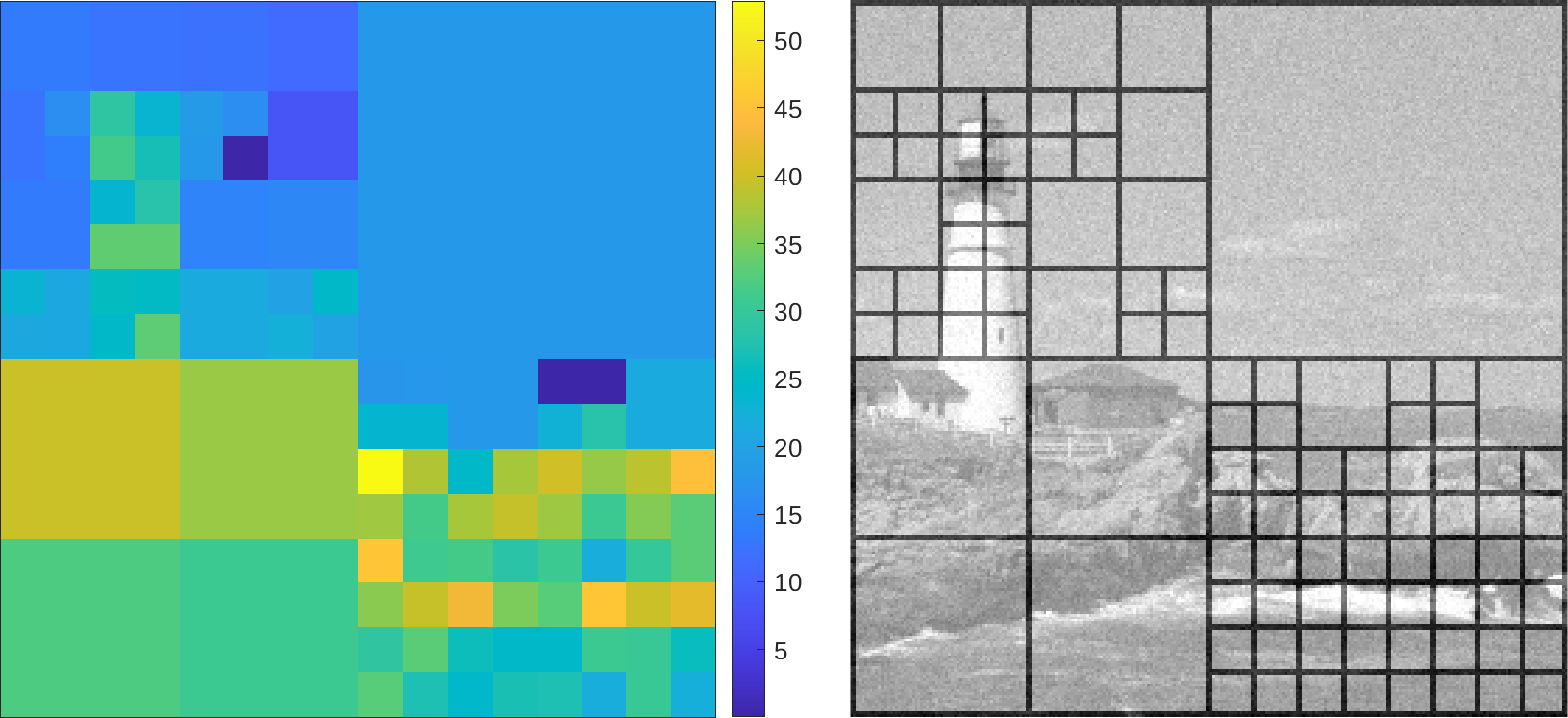}}
         \raisebox{-0.5\height}{\includegraphics[height=0.157\textheight]{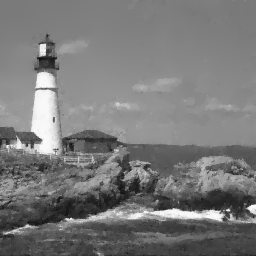}}
     \hfill
     \mbox{}
     \end{center}
     \caption{Lighthouse example. Top row: Clean and noisy images $u_c, u_\eta$. Middle row, left to right: TV result with global parameter, partition and spatially-dependent $\lambda$ arising from Algorithm \ref{alg:subdivision}, and corresponding result with weighted fidelity. Bottom row: TGV results, same order as in the middle row and with $\alpha_0=1, \alpha_1 = 2$.}\label{fig:lighthouse}
\end{figure}

\begin{figure}[htb]
     \begin{center}
     \mbox{} 
     \hfill
         \raisebox{-0.5\height}{\includegraphics[height=0.157\textheight]{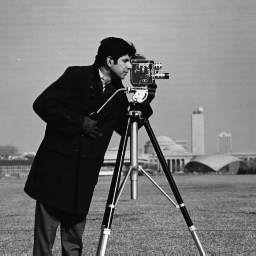}}
         \raisebox{-0.5\height}{\includegraphics[height=0.157\textheight]{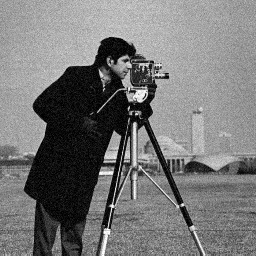}}
     \hfill 
     \mbox{}
     \\\vspace{0.09cm}
     \mbox{} 
     \hfill
         \raisebox{-0.5\height}{\includegraphics[height=0.157\textheight]{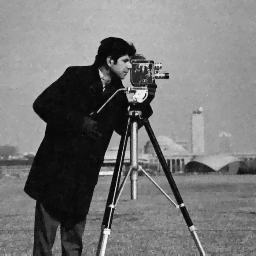}}
         \raisebox{-0.5\height}{\includegraphics[height=0.157\textheight]{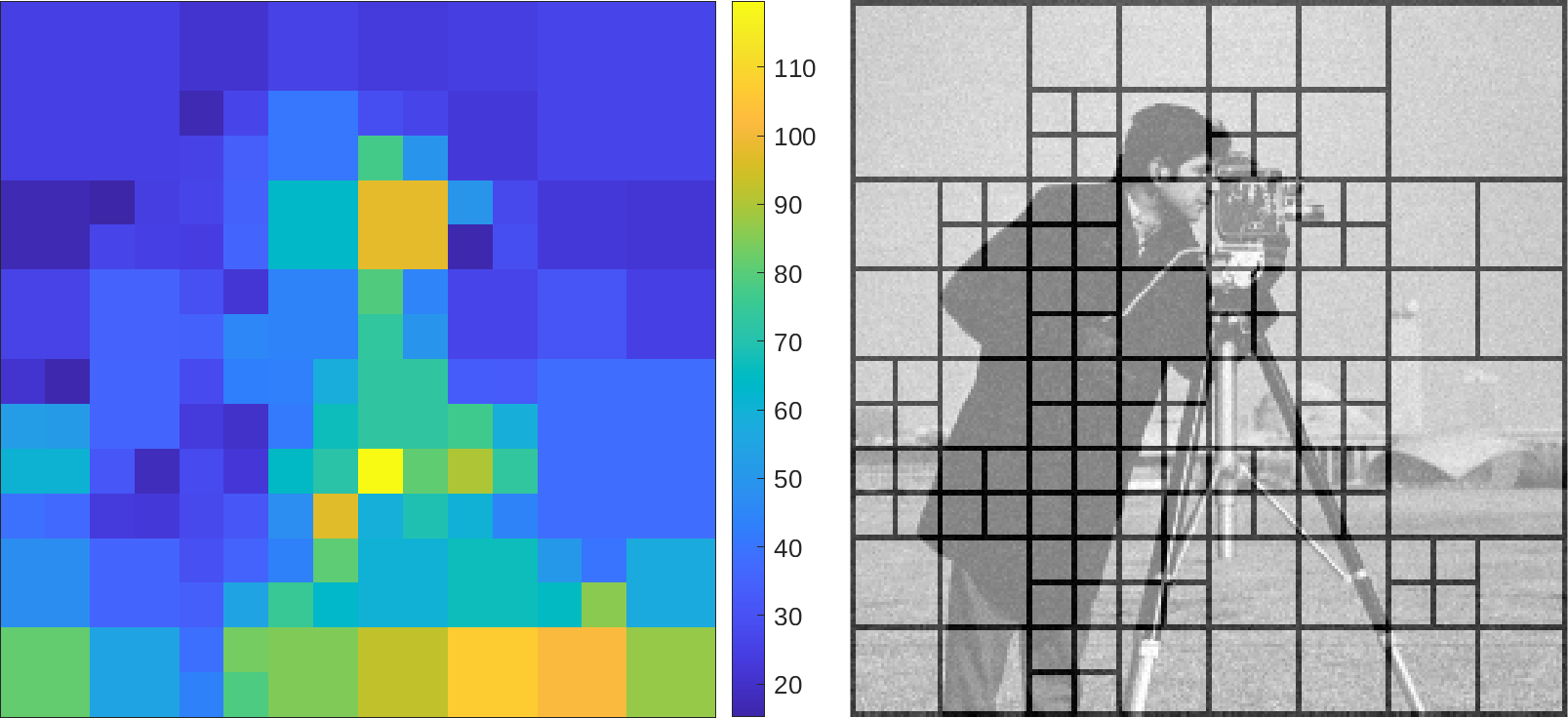}}
         \raisebox{-0.5\height}{\includegraphics[height=0.157\textheight]{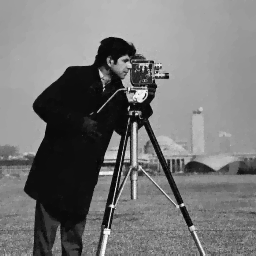}}
     \hfill 
     \mbox{}
     \\\vspace{0.09cm}
         \mbox{}          
     \hfill
         \raisebox{-0.5\height}{\includegraphics[height=0.157\textheight]{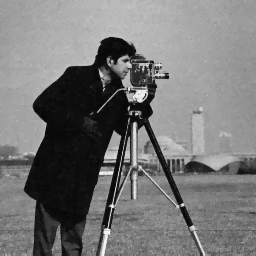}}
         \raisebox{-0.5\height}{\includegraphics[height=0.157\textheight]{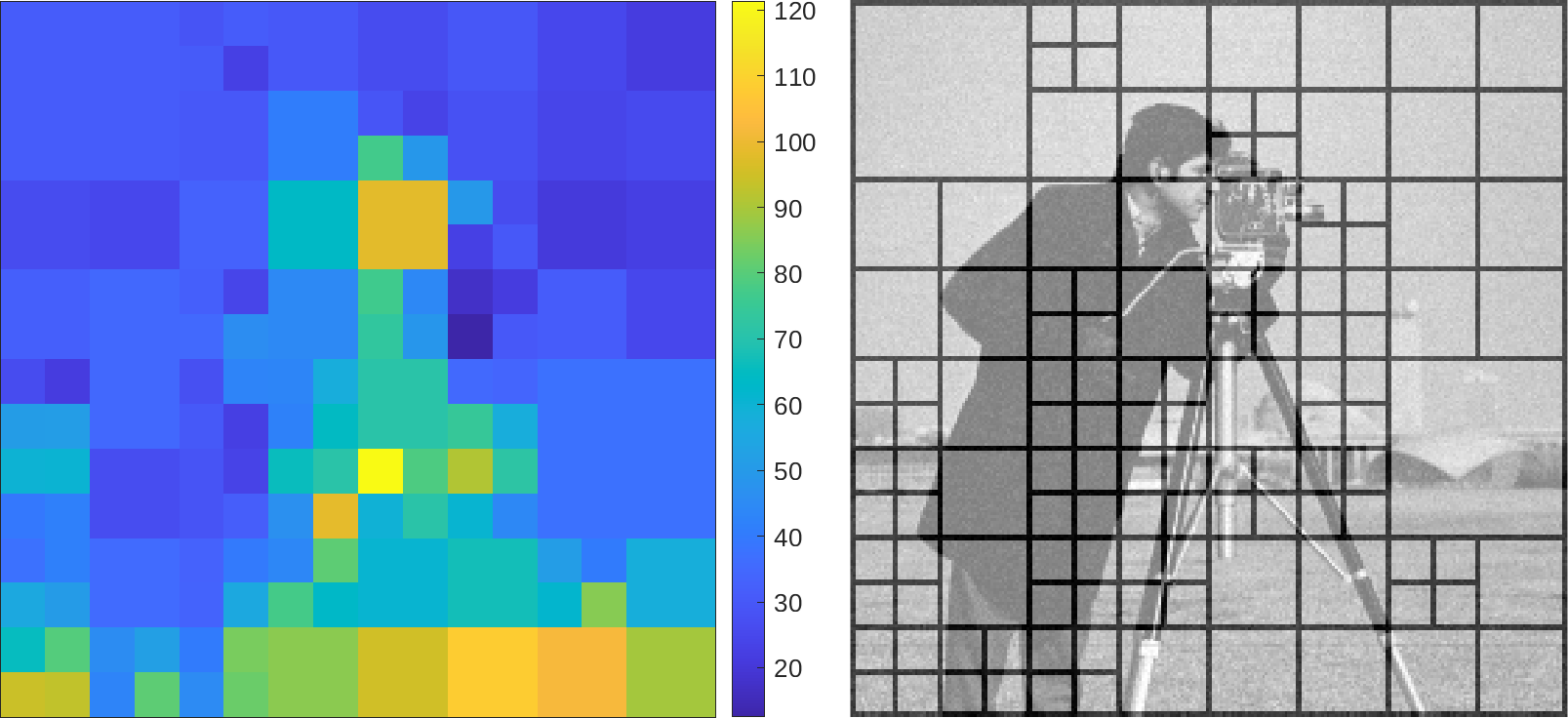}}
         \raisebox{-0.5\height}{\includegraphics[height=0.157\textheight]{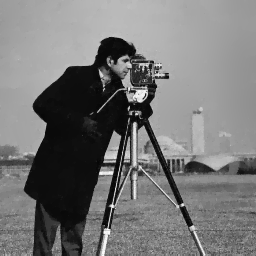}}
     \hfill
     \mbox{}
     \end{center}
     \caption{Cameraman example. Top row: Clean and noisy images $u_c, u_\eta$. Middle row, left to right: TV result with global parameter, partition and spatially-dependent $\lambda$ arising from Algorithm \ref{alg:subdivision}, and corresponding result with weighted fidelity. Bottom row: TGV results, same order as in the middle row and with $\alpha_0=1, \alpha_1 = 10$.}\label{fig:cameraman}
\end{figure}

\begin{figure}[htb]
     \begin{center}
     \mbox{} 
     \hfill
         \raisebox{-0.5\height}{\includegraphics[height=0.157\textheight]{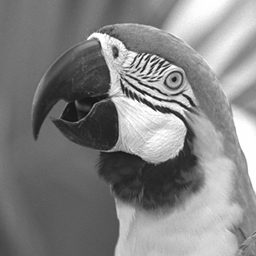}}
         \raisebox{-0.5\height}{\includegraphics[height=0.157\textheight]{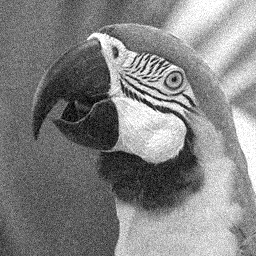}}
     \hfill 
     \mbox{}
     \\\vspace{0.09cm}
     \mbox{} 
     \hfill
         \raisebox{-0.5\height}{\includegraphics[height=0.157\textheight]{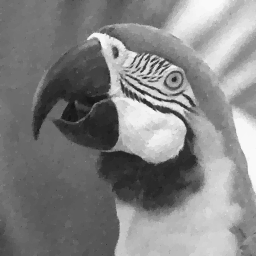}}
         \raisebox{-0.5\height}{\includegraphics[height=0.157\textheight]{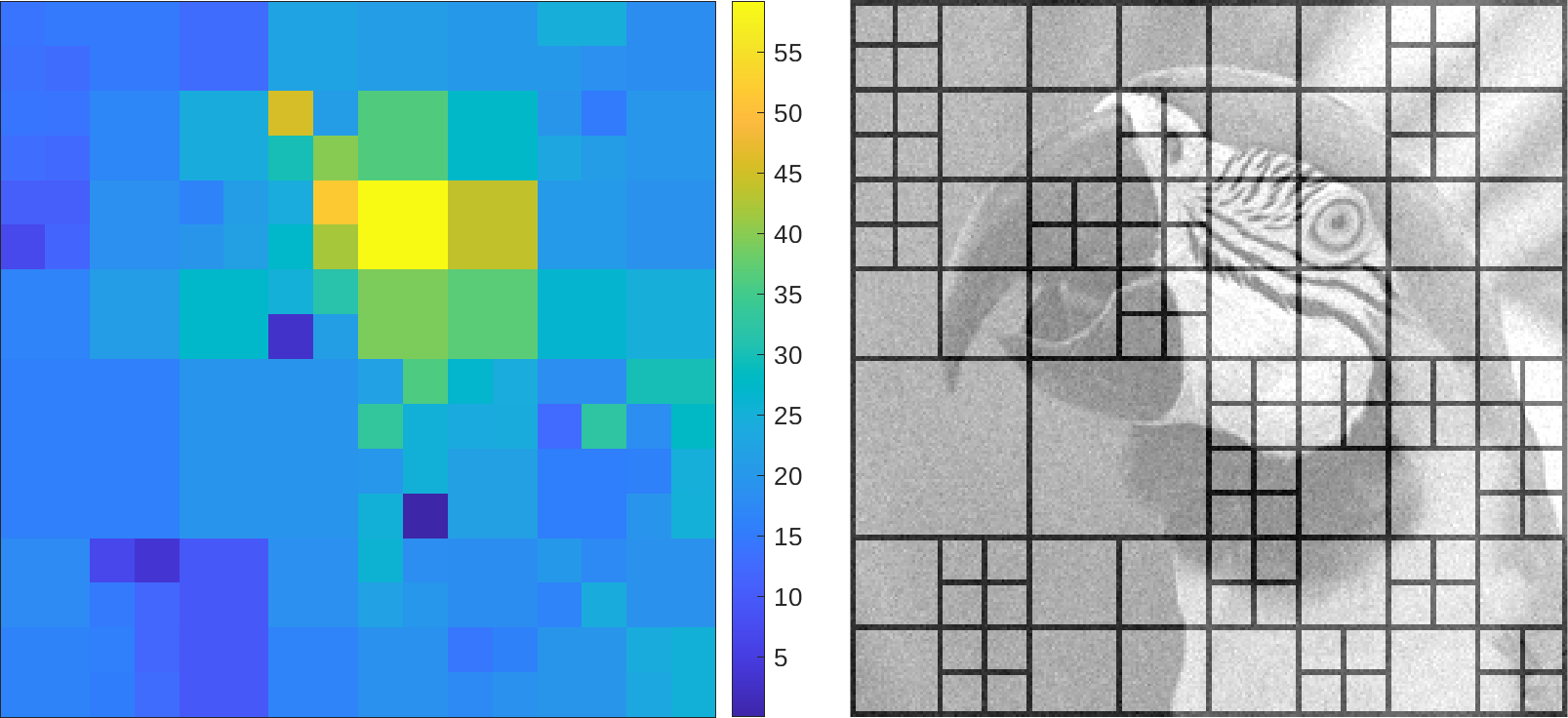}}
         \raisebox{-0.5\height}{\includegraphics[height=0.157\textheight]{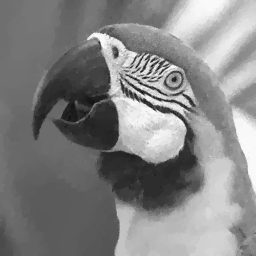}}
     \hfill 
     \mbox{}
     \\\vspace{0.09cm}
         \mbox{}          
     \hfill
         \raisebox{-0.5\height}{\includegraphics[height=0.157\textheight]{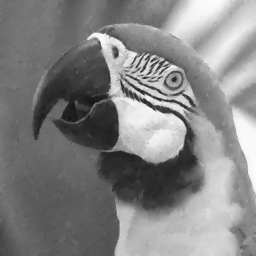}}
         \raisebox{-0.5\height}{\includegraphics[height=0.157\textheight]{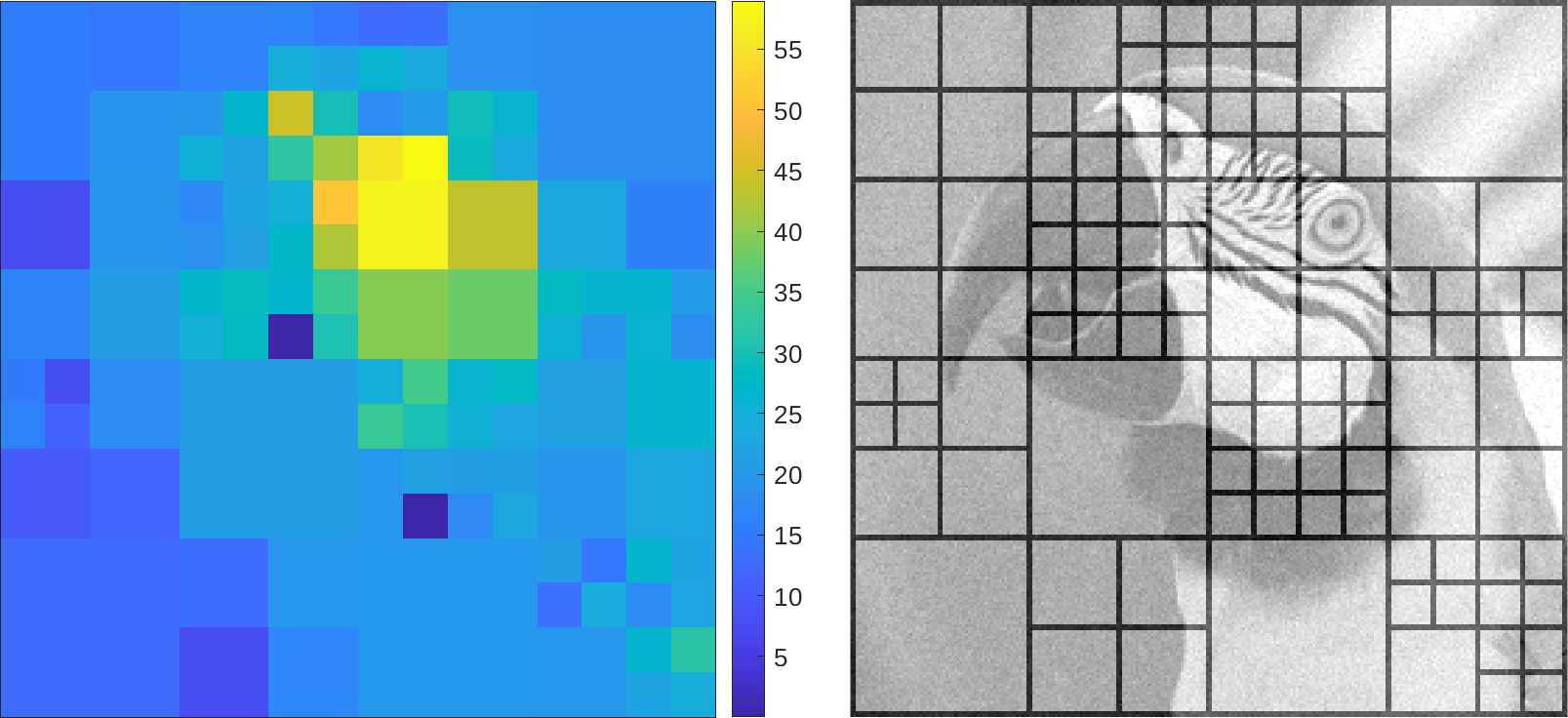}}
         \raisebox{-0.5\height}{\includegraphics[height=0.157\textheight]{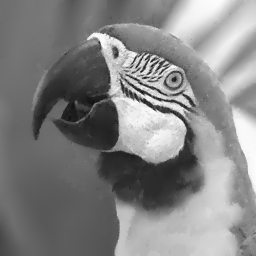}}
     \hfill
     \mbox{}
     \end{center}
     \caption{Parrot example. Top row: Clean and noisy images $u_c, u_\eta$. Middle row, left to right: TV result with global parameter, partition and spatially-dependent $\lambda$ arising from Algorithm \ref{alg:subdivision}, and corresponding result with weighted fidelity. Bottom row: TGV results, same order as in the middle row and with $\alpha_0=1, \alpha_1 = 2$.}\label{fig:parrot}
\end{figure}

\begin{table}
\begin{tabular}{| c | c | c | c | c | c |} 
 \cline{2-6}
 \multicolumn{1}{c|}{}& Noisy & $TV$ global & $TV$ adaptive & $TGV$ global & $TGV$ adaptive \\
 \hline
 Synthetic  & 26.05,\,0.349 & 38.74,\,0.946 & 39.41,\,0.957 & 39.02,\,0.949 & 39.80,\,0.961 \\ 
 \hline
 Lighthouse & 24.64,\,0.496 & 30.42,\,0.853 & 30.82,\,0.886 & 30.44,\,0.855 & 30.90,\,0.890 \\ 
 \hline
 Cameraman  & 28.40,\,0.642 & 32.86,\,0.893 & 33.54,\,0.925 & 32.86,\,0.893 & 33.56,\,0.925 \\ 
 \hline
 Parrot     & 24.67,\,0.447 & 31.86,\,0.880 & 32.37,\,0.898 & 32.10,\,0.889 & 32.72,\,0.909 \\ 
 \hline
\end{tabular} \vspace{1ex}
\caption{PSNR and SSIM values for the examples of Figures \ref{fig:synthetic}, \ref{fig:lighthouse}, \ref{fig:cameraman}, and \ref{fig:parrot}.}\label{table:metrics}
\end{table}

\section*{Acknowledgements}
The work of Elisa Davoli has been supported by the Austrian Science Fund (FWF) through grants 10.55776/F65, 10.55776/V 662, 10.55776/Y1292, and 10.55776/I 4052. Rita Ferreira was partially supported by King Abdullah University of Science and Technology (KAUST) baseline funds and KAUST OSR-CRG2021-4674. Irene Fonseca was partially supported under NSF-DMS1906238 and NSF-DMS2205627. A portion of this work was completed while Jos\'e A. Iglesias was employed at the Johann Radon Institute for Computational and Applied Mathematics (RICAM) of the Austrian Academy of Sciences (\"OAW), during which his work was partially supported by the State of Upper Austria. 

All authors are thankful to Ilaria Perugia for some insightful comments on the topic of mesh refinements, as well as to Carolin Kreisbeck for some interesting discussions on the topic of box constraints. 

\bibliographystyle{plain}
\bibliography{DFFI_LearningSchemes}

\begin{thebibliography}{10}

\bibitem{AcVo94}
R.~Acar and C.~R. Vogel.
\newblock Analysis of bounded variation penalty methods for ill-posed problems.
\newblock {\em Inverse Problems}, 10(6):1217--1229, 1994.

\bibitem{AmFuPa00}
L.~Ambrosio, N.~Fusco, and D.~Pallara.
\newblock {\em Functions of bounded variation and free discontinuity problems}.
\newblock Oxford Mathematical Monographs. The Clarendon Press Oxford University
  Press, New York, 2000.

\bibitem{AtJeNoOr17}
P.~Athavale, R.~L. Jerrard, M.~Novaga, and G.~Orlandi.
\newblock Weighted {TV} minimization and applications to vortex density models.
\newblock {\em J. Convex Anal.}, 24(4):1051--1084, 2017.

\bibitem{AuKo06}
G.~Aubert and P.~Kornprobst.
\newblock {\em {Mathematical problems in image processing. Partial differential
  equations and the calculus of variations. Foreword by Olivier Faugeras. 2nd
  ed.}}
\newblock {New York, NY: Springer}, 2006.

\bibitem{Ba01}
A.~Baldi.
\newblock Weighted {BV} functions.
\newblock {\em Houston J. Math.}, 27(3):683--705, 2001.

\bibitem{BoChPo22}
L.~Bogensperger, A.~Chambolle, and T.~Pock.
\newblock Convergence of a piggyback-style method for the differentiation of
  solutions of standard saddle-point problems.
\newblock {\em SIAM J. Math. Data Sci.}, 4(3):1003--1030, 2022.

\bibitem{Br14}
K.~Bredies.
\newblock Recovering piecewise smooth multichannel images by minimization of
  convex functionals with total generalized variation penalty.
\newblock In A.~Bruhn, T.~Pock, and X.-C. Tai, editors, {\em Efficient
  Algorithms for Global Optimization Methods in Computer Vision}, pages 44--77.
  Springer Berlin Heidelberg, 2014.

\bibitem{BrChHo22}
K.~Bredies, E.~Chenchene, and A.~Hosseini.
\newblock A hybrid proximal generalized conditional gradient method and
  application to total variation parameter learning.
\newblock {\em 2023 European Control Conference (ECC)}, pages 322--327, 2023.

\bibitem{BrHo15}
K.~Bredies and M.~Holler.
\newblock A {TGV}-based framework for variational image decompression, zooming,
  and reconstruction. {P}art {II}: {N}umerics.
\newblock {\em SIAM J. Imaging Sci.}, 8(4):2851--2886, 2015.

\bibitem{BrHo20}
K.~Bredies and M.~Holler.
\newblock Higher-order total variation approaches and generalisations.
\newblock {\em Inverse Problems}, 36(12):123001, 128, 2020.

\bibitem{Br10}
K.~Bredies, K.~Kunisch, and T.~Pock.
\newblock Total generalized variation.
\newblock {\em SIAM J. Imaging Sci.}, 3:492--526, 2010.

\bibitem{BrKuVa13}
K.~Bredies, K.~Kunisch, and T.~Valkonen.
\newblock Properties of {$L^1$}-{${\rm TGV}^2$}: the one-dimensional case.
\newblock {\em J. Math. Anal. Appl.}, 398(1):438--454, 2013.

\bibitem{Bu15}
M.~Burger, K.~Papafitsoros, E.~Papoutsellis, and C.-B. Sch{\"o}nlieb.
\newblock Infimal convolution regularisation functionals of {$BV$} and
  {${L}^p$} spaces. {P}art {I}: the finite {$p$} case.
\newblock {\em J. Math. Imaging Vision}, 55:343--369, 2016.

\bibitem{CaEtAl17}
L.~Calatroni, C.~Cao, J.~C. De~los Reyes, C.-B. Sch{\"o}nlieb, and T.~Valkonen.
\newblock Bilevel approaches for learning of variational imaging models.
\newblock In {\em Variational Methods: In Imaging and Geometric Control}, pages
  252--290. De Gruyter, 2017.

\bibitem{Ca08}
C.~S. Camfield.
\newblock {\em Comparison of {BV} norms in weighted {E}uclidean spaces and
  metric measure spaces}.
\newblock PhD thesis, University of Cincinnati, 2008.

\bibitem{Ca10}
C.~S. Camfield.
\newblock Comparison of {BV} norms in weighted {E}uclidean spaces.
\newblock {\em J. Anal.}, 18:83--97, 2010.

\bibitem{Ch04}
A.~Chambolle.
\newblock An algorithm for total variation minimization and applications.
\newblock {\em J. Math. Imaging Vision}, 20(1-2):89--97, 2004.
\newblock Special issue on mathematics and image analysis.

\bibitem{ChLi97}
A.~Chambolle and P.-L. Lions.
\newblock Image recovery via total variation minimization and related problems.
\newblock {\em Numer. Math.}, 76(2):167--188, 1997.

\bibitem{ChPo11}
A.~Chambolle and T.~Pock.
\newblock A first-order primal-dual algorithm for convex problems with
  applications to imaging.
\newblock {\em J. Math. Imaging Vision}, 40(1):120--145, 2011.

\bibitem{ChPo21}
A.~Chambolle and T.~Pock.
\newblock Learning consistent discretizations of the total variation.
\newblock {\em SIAM J. Imaging Sci.}, 14(2):778--813, 2021.

\bibitem{Ch00}
T.~Chan, A.~Marquina, and P.~Mulet.
\newblock High-order total variation-based image restoration.
\newblock {\em SIAM J. Sci. Comput.}, 22:503--516, 2000.

\bibitem{ChKaSh01}
T.~F. Chan, S.~H. Kang, and J.~Shen.
\newblock {Total variation denoising and enhancement of color images based on
  the CB and HSV color models}.
\newblock {\em Journal of Visual Communication and Image Representation},
  {12}({4}):{422--435}, {2001}.

\bibitem{ChPoRaBi13}
Y.~Chen, T.~Pock, R.~Ranftl, and H.~Bischof.
\newblock Revisiting loss-specific training of filter-based {MRF}s for image
  restoration.
\newblock In J.~Weickert, M.~Hein, and B.~Schiele, editors, {\em Pattern
  Recognition}, pages 271--281. Springer Berlin Heidelberg, 2013.

\bibitem{ChRaPo14}
Y.~Chen, R.~Ranftl, and T.~Pock.
\newblock Insights into analysis operator learning: from patch-based sparse
  models to higher order {MRF}s.
\newblock {\em IEEE Trans. Image Process.}, 23(3):1060--1072, 2014.

\bibitem{ChDeSc16}
C.~V. Chung, J.~C. De~los Reyes, and C.~B. Sch\"{o}nlieb.
\newblock Learning optimal spatially-dependent regularization parameters in
  total variation image denoising.
\newblock Technical report, ModeMat, 2016.

\bibitem{CrFe22}
C.~Crockett and J.~A. Fessler.
\newblock Bilevel methods for image reconstruction.
\newblock {\em Found. Trends Signal Process.}, 15(2-3):121--289, 2022.

\bibitem{Da93}
G.~Dal~Maso.
\newblock {\em An introduction to {$\Gamma$}-convergence}.
\newblock Progress in Nonlinear Differential Equations and their Applications,
  8. Birkh\"auser Boston Inc., Boston, MA, 1993.

\bibitem{Da09}
G.~Dal~Maso, I.~Fonseca, G.~Leoni, and M.~Morini.
\newblock A higher order model for image restoration: the one-dimensional case.
\newblock {\em SIAM J. Math. Anal.}, 40:2351--2391, 2009.

\bibitem{DaFoLi23}
E.~Davoli, I.~Fonseca, and P.~Liu.
\newblock Adaptive image processing: first order {PDE} constraint regularizers
  and a bilevel training scheme.
\newblock {\em J. Nonlinear Sci.}, 33(3):Paper No. 41, 38, 2023.

\bibitem{DaLi18}
E.~Davoli and P.~Liu.
\newblock One dimensional fractional order {TGV}: {G}amma-convergence and
  bilevel training scheme.
\newblock {\em Commun. Math. Sci.}, 16(1):213--237, 2018.

\bibitem{DFKH22}
E.~Davoli, Ferreira R., C.~Kreisbeck, and H.~Sch\"onberger.
\newblock Structural changes in nonlocal denoising models arising through
  bi-level parameter learning.
\newblock {\em Applied Mathematics and Optimization}, 88:Art. 9, 2023.

\bibitem{De23}
J.~C. De~los Reyes.
\newblock Bilevel imaging learning problems as mathematical programs with
  complementarity constraints: reformulation and theory.
\newblock {\em SIAM J. Imaging Sci.}, 16(3):1655--1686, 2023.

\bibitem{DeSc13}
J.~C. De~los Reyes and C.-B. Sch\"{o}nlieb.
\newblock Image denoising: learning the noise model via nonsmooth
  {PDE}-constrained optimization.
\newblock {\em Inverse Probl. Imaging}, 7(4):1183--1214, 2013.

\bibitem{DeScVa16}
J.~C. De~Los~Reyes, C.-B. Sch\"{o}nlieb, and T.~Valkonen.
\newblock The structure of optimal parameters for image restoration problems.
\newblock {\em J. Math. Anal. Appl.}, 434(1):464--500, 2016.

\bibitem{DeScVa17}
J.~C. De~los Reyes, C.-B. Sch\"{o}nlieb, and T.~Valkonen.
\newblock Bilevel parameter learning for higher-order total variation
  regularisation models.
\newblock {\em J. Math. Imaging Vision}, 57(1):1--25, 2017.

\bibitem{DeVi22}
J.~C. De~los Reyes and D.~Villac\'{\i}s.
\newblock Optimality conditions for bilevel imaging learning problems with
  total variation regularization.
\newblock {\em SIAM J. Imaging Sci.}, 15(4):1646--1689, 2022.

\bibitem{DeVi23}
J.~C. De~los Reyes and D.~Villac\'{\i}s.
\newblock Bilevel optimization methods in imaging.
\newblock In {\em Handbook of mathematical models and algorithms in computer
  vision and imaging---mathematical imaging and vision}, pages 909--941.
  Springer, Cham, 2023.

\bibitem{DiEtAl07}
O.~Dietrich, J.~G. Raya, S.~B. Reeder, M.~F. Reiser, and S.~O. Schoenberg.
\newblock Measurement of signal-to-noise ratios in {MR} images: Influence of
  multichannel coils, parallel imaging, and reconstruction filters.
\newblock {\em J. Magn. Reson. Imaging}, 26(2):375--385, 2007.

\bibitem{Do12}
J.~Domke.
\newblock Generic methods for optimization-based modeling.
\newblock In N.~D. Lawrence and M.~Girolami, editors, {\em Proceedings of the
  Fifteenth International Conference on Artificial Intelligence and
  Statistics}, volume~22 of {\em Proceedings of Machine Learning Research},
  pages 318--326, La Palma, Canary Islands, 21--23 Apr 2012. PMLR.

\bibitem{FoLi17}
I.~Fonseca and P.~Liu.
\newblock The weighted {A}mbrosio-{T}ortorelli approximation scheme.
\newblock {\em SIAM J. Math. Anal.}, 49(6):4491--4520, 2017.

\bibitem{HiHoPa18}
M.~Hinterm\"{u}ller, M.~Holler, and K.~Papafitsoros.
\newblock A function space framework for structural total variation
  regularization with applications in inverse problems.
\newblock {\em Inverse Problems}, 34(6):064002, 39, 2018.

\bibitem{HiPa19}
M.~Hinterm\"uller and K.~Papafitsoros.
\newblock Generating structured non-smooth priors and associated primal-dual
  methods.
\newblock In R.~Kimmel and X.-C. Tai, editors, {\em Processing, Analyzing and
  Learning of Images, Shapes, and Forms: Part 2}, volume~20 of {\em Handbook of
  Numerical Analysis}, pages 437--502. Elsevier, 2019.

\bibitem{HiPaRaSu22}
M.~Hinterm\"{u}ller, K.~Papafitsoros, C.~N. Rautenberg, and H.~Sun.
\newblock Dualization and automatic distributed parameter selection of total
  generalized variation via bilevel optimization.
\newblock {\em Numer. Funct. Anal. Optim.}, 43(8):887--932, 2022.

\bibitem{HiPaRa17}
M.~Hinterm\"{u}ller, K.~Papafitsoros, and C.N. Rautenberg.
\newblock Analytical aspects of spatially adapted total variation
  regularisation.
\newblock {\em J. Math. Anal. Appl.}, 454(2):891--935, 2017.

\bibitem{HiRaWuLan17}
M.~Hinterm\"{u}ller, C.~N. Rautenberg, T.~Wu, and A.~Langer.
\newblock Optimal selection of the regularization function in a weighted total
  variation model. {P}art {II}: {A}lgorithm, its analysis and numerical tests.
\newblock {\em J. Math. Imaging Vision}, 59(3):515--533, 2017.

\bibitem{IgWa22}
J.~A. Iglesias and D.~Walter.
\newblock Extremal points of total generalized variation balls in 1{D}:
  characterization and applications.
\newblock {\em J. Convex Anal.}, 29(4):1251--1290, 2022.

\bibitem{Ja12}
K.~Jalalzai.
\newblock {\em Regularization of inverse problems in image processing}.
\newblock PhD thesis, Ecole Polytechnique X, 2012.

\bibitem{KoEtAl23}
A.~Kofler, F.~Altekr\"{u}ger, F.~Antarou~Ba, C.~Kolbitsch, E.~Papoutsellis,
  D.~Schote, C.~Sirotenko, F.~F. Zimmermann, and K.~Papafitsoros.
\newblock Learning regularization parameter-maps for variational image
  reconstruction using deep neural networks and algorithm unrolling.
\newblock {\em SIAM J. Imaging Sci.}, 16(4):2202--2246, 2023.

\bibitem{KuPo12}
K.~Kunisch and T.~Pock.
\newblock A bilevel optimization approach for parameter learning in variational
  models.
\newblock {\em SIAM J. Imaging Sci.}, 2(6):938--983, 2012.

\bibitem{Li17}
P.~Liu.
\newblock {\em Variational and {PDE} Methods for Image Processing}.
\newblock PhD thesis, Carnegie-Mellon University, 2017.

\bibitem{Li19}
P.~Liu.
\newblock Adaptive image processing: a bilevel structure learning approach for
  mixed-order total variation regularizers.
\newblock Preprint arXiv:1902.01122 [math.OC], 2019.

\bibitem{LiLu18}
P.~Liu and X.~Y. Lu.
\newblock Real order (an)-isotropic total variation in image processing -
  {P}art {I}: {A}nalytical analysis and functional properties.
\newblock Preprint arXiv:1805.06761 [math.AP], 2018.

\bibitem{LiLu19}
P.~Liu and X.~Y. Lu.
\newblock Real order (an)-isotropic total variation in image processing -
  {P}art {II}: {L}earning of optimal structures.
\newblock Preprint arXiv:1903.08513 [math.OC], 2019.

\bibitem{LiSc19}
P.~Liu and C.-B. Sch{\"o}nlieb.
\newblock Learning optimal orders of the underlying euclidean norm in total
  variation image denoising.
\newblock Preprint arXiv:1903.11953 [math.AP], 2019.

\bibitem{PaPaRaVi22}
V.~Pagliari, K.~Papafitsoros, B.~Rai{\c{t}}\u{a}, and A.~Vikelis.
\newblock Bilevel {T}raining {S}chemes in {I}maging for {T}otal
  {V}ariation--{T}ype {F}unctionals with {C}onvex {I}ntegrands.
\newblock {\em SIAM J. Imaging Sci.}, 15(4):1690--1728, 2022.

\bibitem{Pa13}
K.~Papafitsoros and K.~Bredies.
\newblock A study of the one dimensional total generalised variation
  regularisation problem.
\newblock {\em Inverse Probl. Imaging}, 9:511--550, 2015.

\bibitem{PaVa15}
K.~Papafitsoros and T.~Valkonen.
\newblock Asymptotic behaviour of total generalised variation.
\newblock In {\em Scale space and variational methods in computer vision},
  volume 9087 of {\em Lecture Notes in Comput. Sci.}, pages 720--714. Springer,
  Cham, 2015.

\bibitem{Po12}
T.~Pock.
\newblock On parameter learning in variational models.
\newblock In {\em International Symposium on Mathematical Programming}, 2012.

\bibitem{PoSc15}
C.~P\"{o}schl and O.~Scherzer.
\newblock Exact solutions of one-dimensional total generalized variation.
\newblock {\em Commun. Math. Sci.}, 13(1):171--202, 2015.

\bibitem{RuOsFa92}
L.~I. Rudin, S.~Osher, and E.~Fatemi.
\newblock {Nonlinear total variation based noise removal algorithms.}
\newblock {\em Physica D}, 60(1-4):259--268, 1992.

\bibitem{TaLiAdFr07}
M.~F. Tappen, C.~Liu, E.~Adelson, and W.~T. Freeman.
\newblock Learning gaussian conditional random fields for low-level vision.
\newblock In {\em 2007 IEEE Conference on Computer Vision and Pattern
  Recognition}, pages 1--8, 2007.

\bibitem{Te83}
R.~Temam.
\newblock {\em Probl\`emes math\'{e}matiques en plasticit\'{e}}, volume~12 of
  {\em M\'{e}thodes Math\'{e}matiques de l'Informatique [Mathematical Methods
  of Information Science]}.
\newblock Gauthier-Villars, Montrouge, 1983.

\end{thebibliography}

\end{document}